\documentclass[12pt,twoside]{article}

\usepackage{mathtools}
\usepackage[colorlinks=true,linkcolor=blue,citecolor=blue,urlcolor=blue]{hyperref}
\usepackage{amsmath}
\usepackage{amsthm}
\usepackage[noabbrev]{cleveref}
\usepackage{amsfonts}
\usepackage{amssymb}
\usepackage{url}

\usepackage{graphicx}
\usepackage{float}
\usepackage[makeroom]{cancel}
\usepackage{framed}
\usepackage{scalerel}
\usepackage{color}
\usepackage{yfonts}
\usepackage{amsmath}
\usepackage{amssymb}
\usepackage{amsthm}
\usepackage[makeroom]{cancel}
\usepackage{mathtools}
\usepackage{color}
\usepackage{tcolorbox}
\usepackage[nottoc]{tocbibind}
\usepackage{flexisym}
\usepackage{yfonts}
\usepackage{caption}

\newcommand{\beaa}{\begin{eqnarray*}}
\newcommand{\eeaa}{\end{eqnarray*}}
\newcommand{\bea}{\begin{eqnarray}}
\newcommand{\eea}{\end{eqnarray}}
\newcommand\mi{\mathrm{i}\hskip 1pt}	

\newcommand{\Pro}{\noindent\textit{Proof.}\ \ }

\def\dd{\hskip 1pt{\rm{d}}}

\def\Var{\mathop{\rm Var\,}\nolimits}
\def\Cov{{\mathop{\rm Cov}}}

\def\bC{\mathbb{C}}
\def\bN{\mathbb{N}}
\def\bR{\mathbb{R}}
\DeclareMathOperator{\tr}{tr}

\DeclareMathOperator{\etr}{etr}

\DeclareMathOperator{\vect}{vech}

\numberwithin{equation}{section}
\theoremstyle{plain}
\newtheorem{theorem}{Theorem}[section]
\newtheorem{corollary}[theorem]{Corollary}
\newtheorem{lemma}[theorem]{Lemma}
\newtheorem{proposition}[theorem]{Proposition}
\newtheorem{remark}[theorem]{Remark}

\newtheorem{example}[theorem]{Example}
\newtheorem{definition}[theorem]{Definition}
\newtheorem{assumption}[theorem]{Assumptions}
\newcommand{\vertiii}[1]{{\left\vert\kern-0.25ex\left\vert\kern-0.25ex\left\vert #1 
    \right\vert\kern-0.25ex\right\vert\kern-0.25ex\right\vert}}
\theoremstyle{definition}

\begin{document}

\title{\bf \large
Integral Transform Methods in Goodness-of-Fit Testing, II: The Wishart Distributions
}

\author{
{Elena Hadjicosta}\thanks{Department of Statistics, Pennsylvania State University, University Park, PA 16802, U.S.A.  E-mail address: \href{mailto:exh963@psu.edu}{exh963@psu.edu}.} 
\ {and Donald Richards}\thanks{Department of Statistics, Pennsylvania State University, University Park, PA 16802, U.S.A.  E-mail address: \href{mailto:richards@stat.psu.edu}{richards@stat.psu.edu}.
\endgraf
\endgraf
}
}

\maketitle

\medskip

\date

\medskip

\begin{abstract}
We initiate the study of goodness-of-fit testing when the data consist of positive definite matrices.  Motivated by the recent appearance of the cone of positive definite matrices in numerous areas of applied research, including diffusion tensor imaging, models of the volatility of financial time series, wireless communication systems, and the analysis of polarimetric radar images, we apply the method of Hankel transforms of matrix argument to develop goodness-of-fit tests for Wishart distributions with given shape parameter and unknown scale matrix. We obtain the limiting null distribution of the test statistic and the corresponding covariance operator. We show that the eigenvalues of the operator satisfy an interlacing property, and we apply our test to some financial data. 
Moreover, we establish the consistency of the test against a large class of alternative distributions and we derive the asymptotic distribution of the test statistic under a sequence of contiguous alternatives. We establish the Bahadur and Pitman efficiency properties of the test statistic and we show the validity of a modified Wieand condition. 
\end{abstract}

\newpage
\tableofcontents

\normalsize

\parindent=25pt

\section{Introduction}
\label{introduction}
\setcounter{equation}{0}

In this paper, we develop goodness-of-fit tests for the Wishart distributions, extending the results of \cite{BT2010,ref27} for the exponential distributions and \cite{hadjicosta19,hadjicostarichards} for the gamma distributions.  In recent years, the cone of positive definite matrices has arisen in numerous applications, e.g., diffusion tensor imaging, financial time series, wireless communication systems, and polarimetric radar images; it is these applications that motivate our study of goodness-of-fit tests for probability distributions on the cone.

Positive definite random matrix data have appeared in medical research, specifically in diffusion tensor imaging (DTI), cf. \cite{dryden,jian,jian2,kimrichards,armin, armin2, armin3, armin4}.  DTI is a magnetic resonance imaging method that has attracted much interest in the study of brain diseases.  DTI is based on the observation that water molecules {\it in vivo} are always in motion; by modelling the diffusion of the water molecules at any location by a three-dimensional Brownian motion, the resulting diffusion tensor image is represented by the $3 \times 3$ positive definite matrix of the local diffusion process at the given location. 

Although DTI is non-invasive, it enables the study of deep brain white-matter fibers. Thus, DTI has been used to study epileptic seizures, Alzheimer's disease, traumatic brain injuries, aging, white-matter abnormalities, developmental disorders, and psychiatric conditions \cite{neumann, rosenbloom, pomara, matthews}. DTI has also been used to study the pathology of organ and tissue types such as the breast, cardiac, kidney, lingual, skeletal muscles, and spinal cord \cite{damon}.  In numerous articles, the Wishart distribution with known degrees-of-freedom and unknown scale matrix has been used to model DTI data \cite{dryden,jian,jian2}.  

The Wishart distributions with known degrees-of-freedom also arise in stochastic volatility models \cite{asai,gourieroux,kubloomfield}.  In this area, the problem is to estimate the covariance matrix of the joint capital returns on several financial assets, with the goal of predicting future returns, devising portfolio allocations, pricing options, and estimating risk.  

The {\it complex} Wishart distributions with known degrees-of-freedom arise in the spectral analysis of multivariate Gaussian time series \cite{goodman}, wireless communications \cite{siriteanu,siriteanu2,tulino} and the analysis of polarimetric synthetic aperture radar \cite{multilook1, multilook2}.  These applications are widespread, for the spectral analysis of such time series arises in signal processing, econometrics, meteorology, and polarimetric radar has become an important remote sensing device due to its heightened ability to distinguish between distinct scattering sources.  The results to follow can be extended, after making obvious necessary changes, to the complex Wishart distributions \cite[p.~488]{james} and even to the Wishart distributions on general symmetric cones \cite{farautkoranyi}.  

The technical details required to develop goodness-of-fit tests for positive definite matrix data 
are extensive.  Naturally, we will need mathematical analysis on the cone of positive definite matrices \cite{maass}, the Bessel and Laguerre polynomials of matrix argument and their zonal polynomial expansions \cite{grossrichards,herz,james,muirhead}, and Hankel transforms of matrix argument \cite{herz}.  Further complications arising from the non-commutative nature of matrix multiplication leads us to impose on the distribution of the sample data an orthogonal invariance condition.  In addition, the Frobenius, spectral, and operator norms arise in the matrix case, and numerous inequalities between them will be needed.  There is also the surprising appearance of Schur's lemma, a result well-known in linear algebra but which appears only rarely in statistical inference.  

We now describe the results in the paper.  Throughout, we will follow as templates the presentations in \cite{BT2010,hadjicostarichards,ref27}.  In Section \ref{propertieshankel_matrix} we provide some properties of the Wishart distributions, and related results for the Bessel functions, Hankel transforms, confluent hypergeometric functions, and generalized Laguerre polynomials, all of matrix argument.  Further, we present uniqueness theorems for the Hankel transform of matrix argument, a Hankel inversion formula, and some limit theorems.  After providing results on a generalized hypergeometric function of two matrix arguments, we define the orthogonally invariant Hankel transform and present some of its properties. 

In Section \ref{goodnessoffittests_wishart}, we propose an integral-type test statistic $\boldsymbol{T}_n^2$ for goodness-of-testing for the Wishart distributions.  Generalizing the one-dimensional cases \cite{BT2010,hadjicostarichards,ref27}, the statistic $\boldsymbol{T}_n^2$ is a squared integral, \eqref{statisticc_wishart}, involving the empirical orthogonally invariant Hankel transform.  We obtain the asymptotic distribution of $\boldsymbol{T}_n^2$ under the null hypothesis, proving that $\boldsymbol{T}_n^2$ converges in distribution to a weighted sum of independent and identically distributed random variables, each having a chi-square distribution with one degree-of-freedom. The coefficients of the weighted sum are the positive eigenvalues of the covariance operator corresponding to a certain zero-mean Gaussian random field.  The determination of the multiplicity of the eigenvalues remains an open problem, but we show that the eigenvalues satisfy an interlacing property and we show the usefulness of the interlacing property in an application of the test statistic to financial data.  Also, we establish the consistency of the test against a large class of alternative distributions.

In Section \ref{contiguousmatrix}, we derive the asymptotic distribution of $\boldsymbol{T}_n^2$ under certain sequences of contiguous alternatives to the null hypothesis.  Specifically, we consider Wishart alternatives with varying shape or scale parameters, some classes of contaminated Wishart models in which the contamination distribution is a generalized inverted Gaussian. 

Finally, in Section \ref{efficiency_matrix}, we establish the Bahadur and Pitman efficiency properties of the statistic $\boldsymbol{T}_n^2$.  We investigate the approximate Bahadur slope of $\boldsymbol{T}_n^2$ under local alternatives and we show the validity of a modified Wieand condition.  A complete extension of Wieand's condition, under which the Bahadur and Pitman efficiencies coincide, remains an open problem.


\section{Wishart Distributions and Hankel Transforms of Matrix Argument}
\label{propertieshankel_matrix}

\subsection{Preliminary results for the Wishart distributions}

Throughout the paper, all needed results on the zonal polynomials and on the special functions of matrix argument are provided by Herz \cite{herz}, Muirhead \cite{muirhead}, or Richards \cite{ref1}, so we will generally conform to the notation in those sources.  We denote the zero matrix of any order by $0$, the order being always determined by the context;further $I_m$ will denote the $m \times m$ identity matrix.  We also denote by $\mathbb{R}^{m \times m}$ the space of $m \times m$ (real) matrices, by $\mathcal{S}^{m \times m}$ the space of $m \times m$ symmetric matrices, by $\mathcal{P}_{+}^{m \times m}$ the cone of $m \times m$ positive-definite matrices, and by $O(m)$ the group of $m \times m$ orthogonal matrices.  To specify that $Y \in \mathcal{P}_{+}^{m \times m}$, we usually write $Y > 0$; more generally, we write $Y_1 > Y_2$ whenever $Y_1 - Y_2 > 0$.  Further, we denote the trace of $Y$ by $\tr(Y)$, the determinant of $Y$ by $\det(Y)$  and we write $\etr(Y)$ for $\exp(\tr Y)$. 

The {\it multivariate gamma function} is defined by
$$
\Gamma_m (a) = \int_{R > 0} (\det R) ^{a-\frac{1}{2}(m+1)} \, \etr(-R) \dd R,
$$
for $a \in \mathbb{C}$, Re$(a) > \frac{1}{2} (m-1)$; this integral is well-known to have the explicit formula, 
$$
\Gamma_m(a)= \pi^{m(m-1)/4} \ \prod_{j=1}^m \Gamma \Big(a-\tfrac12(j-1)\Big).
$$

A $m \times m$ positive-definite random matrix $X$ is said to have a \textit{Wishart distribution} if its probability density function (p.d.f.) is of the form 
\begin{equation}
\label{wishartpdf}
f(X)=\frac{1}{\Gamma_m(\alpha)} (\det \Sigma)^{\alpha} (\det X)^{\alpha-\frac{1}{2}(m+1)} \etr(-\Sigma X),
\end{equation}
$X > 0$, where $\alpha > \frac{1}{2} (m-1)$ and $\Sigma > 0$.  We write $X \sim W_m(\alpha, \Sigma)$ whenever (\ref{wishartpdf}) holds.  The parameter $\alpha$ is called the \textit{shape parameter} and $\Sigma$ is called the \textit{scale matrix} of $X$.  If $\alpha$ is a half-integer then $2\alpha$ is called the \textit{degrees-of-freedom of} $X$.  In general, $E(X)=\alpha \Sigma^{-1}$; also, if $M$ is a $q \times m$ matrix of rank $q$, where $q \le m$, then $M X M' \sim W_q(\alpha, (M \Sigma^{-1} M')^{-1})$ \cite[p.~92, Theorem 3.2.5]{muirhead}.

A \textit{partition} $\kappa=(k_1,\dotsc, k_m)$ is a vector of nonnegative integers, listed in non-increasing order.  The \textit{weight} of $\kappa$ is $|\kappa|=k_1+\dotsc+k_m$, and the \textit{length}, $\ell(\kappa)$, of $\kappa$ is the number of non-zero $k_j$, $j=1,\dotsc,m$. 

For $a \in \mathbb{C}$ and $k=0,1,2,\ldots$, the \textit{shifted factorial} is defined as $(a)_k = a(a+1)(a+2)\cdots (a+k-1)$.  For any partition $\kappa = (k_1,\dotsc, k_m)$, the {\it partitional shifted factorial} is defined as 
$$
[a]_\kappa = \prod_{j=1}^m \big( a-\tfrac{1}{2}(j-1)\big)_{k_j}.
$$

For $Y \in \mathcal{S}^{m \times m}$, we denote by $\det_j(Y)$ the $j$th principal minor of $Y$, $j=1,\ldots,m$.  For any partition $\kappa$, the \textit{zonal polynomial} $C_{\kappa}(Y)$ is defined as 
\begin{equation}
\label{zonal}
C_{\kappa}(Y) =  C_{\kappa}(I_m) \, (\det Y)^{k_m} \int_{O(m)} \prod\limits_{j=1}^{m-1} ({\det}{}_j(H Y H^{-1}))^{k_j-k_{j+1}} \, \dd H,
\end{equation}
where $\dd H$ is the normalized Haar measure on $O(m)$ \cite[(35.4.2)]{ref1}.  By\eqref{zonal}, $C_{\kappa}(Y)$ is homogeneous of degree $|\kappa|$.  

It also follows from the invariance of the Haar measure that $C_{\kappa}(HYH') = C_{\kappa}(Y)$ for all $H \in O(m)$ and $Y \in \mathcal{S}^{m \times m}$; hence, $C_{\kappa}(Y)$ depends only on the eigenvalues of $Y$ and it is a symmetric function of the eigenvalues.  Suppose that $Z \in \mathcal{S}^{m \times m}$ and that $Y^{1/2}$ denotes the unique positive definite square root of $Y \in \mathcal{P}_+^{m \times m}$.  Since the matrices $Y^{1/2}ZY^{1/2}$, $YZ$, and $ZY$ all have the same eigenvalues we will follow a widely-adopted convention, writing $C_\kappa(YZ)$ or $C_\kappa(ZY)$ for $C_\kappa(Y^{1/2}ZY^{1/2})$; throughout the paper, we retain this convention for all orthogonally invariant functions of matrix argument.


With the normalization 
\begin{equation}
\label{zonal_identitymatrix}
C_{\kappa}(I_m)=2^{2 |\kappa|}| \, \kappa|! \, [m/2]_{\kappa} \, \frac{ \prod_{i < j}^{\ell(\kappa)} (2k_i - 2k_j - i + j)}{\prod_{i=1}^{\ell(\kappa)} (2k_i + \ell(\kappa) -i )! },
\end{equation}
the zonal polynomials satisfy the identity, 
\begin{equation}
\label{trace_zonal}
(\tr Y )^{k}=\sum_{|\kappa|=k} C_{\kappa}(Y),
\end{equation}
$k=0,1,2, \dotsc$ (see \cite[Eq.~(iii), p.~228]{muirhead} or \cite[Eq.~(35.4.6)]{ref1}).  Further, for $Y, Z \in \mathcal{S}^{m \times m}$, the zonal polynomials satisfy the {\it mean-value property} \cite[p.~243]{muirhead}, 
\begin{equation}
\label{meanvalue}
\int_{O(m)} C_\kappa(HYH'Z) \dd H = \frac{C_\kappa(Y) \, C_\kappa(Z)}{C_\kappa(I_m)}.
\end{equation}

We will also need in the sequel the identity, 
\begin{equation}
\label{sum_zonalfactorial}
\sum_{|\kappa|=k} C_{\kappa}(I_m) [a]_{\kappa}=(m \, a)_k,
\end{equation}
$a \in \mathbb{C}$, $k=0,1,2,\dotsc$.  This result is established by applying a power series identity, 
\begin{equation}
\label{sum_zonalfactorial1}
\sum_{k=0}^{\infty} \frac{t^k}{k!} \sum_{|\kappa|=k} C_{\kappa}(I_m) [a]_k = (\det(I_m-tI_m))^{-a},
\end{equation}
$|t| <1$; see \cite[p.~495, Eq.~(143)]{james}, \cite[p.~259, Eq.~(4)]{muirhead}.  Writing 
\begin{equation}
\label{sum_zonalfactorial2}
(\det(I_m - tI_m))^{-a} \equiv (1-t)^{-m \, a}=\sum_{k=0}^\infty \frac{t^k}{k!} \ (m\, a)_k,
\end{equation}
then (\ref{sum_zonalfactorial}) is obtained by comparing the coefficients of $t^k$ in (\ref{sum_zonalfactorial1}) and (\ref{sum_zonalfactorial2}).  

The zonal polynomials also satisfy a Laplace transform identity \cite[p.~248]{muirhead}: For Re$(a)> \tfrac12(m-1)$, $Z > 0$, and $M \in \mathcal{S}^{m \times m}$,
\begin{equation}
\label{zonalintegral}
\int_{R > 0} C_{\kappa}(M R) (\det R)^{a-\tfrac12(m+1)} \etr(-R Z) \dd R = [a]_{\kappa} \Gamma_m(a) (\det Z)^{-\alpha} C_{\kappa} (M Z^{-1}).
\end{equation}
For $\kappa = 0$, this result reduces to 
\begin{equation}
\label{gammaintegral}
\int_{R > 0} (\det R)^{a-\tfrac12(m+1)} \etr(-R Z) \dd R = \Gamma_m(a) (\det Z)^{-a},
\end{equation}
from which we confirm that (\ref{wishartpdf}) is a probability density function \cite[p.~61
]{muirhead}.

\subsection{Bessel functions and Laguerre polynomials of matrix argument}
\label{besselLaguerre_matrixargument}

The Bessel function of matrix argument, first treated in detail by Herz \cite{herz}, can be defined in several ways.  Let $\nu \in \mathbb{C}$ be such that $-\nu +\tfrac12(j-m) \notin \mathbb{N}$ for all $j=1,\ldots,m$; these restrictions ensure that $[\nu+\tfrac12(m+1)]_\kappa \neq 0$ for all partitions $\kappa$.  Following Muirhead \cite[Chapter 7]{muirhead}, the \textit{Bessel function (of the first kind) of order $\nu$} is defined for $Y \in \mathcal{S}^{m \times m}$ as 
\begin{equation}
\label{besselseriesdef_matrixargument}
A_{\nu}(Y)=\frac{1}{\Gamma_m(\nu+\tfrac12(m+1))}\sum\limits_{k=0}^{\infty} \frac{(-1)^k}{k!} \sum\limits_{|\kappa|=k} \frac{1}{[\nu+\frac{1}{2}(m+1)]_{\kappa}} C_{\kappa}(Y).
\end{equation}
We also refer to \cite{farautkoranyi,{grossrichards},james,ref1} for further details of these Bessel functions. In particular, the series (\ref{besselseriesdef_matrixargument}) converges absolutely for all $Y \in \mathcal{S}^{m \times m}$ \cite[Theorem 6.3]{grossrichards}.


For Re$(\nu) > \tfrac{1}{2}(m-2)$, the Bessel function $A_\nu$ is also given by Herz's generalization of the classical \textit{Poisson integral} \cite[Eq.~(3.6\textprime)]{herz}: For any $m \times m$ matrix $V$, 
\begin{equation}
\label{besselintegraldef_matrixargument}
A_{\nu}(V'V) =\frac{1}{\pi^{m^2/2} \Gamma_m(\nu+\frac{1}{2})} \int_{Q'Q < I_m} { \etr(2 \, \mi V' Q) \, (\det(I_m - Q'Q))^{\nu - \tfrac12 m}}\, \dd Q,
\end{equation}
where $\mi = \sqrt{-1}$ and the integral is with respect to Lebesgue measure on the set $\{Q \in \mathbb{R}^{m \times m}: QQ' < I_m\}$.  This result leads to an inequality 
that will arise repeatedly in the sequel.

\begin{lemma}
For Re$(\nu) > \tfrac{1}{2}(m-2)$ and $V \in \mathbb{R}^{m \times m}$,
\begin{equation}
\label{besselineq1_matrixargument}
\big| A_\nu(V'V) \big| \le \frac{1}{\Gamma_m(\nu+\tfrac12(m+1))}.
\end{equation}
\end{lemma}

\Pro
Since $|\etr(2\, \mi V' Q)| \le 1$ then it follows from (\ref{besselseriesdef_matrixargument}) and (\ref{besselintegraldef_matrixargument}) that 
\begin{align*}
\big| A_{\nu}(V' V) \big| &\le \frac{1}{\pi^{m^2/2} \Gamma_m(\nu+\frac{1}{2})} \int_{Q'Q < I_m} (\det(I_m - Q'Q))^{\nu - \tfrac12 m} \, \dd Q \nonumber \\
&= A_{\nu}(0) 
=\frac{1}{\Gamma_m(\nu+\tfrac12(m+1))}. \qquad\qquad\qed
\end{align*}

\medskip

For Re$(\nu)> -1$, $M$ symmetric, and $Z > 0$, the Bessel function of matrix argument satisfies the Laplace transform identity, 
\begin{equation}
\label{besselintegral}
\int_{R > 0} \etr(-R Z) A_{\nu}(M R) (\det R)^{\nu} \, \dd R = \etr(-MZ^{-1}) \, (\det Z)^{-\nu-\tfrac12(m+1)}.
\end{equation}
Indeed, this identity is Herz's original definition of $A_\nu(R)$ \cite[Eq.~(2.5)]{herz}.  

Herz \cite[Eq.~(5.8)]{herz} also obtained a fundamental generalization of a classical formula known as \textit{Weber's second exponential integral}: For Re$(\nu) > -1$, $m \times m$ symmetric matrices $\Lambda$ and $M$, and $Z > 0$,
\begin{multline}
\label{besselproductintegral}
\int_{R > 0} \etr(-R Z)  A_{\nu}(\Lambda R) A_{\nu}(M R) (\det R)^{\nu} \,\dd R \\
= (\det Z)^{-\nu-\tfrac12(m+1)} \, \etr (-(\Lambda+M)Z^{-1}) \ A_{\nu}(-\Lambda Z^{-1} M Z^{-1}).
\end{multline}

Let  $a, b \in \mathbb{C}$ where $-b + \tfrac12(j+1) \notin \mathbb{N}$, $j=1,\ldots,m$.  The \textit{confluent hypergeometric function of matrix argument} is defined, for $Y \in \mathcal{S}^{m \times m}$, as 
\begin{equation}
\label{1F1_matrixargument}
{_1}F_1(a;b;Y) = \sum\limits_{k=0}^{\infty} \frac{1}{k!} \sum\limits_{|\kappa|=k} \frac{[a]_{\kappa}}{[b]_{\kappa}} C_{\kappa} (Y).
\end{equation}
We will make repeated use of \textit{Kummer's formula} \cite[Eq.~(2.8)]{herz}, \cite [p.~265] {muirhead}, \cite[\S 35.8]{ref1}: 
\begin{equation}
\label{kummer_matrixargument}
{_1}F_1(a;b;Y) = \etr(Y) \ {_1}F_1(b-a ; b ; -Y).
\end{equation}
There is a Laplace transform relationship between the Bessel function $A_\nu$ and the confluent hypergeometric function ${}_1F_1$ function \cite[p.~489, Eq.~(2.11)]{herz}: For Re$(a) > \tfrac12(m-1)$, symmetric $M$, and $Z > 0$, 
\begin{multline}
\label{laplacebessel1f1}
\Gamma_m(\nu+\tfrac{1}{2}(m+1)) \int_{R > 0} A_{\nu}(MR) (\det R)^{a-\frac{1}{2}(m+1)} \etr (- RZ) \,\dd R \\
= \Gamma_m(a) \, (\det Z)^{-a}  \,{_1}F_1\big( a; \nu+\tfrac12(m+1); -M Z^{-1}\big).
\end{multline}
This result can also be proved by expressing $A_{\nu}( MR )$ as a series of zonal polynomials and then applying \eqref{zonalintegral} to integrate term-by-term.

Given partitions $\kappa$ and $\sigma$, we denote by $\binom{\kappa}{\sigma}$ the \textit{generalized binomial coefficient} \cite[pp.~267-269]{muirhead}, \cite[Eq.~(35.6.3)]{ref1}.  For $\gamma > -1$ and $Y \in \mathcal{S}^{m \times m}$, the \textit{(generalized) Laguerre polynomial} $L^{(\gamma)}_{\kappa}(Y)$, corresponding to $\kappa$, is defined as 
\begin{equation}
\label{laguerredef}
L^{(\gamma)}_\kappa(Y) = \big[ \gamma + \tfrac12 (m+1) \big]_{\kappa} C_{\kappa}(I_m) \sum\limits_{s=0}^ {|\kappa |} \sum_{|\sigma|=s} \binom{\kappa}{\sigma} \frac{C_{\sigma}(-Y)}{[\gamma + \frac{1}{2}(m+1)]_{\sigma} C_{\sigma}(I_m)},
\end{equation}
Setting $Y = 0$ in \eqref{laguerredef}, we obtain 
\begin{equation}
\label{laguerre_at0_matrixcase}
L_{\kappa}^{(\gamma)}(0)= [\gamma+\tfrac12(m+1)]_{\kappa} \ C_{\kappa}(I_m).
\end{equation}
The \textit{normalized (generalized) Laguerre polynomial} corresponding to $\kappa$ is defined by
\begin{equation}
\label{laguerre2_matrixcase}
\mathcal{L}^{(\gamma)}_{\kappa} (Y) : = \left( | \kappa |! \ L_{\kappa}^{(\gamma)}(0) \right)^{-1/2} \ L^{(\gamma)}_{\kappa} (Y),
\end{equation}
$Y \in \mathcal{S}^{m \times m}$.  By \cite[Theorem 7.6.5]{muirhead}, the polynomials $\mathcal{L}^{(\gamma)}_{\kappa} $ are orthonormal with respect to the Wishart distribution $W(\gamma+\tfrac12(m+1),I_m)$: 
\begin{equation}
\label{laguerreorthog_matrixargument}
\frac{1}{\Gamma_m \big( \gamma + \tfrac12 (m+1) \big)} \int_{Y > 0} \mathcal{L}^{(\gamma)}_{\kappa}(Y) \mathcal{L}^{(\gamma)}_{\sigma}(Y) \, (\det Y)^\gamma \, \etr(-Y) \dd Y = 
\begin{cases}
 1, & \kappa=\sigma \\
 0, & \kappa \neq \sigma
\end{cases}.
\end{equation}

By \cite[p.~282]{muirhead}, for $\gamma > -1$ and $Z > 0$, there holds the Laplace transform, 
\begin{multline}
\label{laguerreintegral_matrixcase}
\int_{Y>0} { \etr(-YZ) (\det Y)^{\gamma} L^{(\gamma)}_{\kappa} (Y) }\, \dd Y \\
= [\gamma+ \tfrac12 (m+1)]_{\kappa} \ \Gamma_m(\gamma+\tfrac12(m+1))(\det Z)^{-\gamma-\tfrac12 (m+1)} \ C_{\kappa}(I_m-Z^{-1}).
\end{multline}
Further, by \cite[Theorem 7.6.4, p.~284]{muirhead}, for $\gamma > -1$ and $Z \in \mathcal{S}^{m \times m}$,
\begin{equation}
\label{integralrepreslaguerre_matrixcase}
 \etr(-Z) L_{\kappa}^{(\gamma)} (Z) = \int_{Y>0} { \etr(-Y) (\det Y)^{\gamma}\ C_{\kappa}(Y) \ A_{\gamma}( ZY)}\, \dd Y.
\end{equation}

\smallskip

\begin{lemma}
\label{lemmalaguerreintegral_matrixcase}
Let $Z > 0$ and $\gamma > -1$, then
\begin{equation}
\label{boundlaguerre_matrixcase}
| L_{\kappa}^{(\gamma)}(Z) | \ \le \etr(Z) \ [\gamma+ \tfrac12 (m+1)]_{\kappa} \ C_{\kappa}(I_m).
\end{equation}
Also, for $v \in \mathbb{R}$, $v > 0$, 
\begin{align}
\label{integrallemmalaguerre_matrixcase}
\int_{Y>0} &{ \etr(-v Y) (\tr Y) (\det Y)^{\gamma} L^{(\gamma)}_{\kappa} (Y) }\, \dd Y \nonumber\\ 
&= [\gamma+ \tfrac12 (m+1)]_{\kappa} \ \Gamma_m(\gamma+\tfrac12(m+1))  \ C_{\kappa}(I_m) \nonumber\\
 & \quad\quad \times (v-1)^{|\kappa|-1} v^{-[m(\gamma +(m+1)/2)+|\kappa|+1]} \left(m (\gamma+\tfrac12(m+1)) (v-1) -|\kappa|\right).
\end{align}
\end{lemma}

\begin{proof}
By (\ref{besselineq1_matrixargument}) and (\ref{integralrepreslaguerre_matrixcase}), 
\begin{align*}
| L_{\kappa}^{(\gamma)}(Z) | \ & \le \etr(Z) \ \int_{Y>0} { \etr(-Y) (\det Y)^{\gamma}\ C_{\kappa}(Y) \ \big| A_{\gamma}( ZY) \big|}\, \dd Y \nonumber \\
& \le \frac{1}{\Gamma_m(\gamma+\tfrac12(m+1))} \, \etr(Z) \ \int_{Y>0} { \etr(-Y) (\det Y)^{\gamma}\ C_{\kappa}(Y)}\, \dd Y.
\end{align*}
Applying (\ref{zonalintegral}) to evaluate the latter integral, we obtain \eqref{boundlaguerre_matrixcase}.  

To establish \eqref{integrallemmalaguerre_matrixcase}, we substitute $Z=v I_m$ into (\ref{laguerreintegral_matrixcase}), obtaining 
\begin{multline*}
\int_{Y>0} { \etr(-v Y) (\det Y)^{\gamma} L^{(\gamma)}_{\kappa} (Y) }\, \dd Y \\
= [\gamma+ \tfrac12 (m+1)]_{\kappa} \ \Gamma_m(\gamma+\tfrac12(m+1)) \ (v-1)^{|\kappa|} \ v^{-[m(\gamma +(m+1)/2)+|\kappa|]} \  C_{\kappa}(I_m).
\end{multline*}
Differentiating both sides of the latter equation with respect to $v$ and simplifying the outcome, we obtain the stated result.
\end{proof}

\subsection{Hankel transforms of matrix argument}

Throughout the rest of the paper, if $X$ is a random entity, we denote expectation with respect to the distribution of $X$ by $E_X$ or simply by $E$ whenever the context is clear.  

Let $X > 0$ be a random matrix with probability density function $f(X)$. For Re$(\nu) > \frac{1}{2}(m-2)$, we define the {\it Hankel transform of order $\nu$ of $X$} as the function
\begin{equation}
\label{hankeltransformdefinition_matrixargument}
\mathcal{H}_{X,\nu}( T ) = E_X \big[\Gamma_m(\nu+\tfrac{1}{2}(m+1)) \ A_{\nu}(T X)\big],
\end{equation}
$T > 0$.  The Hankel transform satisfies the following properties:

\begin{lemma} 
\label{existencehankel_matrixargument}
For {\rm{Re}}$(\nu) > \frac{1}{2}(m-2)$,  ${|\mathcal{H}_{X,\nu}(T)|} \le 1$ for all $T >0$, and  $\mathcal{H}_{X,\nu}(T)$ is a continuous function of $T$.
\end{lemma}

\begin{proof}  By (\ref{besselineq1_matrixargument}), 
$$
\Gamma_m(\nu+\tfrac{1}{2}(m+1)) \, \big| A_{\nu}( T X) \big| \le 1
$$
for all $T, X >0$.  Therefore, by the triangle inequality, $|\mathcal{H}_{X,\nu}(T)| \le E_X(1) = 1$.

Since $A_\nu(TX)$ is bounded and continuous in $T > 0$ for every fixed $X  > 0$, 
the continuity of $\mathcal{H}_{X,\nu}(T)$ follows by Dominated Convergence.
\end{proof}


\begin{example} 
{\label{hankeltransformwishartdistn_example}
\rm 
Let $X \sim W_m (\alpha, \Sigma)$, $\alpha > \frac{1}{2} (m-1)$, $\Sigma > 0$. For $T > 0$, it follows from the definition  (\ref{hankeltransformdefinition_matrixargument}) of the Hankel transform that 
$$
\mathcal{H}_{X, \nu}(T)= \frac{\Gamma_m(\nu+\tfrac{1}{2}(m+1))}{\Gamma_m(\alpha)} (\det \Sigma )^{\alpha} \int_{X > 0} { A_{\nu}( T X ) (\det X)^{\alpha-\frac{1}{2}(m+1)} \etr (-\Sigma X)}\,\dd X.
$$
Applying \eqref{laplacebessel1f1} to calculate this integral, we obtain 
\begin{align}
\label{hankeltransformwishartdistn}
\mathcal{H}_{X, \nu}(T)={_1}F_1\big( \alpha; \nu+\tfrac12(m+1) ; -T \Sigma^{-1} \big).
\end{align}
For the case in which $\nu=\alpha-\tfrac12(m+1)$, (\ref{hankeltransformwishartdistn}) reduces to
$$
\mathcal{H}_{X, \alpha-\tfrac12(m+1)}(T)={_1}F_1\big( \alpha; \alpha ; -T \Sigma^{-1} \big)=\etr(-T \Sigma^{-1}).
$$
}\end{example}

\smallskip

\begin{example} 
\label{hankeltransformwishartmixture}
{\rm 
Let $Z \sim W_m(\alpha, I_m)$ and $X > 0$ be a $m \times m$ random matrix that is independent of $Z$.  For $T > 0$,  
\begin{align}
\label{hankeltransformwishartmixture_result}
E_{Z} \, \mathcal{H}_{X, \nu} (T^{1/2} Z T^{1/2}) &= E_{X} \, \mathcal{H}_{Z, \nu} (T^{1/2} X T^{1/2}) \nonumber\\
&=E_{X} \, {_1}F_1 \big(\alpha;\nu+\tfrac12 (m+1); -T X\big).
\end{align}

To prove this result, we again apply (\ref{hankeltransformdefinition_matrixargument}) and the independence of $X$ and $Z$, obtaining
\begin{align*}
E_{Z} \, \mathcal{H}_{X, \nu} (T^{1/2} Z T^{1/2}) &= E_{Z,X} \Gamma_m(\nu+\tfrac{1}{2}(m+1)) A_{\nu}( T^{1/2} Z T^{1/2} X ) \nonumber\\
&=E_{X} \, E_{Z} \, \Gamma_m(\nu+\tfrac{1}{2}(m+1)) \, A_{\nu}(T^{1/2} Z T^{1/2} X).
\end{align*}
Since $A_{\nu}(T^{1/2} Z T^{1/2} X) = A_{\nu}( T^{1/2} X T^{1/2} Z)$, we have
\begin{equation}
\label{hankelwishartmixture1}
E_{Z} \, \mathcal{H}_{X, \nu} (T^{1/2} Z T^{1/2}) = E_{X} \, \mathcal{H}_{Z, \nu} (T^{1/2} X T^{1/2}).
\end{equation}
Applying Example \ref{hankeltransformwishartdistn_example}, we obtain
\begin{equation}
\label{hankelwishartmixture2}
E_{X} \mathcal{H}_{Z, \nu} (T^{1/2} X T^{1/2}) = E_{X} \, {_1}F_1 \big(\alpha;\nu+\tfrac12 (m+1); -T X\big).
\end{equation}
Combining (\ref{hankelwishartmixture1}) and (\ref{hankelwishartmixture2}), we obtain (\ref{hankeltransformwishartmixture_result}). 

In particular, if $\nu=\alpha-\frac{1}{2}(m+1)$ then, by Kummer's formula \eqref{kummer_matrixargument}, we obtain 
\begin{align*}
E_{Z} \mathcal{H}_{X, \alpha-\tfrac12(m+1)} (T^{1/2} Z T^{1/2}) &= E_{X} \mathcal{H}_{Z,  \alpha-\tfrac12(m+1)} (T^{1/2} X T^{1/2}) \\
&= E_{X} \etr (-T X),
\end{align*}
the Laplace transform of $X$.
}\end{example}

\smallskip

Throughout the remainder of the paper, if $X$ and $Y$ are random entities we write $X \stackrel{d}{=} Y$ whenever $X$ and $Y$ have the same distribution.  If $\{X_n, n \ge 1\}$, is a sequence of random entities, we write $X_n \xrightarrow{d} X$ whenever $X_n$ converges in distribution to $X$.   

\medskip

\begin{theorem}{\rm (Uniqueness of the Hankel transform).}
\label{uniqueness_matrixargument}
Let $X$ and $Y$ be $m \times m$ positive definite random matrices with Hankel transforms $\mathcal{H}_{X, \nu }$ and $\mathcal{H}_{Y, \nu}$, respectively.  Then $\mathcal{H}_{X, \nu} = \mathcal{H}_{Y, \nu}$ if and only if $X \stackrel{d}{=} Y$.
\end{theorem}

\begin{proof}
If $X \stackrel{d}{=} Y$ then it is clear that $\mathcal{H}_{X, \nu}$=$\mathcal{H}_{Y, \nu}$.

Conversely, suppose that $Z \sim W_m (\nu+\frac{1}{2}(m+1), I_m)$, independently of $X$ and $Y$. Let $\Psi_{X}$ and $\Psi_{Y}$ be the Laplace transforms of $X$ and $Y$ respectively; then, for all $T > 0$, 
\begin{eqnarray*}
\mathcal{H}_{X, \nu}(T^{1/2} Z T^{1/2})=\mathcal{H}_{Y, \nu}(T^{1/2} Z T^{1/2}), 
\end{eqnarray*} 
and therefore 
\begin{eqnarray}
\label{hankelxyz_matrix}
E_{Z} \,\mathcal{H}_{X, \nu}(T^{1/2} Z T^{1/2}) = E_{Z} \, \mathcal{H}_{Y, \nu}(T^{1/2} Z T^{1/2}).
\end{eqnarray}
By Example \ref{hankeltransformwishartmixture},
\begin{eqnarray}
\label{laplacex_matrix}
E_{Z} \, \mathcal{H}_{X, \nu}(T^{1/2} Z T^{1/2}) = E_{X} \,\etr (-T X) = \Psi_{X}(T) ,
\end{eqnarray}
and
\begin{eqnarray}
\label{laplacey_matrix}
E_{Z} \, \mathcal{H}_{Y, \nu}(T^{1/2} Z T^{1/2}) = E_{Y} \,\etr (-T Y) = \Psi_{Y}(T)
\end{eqnarray}
for all $T > 0$.  Combining (\ref{hankelxyz_matrix}), (\ref{laplacex_matrix}) and (\ref{laplacey_matrix}), we obtain $\Psi_{X}(T)=\Psi_{Y}(T)$, for all $T > 0$. By the uniqueness theorem for multivariate Laplace transforms \cite[p.~16, Theorem 2.1.9]{farrell} we conclude that $X \stackrel{d}{=} Y$.
\end{proof}

\medskip

We denote by $L^2_\nu$ the space of functions $\phi: \mathcal{P}^{m \times m}_+ \to \mathbb{C}$ such that 
$$
\int_{\mathcal{P}^{m \times m}_+} |\phi(X)|^2 \ (\det X)^{-\nu} \, \dd X < \infty.
$$

The following inversion theorem is obtained by applying the Hankel inversion theory of Herz \cite[Section 3]{herz}.  We refer to Hadjicosta \cite{hadjicosta19} for the full details. 

\begin{theorem}{\rm (Inversion of the Hankel transform)}.
\label{hankelinversion_matrixargument}
Let $X > 0$ be a random matrix with Hankel transform $\mathcal{H}_{X, \nu}$, and with a probability density function $f \in L^2_\nu$.  Then, 
$$
f(X)=\frac{1}{\Gamma_m(\nu+\tfrac{1}{2}(m+1))} \ \int_{\mathcal{P}^{m \times m}_+} { A_{\nu}( T X) \, (\det T X)^{\nu} \, \mathcal{H}_{X, \nu}(T)}\, \dd T.
$$
\end{theorem}


\begin{theorem} {\rm (Hankel Continuity)}.
Let $\{ X_{n}, n \in \mathbb{N} \}$ be a sequence of $m \times m$ positive-definite random matrices with corresponding Hankel transforms $ \{ \mathcal{H}_{X_n}, n \in \mathbb{N} \}$.  If there exists a $m \times m$ positive semi-definite random matrix $X$ with Hankel transform $\mathcal{H}_X$ such that $X_{n} \xrightarrow{d} X$ then, for each $T > 0$, 
\begin{equation}
\label{continuityhankel_matrixargument}
\lim_{n \rightarrow \infty} \mathcal{H}_{X_n}(T)=\mathcal{H}_X(T).
\end{equation}

Conversely, suppose there exists a function $\mathcal{H}: \mathcal{P}_{+}^{m \times m} \rightarrow \mathbb{R}$ such that $\mathcal{H}(T) \to 1$ as $T \to 0$, $\mathcal{H}$ is continuous at $0$, and (\ref{continuityhankel_matrixargument}) holds. Then $\mathcal{H}$ is the Hankel transform of an $m \times m$ positive semi-definite random matrix $X$, and $X_{n} \xrightarrow{d} X$.
\end{theorem}

\begin{proof}
Suppose that $X_{n} \xrightarrow{d} X$ then, by the Continuous Mapping Theorem for random vectors \cite[p.~336]{ref12}, 
$
A_{\nu}( T X_n) \xrightarrow{d} 
A_{\nu}( T X )
$ 
as $n \to \infty$, for all $T > 0$.  By (\ref{besselineq1_matrixargument}), $A_{\nu}( T X_n)$ is uniformly bounded for all $n \in \mathbb{N}$ and $T > 0$; thus, by the Dominated Convergence Theorem,
$
E
A_{\nu}( T X_n)
\to 
E
A_{\nu}( T X)
$ 
as $n \to \infty$, for all $T > 0$, and therefore \eqref{continuityhankel_matrixargument} holds.  

Conversely, suppose that $Z \sim W_m(\nu+\frac{1}{2}(m+1), I_m)$ where $Z$ is independent of the sequence $\{ X_{n}, n \in \mathbb{N} \}$.  Also, let $\Psi_{X_{n}}$ be the Laplace transform of $X_{n}$. By Example \ref{hankeltransformwishartmixture}, we have 
$$
\Psi_{X_{n}}(T)=E_{Z}[\mathcal{H}_{n}(T^{1/2} Z T^{1/2})], 
$$
for all $T > 0$. Further, by Lemma \ref{existencehankel_matrixargument}, $|\mathcal{H}_{n}(T^{1/2} Z T^{1/2})| \ \le 1$ for all $T > 0$. Thus, by the Dominated Convergence Theorem, as $n \to \infty$, 
$$
\Psi_{X_{n}}( T) \to
E_{Z}[ \mathcal{H}(T^{1/2} Z T^{1/2})]=\Psi(T), 
$$
for all $T > 0$. Since $\mathcal{H}$ is continuous at $0$ and $\mathcal{H} (0)=1$ then $\Psi(T)$ also is continuous at $0$ and $\Psi(0)=1$.  By the continuity theorem for multivariate Laplace transforms \cite[p.~63, Theorem 4.3]{kallenberg}, there is a $m \times m$ positive semi-definite random matrix $X$ whose Laplace transform is $\Psi$, and $X_{n} \xrightarrow{d} X$. 
\end{proof}

\smallskip

The next result constitutes a characterization of the Wishart distributions using the Hankel transform $\mathcal{H}_{X,\nu}$, where Re($\nu) > \tfrac12(m-2)$.  The result enables the extension, to the Wishart case, of some results of Baringhaus and Taherizadeh \cite{BT2013} on a supremum norm test statistic.  

\begin{theorem}
\label{hankelopenset_matrixargument}
Let ${X}$ be an $m \times m$ positive-definite random matrix with an orthogonally invariant distribution and Hankel transform $\mathcal{H}_{{X}, \nu}$.  If there exist $\epsilon > 0$ and $\alpha > \frac{1}{2} (m-1)$ such that for all $T$ satisfying $0 < T \le \epsilon I_m$, 
$$
\mathcal{H}_{{X},\nu}(T) = {}_1F_1(\alpha;\nu+\tfrac12(m+1);-T),
$$
then $\widetilde{X} \sim W_m(\alpha,I_m)$.  
\end{theorem}

We refer the reader to Hadjicosta \cite{hadjicosta19}, where three proofs of this result are given.  We provide here the third and briefest proof, which uses the principle of analytic continuation. 

\begin{proof}
The Hankel transform, $\mathcal{H}_{{X},\nu}(T)$, of ${X}$ is holomorphic (analytic) in $T$. Also, the hypergeometric function ${}_1F_1(\alpha;\nu+\tfrac12(m+1); -T)$ is holomorphic in $T$. Since these two functions agree on the open neighborhood $\{T: 0 < T < \epsilon I_m\}$ then, by analytic continuation, they agree wherever they both are well-defined. Since they both are well-defined everywhere then we conclude that $\mathcal{H}_{{X},\nu}(T) = {}_1F_1(\alpha;\nu+\tfrac12(m+1); -T)$ for all $T > 0$. By Example \ref{hankeltransformwishartdistn_example} and Theorem \ref{uniqueness_matrixargument}, the uniqueness theorem for Hankel transforms, it follows that ${X} \sim W_m(\alpha,I_m)$.
\end{proof}


\subsection{Orthogonally invariant Hankel transforms of matrix argument}

For $\nu \in \mathbb{C}$ such that $-\nu +\tfrac12(j-m) \notin \mathbb{N}$, for all $j=1, \dotsc, m$, and $X, Y \in \mathcal{S}^{m \times m}$, the \textit{Bessel function (of the first kind) of order $\nu$ with two matrix arguments} is defined as the infinite series 
\begin{equation}
\label{besselseriesdef_2matrixargument}
A_{\nu} ( X, Y )=\frac{1}{\Gamma_m(\nu+\tfrac{1}{2}(m+1))} \ \sum\limits_{k=0}^{\infty} \frac{(-1)^k}{k!} \sum\limits_{|\kappa|=k} \frac{C_{\kappa} (X) C_{\kappa} (Y)}{[\nu+\frac{1}{2}(m+1)]_{\kappa}  C_{\kappa} (I_m)}.
\end{equation}

It is straightforward from \eqref{meanvalue} and \eqref{besselseriesdef_matrixargument} to see that 
\begin{equation}
\label{bessel_2matrixargument}
A_{\nu}(X,Y) = \int_{O(m)} A_{\nu}(HXH'Y) \dd H,
\end{equation}
$X, Y \in \mathcal{S}^{m \times m}$ \cite[p.~260]{muirhead}.  Also, by applying the inequality (\ref{besselineq1_matrixargument}) for $A_\nu(X)$, we obtain 
\begin{equation}
\label{2besselineq_matrixargument}
|A_{\nu}(X,Y)| \le \frac{1}{\Gamma_m(\nu+\tfrac{1}{2}(m+1))}.
\end{equation}

\smallskip

\begin{definition}
{\rm Let $X$ be an $m \times m$ positive-definite random matrix with p.d.f. $f(X)$. For Re$(\nu) > \frac{1}{2}(m-2)$ and $T > 0$, we define the {\it orthogonally invariant Hankel transform of order $\nu$ of $X$} as the function
\begin{equation}
\label{hankeltransformdefinition_2matrixargument}
\mathcal{\widetilde{H}}_{X,\nu}(T) = E_X \, \big[\Gamma_m(\nu+\tfrac{1}{2}(m+1)) A_{\nu}( T, X )\big].
\end{equation}
}
\end{definition}

\smallskip

\begin{remark}
{\rm By (\ref{bessel_2matrixargument}) and the definition (\ref{hankeltransformdefinition_matrixargument}) of $\mathcal{H}_{X,\nu}$, we have 
\begin{equation}
\label{hankel_2matrixargument}
\mathcal{\widetilde{H}}_{X,\nu}(T) = \int_{O(m)} {\mathcal{H}_{X, \nu} (H T H')} \, \dd H.
\end{equation}
Further, since $\int_{O(m)} \dd H = 1$, then $\mathcal{\widetilde{H}}_{X,\nu}$ also satisfies the properties in Lemma \ref{existencehankel_matrixargument}.
}\end{remark}

\medskip

Let  $a$, $b \in \mathbb{C}$, where $-b+\tfrac{1}{2}(j-1) \notin \mathbb{N}$, for all $j=1, \dotsc, m$. The \textit{confluent hypergeometric function of two matrix arguments} is defined, for $X, Y \in \mathcal{S}^{m \times m}$, as the infinite series,
\begin{equation}
\label{1F1seriesdef_2matrixargument}
{_1}F_1(a ; b ; X, Y)=\sum\limits_{k=0}^{\infty} \frac{1}{k!} \sum\limits_{|\kappa|=k} \frac{[a]_{\kappa}}{[b]_{\kappa}} \frac{C_{\kappa}(X) C_{\kappa} (Y)}{ C_{\kappa}(I_m)}.
\end{equation}
It is clear from the definition that 
$
{_1}F_1(a; b; X, I_m)={_1}F_1(a ; b; X).
$ 
Similar to \eqref{bessel_2matrixargument}, it follows from \eqref{meanvalue} that for $X, Y \in \mathcal{S}^{m \times m}$, 
\begin{equation}
\label{1F1_2matrixargument}
{_1}F_1(a;b;X,Y) = \int_{O(m)} {_1}F_1(a;b;HXH'Y) \dd H.
\end{equation}

\smallskip

\begin{example} 
\label{orthogonallyhankeltransformwishartdistn_example}
{\rm 
Let $X \sim W_m (\alpha, \Sigma)$ where $\alpha > \frac{1}{2} (m-1)$ and $\Sigma > 0$.  For $T > 0$, it follows from Example \ref{hankeltransformwishartdistn_example}, \eqref{hankel_2matrixargument}, and \eqref{1F1_2matrixargument} that 
\begin{align}
\label{orthogonallyhankeltransformwishartdistn}
\mathcal{\widetilde{H}}_{X,\nu}( T ) &= \int_{O(m)} {{_1}F_1(\alpha; \nu +\tfrac12(m+1); -H T H' \Sigma^{-1})} \, \dd H \nonumber \\
&= {_1}F_1\big( \alpha; \nu+\tfrac12(m+1) ; -T, \Sigma^{-1} \big).
\end{align}
}\end{example}

\smallskip

\begin{theorem}{ \rm (Uniqueness of orthogonally invariant Hankel transforms).}
\label{uniqueness_hankeltilde}
Let ${X}$ and ${Y}$ be $m \times m$ positive-definite random matrices with orthogonally invariant distributions and orthogonally invariant Hankel transforms $\mathcal{\widetilde{H}}_{{X}, \nu }$ and $\mathcal{\widetilde{H}}_{{Y}, \nu}$, respectively. Then $\mathcal{\widetilde{H}}_{{X}, \nu} = \mathcal{\widetilde{H}}_{{Y}, \nu}$ if and only if ${X} \stackrel{d}{=} {Y}$.
\end{theorem}

\begin{proof}
By Eq. (\ref{bessel_2matrixargument}) and the definition of the orthogonally invariant Hankel transform (\ref{hankeltransformdefinition_2matrixargument}), we have 
$$
\mathcal{\widetilde{H}}_{X,\nu}(T) = E_X E_H \, \Gamma_m(\nu+\tfrac{1}{2}(m+1)) A_{\nu}(HTH'X).
$$
Since the distribution of $\widetilde{X}$ is orthogonally invariant, $X \stackrel{d}{=} HXH'$ for all $H \in O(m)$; therefore, for all $T > 0$, 
$$
\mathcal{\widetilde{H}}_{X,\nu}(T) = E_X \, \Gamma_m(\nu+\tfrac{1}{2}(m+1)) A_\nu(TX) = \mathcal{H}_{X,\nu}(T),
$$
and similarly for $Y$. By applying Theorem \ref{uniqueness_matrixargument}, the Uniqueness Theorem for Hankel transforms, we deduce the desired result. 
\end{proof}

\section{Goodness-of-Fit Tests for the Wishart Distributions}
\label{goodnessoffittests_wishart}

\subsection{The test statistic}
\setcounter{equation}{0}

Let $X_1, \dotsc, X_n$ be independent, identically distributed (i.i.d.), $m \times m$ positive-definite random matrices, each with probability density function $f(X)$ and positive-definite mean $\mu=E(X_1)$. We assume also that the density function of $X_1$ is of the form 
\begin{equation}
\label{assumption_orthogonally_invariant}
f(X_1)=f_0({\mu}^{-1/2} X_1 {\mu}^{-1/2}),
\end{equation}
where $f_0$ is orthogonally invariant. 

\begin{lemma}
\label{lemma_orthogonally_invariant}
Under the assumption (\ref{assumption_orthogonally_invariant}), the distribution of $\mu^{-1/2} X_{1} \mu^{-1/2}$ is orthogonally invariant. 
\end{lemma}

\begin{proof}
Let $\widetilde{Y}=\mu^{-1/2} X_{1} \mu^{-1/2}$; then $X_1=\mu^{1/2} \widetilde{Y} \mu^{1/2}$ and the Jacobian of the transformation from $X_1$ to $\widetilde{Y}$ is $(\det \mu)^{(m+1)/2}$ \cite[p.~58]{muirhead}. Therefore, the p.d.f. of $\widetilde{Y}$ is 
$$
g(\widetilde{Y}) = (\det \mu)^{(m+1)/2} \, f(\mu^{1/2} \widetilde{Y} \mu^{1/2}) 
= (\det \mu)^{(m+1)/2} \, f_0 (\widetilde{Y}).
$$
Since $f_0$ is orthogonally invariant then it follows that $g$ is orthogonally invariant. 
\end{proof}

Denote by $P$ the distribution of $X_1$.  On the basis of the random sample $X_1, \dotsc, X_n$, we wish to test the null hypothesis, $H_0: P \in \{W_m(\alpha, \Sigma), \Sigma > 0 \}$, against the alternative, $H_1: P \cancel{\in} \{W_m(\alpha, \Sigma), \Sigma > 0 \}$, where $\alpha$ is known.

Since $\Sigma$ is unspecified by $H_0$, the data $X_1, \dotsc, X_n$ cannot be used to construct a test statistic. Thus, with $\bar{X}_{n}=n^{-1} \sum_{j=1}^n X_j$ denoting the sample mean, define $Y_j =\bar{X}_{n}^{-1/2} X_j \bar{X}_{n}^{-1/2}$, for $j=1,\dotsc,n$.   Under $H_0$, the distribution of $Y_1,\ldots,Y_n$ does not depend on $\Sigma$, so a test statistic can be based on them.  Let $P_0$ denote the probability measure corresponding to the $W_m(\alpha,I_m)$ distribution. For Re$(\nu) > \frac{1}{2}(m-2)$, define the \textit{empirical orthogonally invariant Hankel transform of order $\nu$} of $Y_1,\ldots,Y_n$ as
\begin{equation}
\label{empiricalhankel_matrixargument}
\mathcal{\widetilde{H}}_{n,\nu}(T) = \Gamma_m(\nu+\tfrac{1}{2}(m+1)) \frac{1}{n}\sum_{j=1}^{n} A_{\nu}(T,Y_j), 
\end{equation}
$T > 0$.  Further, define the test statistic 
\begin{equation}
\label{statisticc_wishart}
\boldsymbol{T}^2_{n,\nu} = n \, \int_{T > 0} {\big[\mathcal{\widetilde{H}}_{n, \nu}(T)-{_1}F_1(\alpha; \nu+\tfrac12 (m+1); -T/\alpha)\big]^2 }\, \dd P_0(T).
\end{equation}

To provide motivation for this test statistic, suppose that $H_0$ is valid; then $E(X_{1})=\alpha \Sigma^{-1}$ and, for large $n$, we can expect that $Y_j=\bar{X}_{n}^{-1/2} X_j \bar{X}_{n}^{-1/2} \simeq \alpha^{-1}{\Sigma}^{1/2} X_j {\Sigma}^{1/2}$, almost surely. 
By the Continuous Mapping Theorem, the sequence of random variables $A_{\nu}( T, Y_j)$ should approximate the i.i.d. sequence $A_{\nu}( T, \alpha^{-1} {\Sigma}^{1/2} X_j {\Sigma}^{1/2})$, $j=1, \dotsc, n$, for each $T > 0$ and for sufficiently large $n$. 
Applying to (\ref{empiricalhankel_matrixargument}) the Strong Law of Large Numbers, we can expect that, for large $n$, $\mathcal{\widetilde{H}}_{n, \nu}(T) \simeq \mathcal{\widetilde{H}}_{\alpha^{-1} \Sigma^{1/2} X_{1} \Sigma^{1/2}, \nu}(T)$, almost surely. 

By Example \ref{orthogonallyhankeltransformwishartdistn_example}, we deduce that  
\begin{align*}
\mathcal{\widetilde{H}}_{\alpha^{-1} \Sigma^{1/2} X_{1} \Sigma^{1/2}, \nu}(T)&={_1}F_1 \bigg(\alpha ; \nu+\tfrac12(m+1) ; -\alpha^{-1} T, I_m \bigg)\\
&={_1}F_1 \bigg(\alpha ; \nu+\tfrac12(m+1) ; -\alpha^{-1} T),
\end{align*}
for $T > 0$. Therefore, by Lemma \ref{lemma_orthogonally_invariant} and Theorem \ref{uniqueness_hankeltilde}, small values of $\boldsymbol{T}_{n,\nu}^2$ provide strong evidence in support of $H_0$, and we will reject $H_0$ for large values of $\boldsymbol{T}^2_{n,\nu}$.

\smallskip

{\it For the remainder of the paper, we set} 
$$
\nu=\alpha-\frac{1}{2}(m+1).
$$
Since $\nu >\frac{1}{2}(m-2)$ then $\alpha >\frac{1}{2}(2m-1)$. We also denote $\boldsymbol{T}^2_{n,\nu}$ and $\mathcal{\widetilde{H}}_{n,\nu}$ by $\boldsymbol{T}^2_n$ and $\mathcal{\widetilde{H}}_n$, respectively. By Kummer's formula (\ref{kummer_matrixargument}), the statistic (\ref{statisticc_wishart}) becomes 

\begin{equation}
\label{statistic_wishart}
\boldsymbol{T}^2_n= n \int_{T > 0} \bigg[\mathcal{\widetilde{H}}_n(T) - \etr(-T/\alpha)\bigg]^2 \, \dd P_0(T).
\end{equation}
This integral represents $\boldsymbol{T}^2_n$ as a weighted integral of the squared difference between the empirical orthogonally invariant Hankel transform $\mathcal{\widetilde{H}}_n$ and its almost sure limit under the null hypothesis.

\medskip

We now evaluate the test statistic $\boldsymbol{T}_n^2$ for a given random sample. 
\begin{proposition}
\label{helpfulrepr_matrixcase}
The test statistic (\ref{statistic_wishart}) is a $V$-statistic of order 2. Specifically,
$$
\boldsymbol{T}^2_{n} = \frac{1}{n}\sum_{i=1}^{n} \sum_{j=1}^{n} h(Y_{i},Y_{j})
$$
where, for $X ,Y > 0$, 
\begin{multline*}
h(X,Y)= \Gamma_m(\alpha) \etr (-X-Y ) \ A_{\nu}( -X, Y)\\
-\bigg( \frac{\alpha}{\alpha+1} \bigg)^{m \alpha} \bigg[ \etr \bigg(-\frac{\alpha}{\alpha+1} X \bigg) +  \etr \bigg(-\frac{\alpha}{\alpha+1} Y \bigg) \bigg] +\bigg( \frac{2}{\alpha}+1 \bigg)^{-m \alpha}.
\end{multline*}
\end{proposition}

\begin{proof}
After squaring the integrand in (\ref{statistic_wishart}), we see that there are three terms to be computed. First, 
\begin{align*}
n \,\int_{T > 0} \mathcal{\widetilde{H}}^2_n(T) \, \dd P_0(T) 
&= \frac{1}{n} \int_{T >0} \bigg( \sum_{i=1}^{n} \Gamma_m(\alpha) A_{\nu}( T, Y_i) \bigg)^2  \, \dd P_0(T) \\
&= \frac{\Gamma_m(\alpha)}{n} \sum_{i=1}^n \sum_{j=1}^n \int_{T > 0} A_{\nu}(T,Y_i) \, A_\nu(T,Y_j) \, (\det T)^\nu \, \etr(-T) \dd T.
\end{align*}
By (\ref{bessel_2matrixargument}) and Fubini's theorem, 
\begin{multline}
\label{helpful_eq1}
\int_{T > 0} A_{\nu}( T, Y_i) A_{\nu} (T, Y_j) (\det T)^{\nu}\etr(-T) \, \dd T \\
= \int_{O(m)} \int_{O(m)} \int_{T > 0} A_{\nu} (H T H' Y_i) A_{\nu}( K T K' Y_j) (\det T)^{\nu} \etr(-T) \, \dd T \dd H \dd K.
\end{multline}
Writing $A_{\nu}( H T H' Y_j)=A_{\nu}( H' Y_j H T)$, $j=1,\ldots,n$, and applying Herz's generalization (\ref{besselproductintegral}) of Weber's second exponential integral, we find that (\ref{helpful_eq1}) equals 
\begin{multline}
\label{helpful_eq2}
\int_{O(m)} \int_{O(m)} \etr(-H' Y_i H- K'Y_j K) A_{\nu}( -H' Y_i H K' Y_j K) \dd H \dd K \\
= \etr(-Y_i - Y_j) \int_{O(m)} \int_{O(m)} A_{\nu}( -H' Y_i H K' Y_j K) \dd H \dd K.
\end{multline}
On the right-hand side of \eqref{helpful_eq2}, we replace $H$ by $HK$ and apply the group invariance of the Haar measure and its normalization; then we find that \eqref{helpful_eq2} reduces to 
$$
\int_{O(m)} A_{\nu}( -H' Y_i HY_j) \dd H \equiv A_{\nu}(-,Y_i,Y_j).
$$
Therefore, 
$$
n \int_{T > 0} \mathcal{\widetilde{H}}^2_n(T) \, \dd P_0(T) = \frac{\Gamma_m(\alpha)}{n} \sum_{i=1}^n \sum_{j=1}^n \etr(-Y_i-Y_j ) A_{\nu} (-Y_i, Y_j).
$$

The second term to be calculated is 
\begin{multline*}
-\frac{2n}{\Gamma_m(\alpha)}  \int_{T > 0} \mathcal{\widetilde{H}}_n(T) \etr(-T/\alpha) \, \dd P_0(T) \\
= -2 \sum_{i=1}^n \int_{T > 0} A_{\nu} (T, Y_i) \, (\det T)^{\nu} \, \etr(-(I_m+\alpha^{-1} I_m)T) \dd T.
\end{multline*}
Similar to the previous calculation, we use (\ref{bessel_2matrixargument}) to express $A_{\nu} (T, Y_i)$ as an average over $O(m)$ and apply Fubini's theorem to reverse the order of integration. The resulting integral is a special case of \eqref{besselintegral}, so we conclude that 
\begin{align*}
-\frac{2n}{\Gamma_m(\alpha)} \int_{T > 0} & \mathcal{\widetilde{H}}_n(T) \etr(-T/\alpha) \, \dd P_0(T) \\
&= -2 \left(\frac{\alpha}{\alpha+1}\right)^{m \alpha}  \sum_{i=1}^n \etr \bigg(-\frac{\alpha}{\alpha+1}Y_i \bigg) \\
&\equiv -\frac{1}{n} \bigg( \frac{\alpha}{\alpha+1}\bigg)^{m \alpha} \sum_{i=1}^{n}\sum_{j=1}^{n} \bigg[ \etr\bigg(-\frac{\alpha}{\alpha+1}Y_i \bigg)+\etr\bigg(-\frac{\alpha}{\alpha+1}Y_j \bigg) \bigg].
\end{align*}

The third and last integral, which we evaluate using the gamma integral (\ref{gammaintegral}) is
\begin{align*}
n \int_{T > 0} { \etr(-2T/\alpha) }\, \dd P_0(T) 
&= n (\det(2\alpha^{-1} I_m +I_m))^{-\alpha} \\
&= n (2\alpha^{-1}+1)^{-m \alpha} 
\equiv \frac{1}{n} \sum_{i=1}^n \sum_{j=1}^n \bigg(\frac{2}{\alpha}+1 \bigg)^{-m \alpha}.
\end{align*}
Collecting together the three terms, we obtain the desired result. 
\end{proof}

\subsection{The limiting null distribution of the test statistic}
\label{sectionlimitingnull_matrixcase}

We denote by $L^2=L^2(P_0)$ the space of (equivalence classes of) orthogonally invariant Borel measurable functions $f: \mathcal{P}_{+}^{m \times m} \rightarrow \mathbb{C}$ that are square-integrable with respect to the probability measure $P_0$, i.e., for which $\int_{X > 0} {|f(X)|^2}\, \dd P_0(X) < \infty$. 
The space $L^2$ is a separable Hilbert space when equipped with the inner product
$$
\langle f, g \rangle_{L^2}=\int_{X > 0} {f(X) \ \overline{g(X)}}\, \dd P_0(X),
$$
and the corresponding norm
$$
||f||_{L^2}=\sqrt{\langle f, f \rangle_{L^2}},
$$
$f, g \in L^2$. Moreover, the set of normalized Laguerre polynomials $\{\mathcal{L}_\kappa^{(\nu)}\}$, with $\kappa$ ranging over all partitions, defined in Section \ref{besselLaguerre_matrixargument}, forms an orthonormal basis for the space $L^2$; see Herz \cite[p.~502, Theorem 4.6]{herz} and Constantine \cite[Section 3]{constantine}.

We now define the stochastic process
 \begin{equation}
 \label{zn_matrixargument}
 \mathcal{Z}_{n}(T)=\frac{1}{\sqrt{n}} \sum_{j=1}^n \bigg[ \Gamma_m(\alpha) A_{\nu}(T, Y_j) - \etr(-T/\alpha) \bigg], 
 \end{equation}
$T > 0$. We view the random field $\mathcal{Z}_{n}:=\{ \mathcal{Z}_{n}(T), T > 0 \}$ as a random element in $L^2$ since, as we now show, its sample paths are in $L^2$.  

\begin{lemma}
\label{lemmazn_matrixargument}
The test statistic (\ref{statistic_wishart}) can be written as 
$$
\boldsymbol{T}^2_{n}=\int_{T >0} { \big( \mathcal{Z}_{n}(T) \big)^2}\, \dd P_0(T)=||\mathcal{Z}_n||^2_{L^2}.
$$
In particular, $||\mathcal{Z}_n||^2_{L^2} \ < \infty$. 
\end{lemma}

This result follows immediately from \eqref{empiricalhankel_matrixargument}, \eqref{statistic_wishart}, and \eqref{zn_matrixargument}.  

\begin{remark}{\rm 
By \cite[Example 1.4]{guptarichards} $(Y_1, \dotsc, Y_n)$ has a matrix Liouville distribution, of the second kind, that does not depend on $\Sigma$. Therefore, without loss of generality, we will set $\Sigma=I_m$ in deriving the limiting null distribution of $\boldsymbol{T}^2_n$. 

We also note that, for each $j=1,\ldots,n$, the matrices $Y_j=\bar{X}_{n}^{-1/2} X_j \bar{X}_{n}^{-1/2}$ and $Z_j=X_{j}^{1/2} \bar{X}_n^{-1} X_{j}^{1/2}$ have the same spectrum; this result is proved by verifying that $Y_j$ and $Z_j$ have the same characteristic polynomial.  Consequently, 
\begin{equation}
\label{remark_sameeigen}
A_{\nu}(T, Y_j)=A_{\nu} (T, Z_j),
\end{equation}
$j=1, \dotsc, n$, so we can replace $Y_j$ by $Z_j$ in the definition (\ref{empiricalhankel_matrixargument}) of the test statistic. 
}\end{remark}

We now state the main result of this section. 

\begin{theorem}
\label{limitingnulldistribution_matrixcase}
Let $m \ge 2$ and $X_1, \dotsc, X_n$ be i.i.d. $P_0$-distributed random matrices, where $\alpha > \max\{ \tfrac12(2m-1), \tfrac12(m+3) \}$, and let $\mathcal{Z}_{n}:=(\mathcal{Z}_{n}(T), T > 0)$ be the random field defined in (\ref{zn_matrixargument}). Then, there exists a centered Gaussian field $\mathcal{Z}:=(\mathcal{Z}(T), T > 0)$, with sample paths in $L^2$ and with covariance function, 
\begin{equation}
\label{covariancefn_matrixcase}
K(S,T) = \etr(-\alpha^{-1}(S+T)) 
\bigg[\Gamma_m(\alpha) A_{\nu}(-\alpha^{-2}S,T) - \frac{1}{\alpha^3 m} (\tr S)(\tr T) - 1\bigg],
\end{equation}
$S, T > 0$, such that $\mathcal{Z}_{n} \xrightarrow{d} \mathcal{Z}$ in $L^2$ as $n \rightarrow \infty$.  Moreover,
$$
\boldsymbol{T}^2_{n} \xrightarrow{d}\int_{T >0} {\mathcal{Z}^2(T)}\, \dd P_0(T).
$$
\end{theorem}

\medskip

The remainder of this section is devoted to proving Theorem \ref{limitingnulldistribution_matrixcase}, so readers who wish to postpone reading the detailed derivation may continue directly to Section \ref{sec:eigen_matrix}.

\subsubsection{Preliminary details}
\label{prelimforasympdistn}

Here, we provide details on the Frobenius norm of a matrix, the Taylor expansion of functions on the space $\mathcal{S}^{m \times m}$ of symmetric matrices, and various preliminary lemmata necessary for the derivation of the asymptotic distribution of $\boldsymbol{T}^2_n$.  

For $X, Y \in \mathcal{S}^{m \times m}$, the {\it inner product} between $X$ and $Y$ is defined by
$\langle X, Y \rangle = \tr(XY),$ 
and the {\it Frobenius norm} of $X$ is defined by
$\lVert X \rVert^2_F = \langle X, X \rangle = \tr(X^2).$ 
By \cite[Section 5.6, p.~291]{hornjohnson}, the Frobenius norm satisfies the triangle inequality, 
$\lVert X + Y \rVert_F \le \lVert X \rVert_F + \lVert Y \rVert_F,$ 
and moreover, it is sub-multiplicative, 
$\lVert XY \rVert_F \le \lVert X \rVert_F \cdot \lVert Y \rVert_F.$ 

We use the usual notation for Kronecker's delta, viz., $\delta_{ij} = 1$ or $0$ for $i=j$ or $i\neq j$, respectively.  For $Z=(z_{ij}) \in \mathcal{S}^{m \times m}$, the {\it gradient operator} is the $m \times m$ matrix 
$$
\nabla_Z = \left(\tfrac12 (1+\delta_{i j}) \frac{\partial}{\partial z_{i j}}\right)_{i,j=1,\ldots,m}.
$$
For example, is straightforward to see that 
$\nabla_Z e^{\langle T, Z \rangle} = e^{\langle T, Z \rangle} \, T.$ 

Let $F : \mathcal{S}^{m \times m} \rightarrow \mathbb{C}$ be a $C^1$ function; that is, $F$ is differentiable of order $1$ and its partial derivatives are continuous on $\mathcal{S}^{m \times m}$.  The {\it Taylor expansion of order $1$} of the function $F$, at $Z_0 \in \mathcal{S}^{m \times m}$, is 
\begin{equation}
\label{taylorexpansion1}
F(Z) = F(Z_0) + \langle Z-Z_0, \nabla_U F(U) \rangle,
\end{equation}
where $U= t Z +(1-t) Z_0$, for some $t \in [0,1]$.

\smallskip
\begin{lemma}
For $T, Z > 0$, 
\begin{equation}
\label{gradient_bessel}
\nabla_Z A_{\nu}(T, Z) = \int_{O(m)}  M^{1/2} \, \nabla_Y A_{\nu}(Y) \, M^{1/2} \, \dd H ,
\end{equation}
where $M:= H T H'$ and $Y:=M^{1/2} Z M^{1/2}$.
\end{lemma}

\begin{proof}
By (\ref{bessel_2matrixargument}), 
$$
A_{\nu}(T, Z)=\int_{O(m)} A_{\nu}(H T H' Z) \, \dd H .
$$
It is straightforward to verify that the conditions given by Burkill and Burkill \cite[p.~289,~Theorem 8.72]{burkill} for interchanging derivatives and integrals are satisfied; therefore, 
\begin{equation}
\label{nabla_bessel_2matrixargument}
\nabla_Z A_{\nu}(T, Z) = \int_{O(m)} \nabla_Z A_{\nu}(H T H' Z) \, \dd H .
\end{equation}
Setting $M = H T H'$ and $Y = M^{1/2} Z M^{1/2}$, we have $Z = M^{-1/2} Y M^{-1/2}$.  By Maass \cite[p.~64]{maass}, $\nabla_Z = M^{1/2} \nabla_{Y} M^{1/2}$; therefore,
\begin{align}
\label{operator_xy}
\nabla_{Z} A_{\nu}(M Z)&=\nabla_{Z} A_{\nu}(M^{1/2} Z M^{1/2})\nonumber\\
&= M^{1/2} \nabla_{Y} M^{1/2} A_{\nu}(Y)
= M^{1/2} \, \nabla_{Y} A_{\nu}(Y) \, M^{1/2},
\end{align}
since $A_{\nu}(Y)$ is scalar-valued. Combining (\ref{nabla_bessel_2matrixargument}) and (\ref{operator_xy}), we obtain \eqref{gradient_bessel}.
\end{proof} 

We note that all further interchanges of derivatives and integrals are justifiable by appeal to \cite[loc. cit.]{burkill}, so we will perform such interchanges without further citation.  Also, various positive constants arise in the following calculations, and we will denote them generically by $c, c_j, C_j$, $j \ge 1$.

\begin{lemma} 
\label{lemma_bound__nabla_trace_squareroot1}
Let $Q$ be an $m \times m$ matrix such that $0 < QQ' < I_m$. Also, let $Y$ be an $m \times m$ positive-definite matrix. Then, there exists a constant $c > 0$ such that
\begin{equation}
\label{bound__nabla_trace_squareroot1}
\lVert \nabla_Y (\tr Q Y^{1/2}) \rVert _F \le 
c \, (\lambda_\min (Y))^{-1/2}.
\end{equation}
\end{lemma}

\Pro
Since the trace is a linear operator, we have 
\begin{align*}
\nabla_Y (\tr Q Y^{1/2})&=\left(\tfrac12 (1+\delta_{i j}) \frac{\partial}{\partial y_{i j}} \tr Q Y^{1/2}\right)\\
&= \left(\tr Q \bigg(\tfrac12 (1+\delta_{i j}) \frac{\partial}{\partial y_{i j}} Y^{1/2} \bigg) \right) 
= \left(\tr \big[ Q \big(\nabla_Y \otimes Y^{1/2} \big)_{i j} \big] \right),
\end{align*}
where $\nabla_Y \otimes Y^{1/2}$ is the Kronecker product of the gradient $\nabla_Y$ acting on the matrix $Y^{1/2}$, and $V_{ij}:= \big(\nabla_Y \otimes Y^{1/2} \big)_{ij}$ is the $(i,j)$th block matrix in that Kronecker product.  

By the Cauchy-Schwarz inequality, and the fact that $QQ' < I_m$ implies $\tr(Q Q') \le m$, we obtain
\begin{align}
\label{bound_nabla_trace_squareroot2}
\lVert \nabla_Y (\tr Q Y^{1/2}) \rVert ^2_F &= \sum_{i}\sum_{j} \bigg[ \tr(Q V_{ij} ) \bigg]^2\nonumber\\
& \le  \sum_{i}\sum_{j} \tr(Q Q') \tr (V^2_{ij})\nonumber\\
& \le  m \sum_{i}\sum_{j}  \tr (V^2_{ij}) 
=m \lVert (\nabla_Y \otimes Y^{1/2}) \rVert^2_F.
\end{align}
Recall from \cite[p.~13]{bishopmoralniclas} the \textit{multi-linear operator norm}, $\vertiii{\cdot}$, which we define here in the following context: If $K_{ij}$ denotes the $(i,j)$th element of a $m \times m$ matrix $K$ and $(V_{ij})_{kl}$ denotes the $(k,l)$th element of $V_{ij}:= \big(\nabla_Y \otimes Y^{1/2} \big)_{ij}$, the $(i,j)$th block in the tensor product $\nabla_Y \otimes Y^{1/2}$, then 
$$
((\nabla_Y \otimes Y^{1/2}) \cdot K)_{kl} = \sum_{i}\sum_{j} K_{ij} \, (V_{ij})_{kl},
$$
and we define 
$$
\vertiii{\nabla_Y \otimes Y^{1/2}} := \sup_{\lVert K \rVert_F=1} \lVert (\nabla_Y \otimes Y^{1/2}) \cdot K \rVert_F .
$$

Since all norms on a finite-dimensional space are equivalent, there exists a constant $c > 0$ such that
$
\lVert \nabla_Y \otimes Y^{1/2} \rVert_F \le  2c \, \vertiii{\nabla_Y \otimes Y^{1/2}} 
$.
By \cite[p.~262, Eq. (6)]{moralniclas}, there holds the crucial inequality, 
$$\vertiii{ \nabla_Y \otimes Y^{1/2} } \le 2^{-1} (\lambda_\min (Y))^{-1/2}.$$  Hence, 
$$
\lVert (\nabla_Y \otimes Y^{1/2}) \rVert_F \le c \, (\lambda_\min (Y))^{-1/2},
$$ 
so we obtain  
\begin{equation}
\label{bound_norm_squareroot}
\lVert (\nabla_Y \otimes Y^{1/2}) \rVert^2_F =\sum_{i} \sum_{j} \tr (V^2_{ij}) \le c^2 \, (\lambda_\min (Y))^{-1}.
\end{equation}
Combining (\ref{bound_nabla_trace_squareroot2}) and (\ref{bound_norm_squareroot}), we obtain \eqref{bound__nabla_trace_squareroot1}.
$\qed$

\smallskip

\begin{lemma}
\label{lemma_norm_bound_nabla_bessel}
For $T, Z > 0$, there exists a constant $C > 0$ such that 
\begin{equation}
\label{norm_bound_nabla_bessel_2matrixargument}
\lVert \nabla_Z A_{\nu}(T, Z) \rVert_F \le C \, \lVert T \rVert_F \ (\lambda_\min (T))^{-1/2} (\lambda_\min (Z))^{-1/2}.
\end{equation}
\end{lemma}

\begin{proof}
By Eq. (\ref{gradient_bessel}),
$$
\nabla_Z A_{\nu}(T, Z) = \int_{O(m)} M^{1/2} \nabla_Z A_{\nu}(Y) M^{1/2} \, \dd H ,
$$
where $M:= H T H'$ and $Y:=M^{1/2} Z M^{1/2}$.  By Minkowski's inequality for integrals, 
\begin{align}
\label{nabla_bessel_bound_2matrix}
\lVert \nabla_Z A_{\nu}(T, Z) \rVert_F & \le \int_{O(m)}  { \rVert M^{1/2} [\nabla_Y A_{\nu}(Y) ] M^{1/2} \lVert_F}\, \, \dd H  \nonumber \\
&= \int_{O(m)}  { \rVert M \ \nabla_Y A_{\nu}(Y) \lVert_F}\, \, \dd H \nonumber \\
&\le \int_{O(m)}  { \lVert M\rVert_F \cdot \lVert \nabla_Y A_{\nu}(Y) \rVert_F}\, \, \dd H,
\end{align}
since the Frobenius norm is sub-multiplicative.

By Herz's generalization, (\ref{besselintegraldef_matrixargument}), of the Poisson integral, 
$$
A_{\nu}(Y) = 
c_1 \int_{Q'Q < I_m} { \etr(2 \, \mi Y^{1/2} Q) (\det(I_m - Q'Q))^{\alpha-\tfrac12(2m+1)}}\, \dd Q,
$$
where $c_1 > 0$.  Therefore, 
$$
\nabla_Y A_{\nu}(Y) = 2 \, \mi c_1 
\int_{Q'Q < I_m} { \etr(2\, \mi Y^{1/2} Q) (\det(I_m - Q'Q))^{\alpha-\tfrac12(2m+1)} \ \nabla_Y (\tr Q Y^{1/2}) }\, \dd Q.
$$
Applying Minkowski's inequality and then using (\ref{bound__nabla_trace_squareroot1}) to bound the integrand, we obtain  
\begin{align}
\lVert \nabla_Y A_{\nu}(Y) \rVert_F &\le 
2 c_1 \int_{Q'Q < I_m} (\det(I_m - Q'Q))^{\alpha-\tfrac12(2m+1)} \ \lVert \nabla_Y (\tr Q Y^{1/2}) \rVert_F \, \dd Q \nonumber \\
\label{bound_nabla_1matrixargument}
&\le 2 c_1 c (\lambda_\min (Y))^{-1/2}
\int_{Q'Q < I_m} (\det(I_m - Q'Q))^{\alpha-\tfrac12(2m+1)} \, \dd Q \nonumber \\
&= 
C (\lambda_\min (Y))^{-1/2}.
\end{align}
Combining (\ref{nabla_bessel_bound_2matrix}) and (\ref{bound_nabla_1matrixargument}), we obtain
$$
\lVert \nabla_Z A_{\nu}(T, Z) \rVert_F \le 
C \int_{O(m)} \lVert M\rVert_F \cdot (\lambda_\min (Y))^{-1/2} \, \dd H .
$$
For $H \in O(m)$, 
$
\lVert M \rVert_F = \lVert HTH' \rVert_F = \lVert T \rVert_F
$ 
and 
\begin{align*}
\lambda_\min(Y) &= \lambda_\min(M^{1/2} Z M^{1/2}) = \lambda_\min(MZ) \\
&= \lambda_\min(HTH'Z) \ge \lambda_\min (H T H') \lambda_\min (Z) = \lambda_\min (T) \lambda_\min (Z).
\end{align*}
Hence, 
\begin{align*}
\lVert \nabla_Z A_{\nu}(T,Z) \rVert_F & \le 
C \, \lVert T \rVert_F (\lambda_\min(T))^{-1/2} \, (\lambda_\min(Z))^{-1/2} \, \int_{O(m)} \dd H  \nonumber\\
&= C \, \lVert T \rVert_F (\lambda_\min(T))^{-1/2} \, (\lambda_\min(Z))^{-1/2},
\end{align*}
which completes the proof.
\end{proof}

\smallskip

\begin{lemma}
\label{lemma_norm_bound_nabladiff_bessel_2matrixarg}
For $T, Z_1, Z_2 > 0$, there exist constants $C_1, C_2 >0$ such that
\begin{multline}
\label{norm_bound_nabladiff_bessel_2matrixargument}
\bigg\lVert \nabla_{Z_1} A_{\nu}(T, Z_1) - \nabla_{Z_2} A_{\nu}(T, Z_2) \bigg\rVert_F \\
\le \frac{\rVert Z_1-Z_2 \lVert^{1/2}_F \ \lVert T \rVert^{3/2}_F}{\lambda_{\min}(Z_2^{1/2})}
\left[\frac{C_1}{\lambda_{\min}(T) \lambda_{\min}(Z_1^{1/2})} + \frac{C_2}{\lambda_{\min}(T^{1/2})} \right].
\end{multline}
\end{lemma}

\begin{proof}
By (\ref{gradient_bessel}),
\begin{equation}
\label{nabla_diffe_twoargument}
\nabla_{Z_1} A_{\nu}(T, Z_1) - \nabla_{Z_2} A_{\nu}(T, Z_2) 
= \int_{O(m)} H T^{1/2} H' \bigg[ \nabla_{Y_1} A_{\nu}(Y_1) -  \nabla_{Y_2} A_{\nu}(Y_2)\bigg]  H T^{1/2} H'  \dd H ,
\end{equation}
where $Y_j:= M^{1/2} Z_j M^{1/2}$, $j=1,2$, and $M:=HT H'$. 
Applying \eqref{besselintegraldef_matrixargument} and interchanging derivatives and integrals, 
we obtain 
\begin{align*}
&\nabla_{Y_1} A_{\nu}(Y_1) - \nabla_{Y_2} A_{\nu}(Y_2) \\
&\ \ = 2 \, \mi c_1 \int_{Q'Q < I_m} \bigg[\etr(2 \, \mi Y_1^{1/2} Q) \nabla_{Y_1} (\tr Q Y_1^{1/2})- \etr(2 \, \mi Y_2^{1/2} Q) \nabla_{Y_2} (\tr Q Y_2^{1/2}) \bigg] \, \dd\mu(Q),
\end{align*}
where $\dd\mu(Q):= (\det(I_m - Q'Q))^{\alpha-\tfrac12(2m+1)} \dd Q$.  Therefore, 
\begin{align*}
\bigg\lVert & \nabla_{Y_1} A_{\nu}(Y_1) - \nabla_{Y_2} A_{\nu}(Y_2) \bigg\rVert_F \\
&\le 2 c_1 \int_{Q'Q < I_m} { \bigg\lVert \etr(2 \, \mi Y_1^{1/2} Q) \nabla_{Y_1} (\tr Q Y_1^{1/2})- \etr(2 \, \mi Y_2^{1/2} Q) \nabla_{Y_2} (\tr Q Y_2^{1/2}) \bigg\rVert_F}\, \dd \mu(Q).
\end{align*}

Let $\theta_j := 2\tr(Y_j^{1/2}Q)$ 
and $N_j:= \nabla_{Y_j} (\tr Q Y_j^{1/2})$, $j=1,2$; then we observe that
\begin{align*}
\lVert e^{\mi \theta_1} N_1 - e^{\mi \theta_2} N_2 \rVert_F 
&= \lVert e^{\mi \theta_1} (N_1 - N_2) + (e^{\mi \theta_1} - e^{\mi \theta_2}) N_2 \rVert_F\\
&\le \lVert N_1 - N_2 \rVert_F + |e^{\mi \theta_1} - e^{\mi \theta_2}| \cdot \lVert N_2 \rVert_F,
\end{align*}
since $|e^{\mi \theta_1}| = 1$. 
Also, 
using the identity 
\begin{align*}
|e^{\mi \theta_1}-e^{\mi \theta_2}|^2 
= 4 \sin^2 \left(\tfrac12(\theta_1-\theta_2)\right),
\end{align*}
we find that  
\begin{align}
\label{bounddiff}
\bigg\lVert \nabla_{Y_1} & A_{\nu}(Y_1) - \nabla_{Y_2} A_{\nu}(Y_2) \bigg\rVert_F \nonumber\\
&\le 2 c_1 \int_{Q'Q < I_m} { \bigg[ \ \bigg\lVert \nabla_{Y_1} (\tr Q Y_1^{1/2}) - \nabla_{Y_2} (\tr Q Y_2^{1/2}) \bigg\rVert_F} \nonumber\\
& \qquad\qquad\qquad\quad + {2 \, |\sin\big(\tr(Y_1^{1/2}-Y_2^{1/2})Q\big)| \, \cdot \, \lVert \nabla_{Y_2} (\tr Q Y_2^{1/2}) \rVert_F \bigg] }\, \dd \mu(Q).
\end{align}
By applying the same argument as in Lemma \ref{lemma_bound__nabla_trace_squareroot1}, we obtain 
$$
\nabla_{Y_1} (\tr Q Y_1^{1/2}) - \nabla_{Y_2} (\tr Q Y_2^{1/2}) =
\left(\tr \big[ Q \big(\nabla_{Y_1} \otimes Y_1^{1/2} \big)_{ij} \big] \right)- \left(\tr \big[ Q \big(\nabla_{Y_2} \otimes Y_2^{1/2} \big)_{ij} \big] \right);
$$
so, by the Cauchy-Schwarz inequality and the fact that $Q'Q < I_m$ implies $\tr(Q Q') \le m$, we obtain
\begin{align}
\bigg \lVert \nabla_{Y_1} (\tr Q & Y_1^{1/2}) - \nabla_{Y_2} (\tr Q Y_2^{1/2}) \bigg\rVert ^2_F \nonumber \\
&= \sum_{i}\sum_{j} \bigg( \tr(Q [(\nabla_{Y_1} \otimes Y_1^{1/2})_{ij} - (\nabla_{Y_2} \otimes Y_2^{1/2})_{ij} ]) \bigg)^2 \nonumber\\
& \le  \sum_{i}\sum_{j} \tr(Q Q') \tr \bigg(\nabla_{Y_1} \otimes Y_1^{1/2})_{ij}-(\nabla_{Y_2} \otimes Y_2^{1/2})_{ij} \bigg)^2 \nonumber\\
& \le  m \sum_{i}\sum_{j}  \tr \bigg((\nabla_{Y_1} \otimes Y_1^{1/2})_{ij}-(\nabla_{Y_2} \otimes Y_2^{1/2})_{ij} \bigg)^2 \nonumber\\
\label{boundtrace_diff}
&= m \bigg\lVert (\nabla_{Y_1} \otimes Y_1^{1/2})-(\nabla_{Y_2} \otimes Y_2^{1/2}) \bigg\rVert_F^2.
\end{align}
Since the norms $\|\cdot\|_F$ and $\vertiii{\cdot}$ are equivalent, there exists $c > 0$ such that
\begin{align}
\label{frobenius_operator_norm}
\bigg\lVert (\nabla_{Y_1} \otimes Y_1^{1/2}) - (\nabla_{Y_2} \otimes Y_2^{1/2}) \bigg\rVert_F 
& \le c \vertiii{ (\nabla_{Y_1} \otimes Y_1^{1/2})-(\nabla_{Y_2} \otimes Y_2^{1/2})} \nonumber \\
& \equiv c \sup_{\| K \|_F=1} \left\| \bigg( (\nabla_{Y_1} \otimes Y_1^{1/2})-(\nabla_{Y_2} \otimes Y_2^{1/2}) \bigg) \cdot K \right\|_F .
\end{align}
By a result of Del Moral and Niclas \cite[Theorem 1.1, Eq. (4)]{moralniclas}, 
\begin{multline*}
\left((\nabla_{Y_1} \otimes Y_1^{1/2})-(\nabla_{Y_2} \otimes Y_2^{1/2})\right) \cdot K \\
= \int^{\infty}_0 \bigg[\exp(-tY_1^{1/2}) K \exp(-tY_1^{1/2})-\exp(-t Y_2^{1/2}) K \exp(-t Y_2^{1/2}) \bigg]\, \dd t,
\end{multline*}
where 
$
\exp(M) = \sum_{j=0}^\infty M^j/j!
$ 
is the matrix exponential function.  Therefore, 
\begin{multline*}
\bigg\lVert \left( (\nabla_{Y_1} \otimes Y_1^{1/2})-(\nabla_{Y_2} \otimes Y_2^{1/2})\right) \cdot K \bigg\rVert_F \\
\le \int^{\infty}_0 { \bigg\lVert \exp(-tY_1^{1/2}) K \exp(-tY_1^{1/2}) - \exp(-t Y_2^{1/2}) K \exp(-t Y_2^{1/2}) \bigg\rVert_F}\, \dd t.
\end{multline*}
For any $m \times m$ matrices $M_1$ and $M_2$, and for any $K$ such that $\lVert K \rVert_F=1$, 
\begin{align*}
\big\lVert M_1KM_1 - M_2KM_2 \big\lVert_F 
&= \big\lVert M_1K(M_1-M_2) + (M_1-M_2)KM_2 \big\lVert_F \\
&\le \big\lVert M_1\big\lVert_F  \big\lVert M_1-M_2\big\lVert_F  + \big\lVert M_1-M_2\big\lVert_F \big\lVert M_2 \big\lVert_F \\
&= \big(\big\lVert M_1\big\lVert_F + \big\lVert M_2 \big\lVert_F\big) \ \big\lVert M_1-M_2\big\lVert_F.
\end{align*}
Now setting $M_j = \exp(-tY_j^{1/2})$, $j=1,2$, we obtain 
\begin{multline*}
\bigg\lVert \exp(-tY_1^{1/2})  K \exp(-tY_1^{1/2})- \exp(-t Y_2^{1/2}) K \exp(-t  Y_2^{1/2}) \bigg\rVert_F \\
\le \bigg(\bigg\lVert \exp(-tY_1^{1/2}) \bigg\rVert_F+ \bigg\lVert \exp(-tY_2^{1/2}) \bigg\rVert_F \bigg) \ \bigg\lVert \exp(-tY_1^{1/2}) -\exp(-tY_2^{1/2}) \bigg\rVert_F.
\end{multline*}
Therefore, 
\begin{multline}
\label{grady1y2}
\bigg\lVert \bigg( (\nabla_{Y_1} \otimes Y_1^{1/2}) - (\nabla_{Y_2} \otimes Y_2^{1/2})\bigg) \cdot K \bigg\rVert_F \\
\le \int^{\infty}_0 \bigg(\bigg\lVert \exp(-tY_1^{1/2}) \bigg\rVert_F+ \bigg\lVert \exp(-tY_2^{1/2}) \bigg\rVert_F \bigg) 
\bigg\lVert \exp(-tY_1^{1/2}) -\exp(-tY_2^{1/2}) \bigg\rVert_F \, \dd t.
\end{multline}

For any $m \times m$ positive-definite matrix $Y$ and for $t \ge 0$,
$$
\tr(\exp(-2tY)) = \sum_{i=1}^m \exp(-2t\lambda_i(Y)) \le m \exp(-2t\lambda_{\min}(Y));
$$
hence, for $t \ge 0$, and $j=1,2$, 
\begin{equation}
\label{norm_exp_bound}
\bigg\lVert \exp(-tY_j^{1/2})\bigg\rVert_F =\bigg[ \tr(\exp(-2tY_j^{1/2}))\bigg] ^{1/2} \le m^{1/2}\exp(-t\lambda_{\min}(Y_j^{1/2})).
\end{equation}
Therefore, for $\lVert K \rVert_F=1$, the right-hand side of \eqref{grady1y2} is bounded above by 
\begin{align*}
m^{1/2} \int^{\infty}_0 {\hskip -1pt} \bigg[\exp(-t\lambda_{\min}(Y_1^{1/2}))+ \exp(-t \lambda_{\min}(Y_2^{1/2})) \bigg] 
\bigg\lVert \exp(-tY_1^{1/2}) - \exp(-tY_2^{1/2}) \bigg\rVert_F \dd t.
\end{align*}

Define $X(t):=\exp(-tY_1^{1/2})$, $Y(t):= \exp(-t Y_2^{1/2})$, and $\psi(t):= X(t)-Y(t)$, $t \ge 0$.  Notice that
$$
X'(t) = -Y_1^{1/2} \exp(-tY_1^{1/2}) = -Y_1^{1/2} X(t)
$$
and
$$
Y'(t) = -Y_2^{1/2} \exp(-t Y_2^{1/2}) = -Y_2^{1/2} Y(t),
$$
with $X(0) = Y(0) = I_m$.  Then $\psi(t)$ satisfies the inhomogeneous differential equation 
\begin{align*}
\psi'(t) &= -Y_1^{1/2} X(t) + Y_2^{1/2} Y(t) \\
&= -Y_2^{1/2} \psi(t) - (Y_1^{1/2}-Y_2^{1/2}) X(t),
\end{align*}
with boundary condition $\psi(0)=0$.  
By following the approach of K{\aa}gstr{\"o}m \cite[Section 4]{kagstrom}, we find that the solution of this differential equation is
\begin{align*}
\psi(t) 
&= -\int^{t}_0 \exp(-(t-s) Y_2^{1/2}) (Y_1^{1/2}-Y_2^{1/2})\exp(-sY_1^{1/2}) \, \dd s.
\end{align*}
By Minkowski's inequality and the sub-multiplicative property of the Frobenius norm, 
\begin{align*}
\lVert \psi(t) \rVert_F 
& \le  \int^{t}_0{ \bigg\rVert \exp(-(t-s) Y_2^{1/2}) \bigg\lVert_F \cdot \bigg\lVert Y_1^{1/2}-Y_2^{1/2} \bigg\rVert_F \cdot \bigg\lVert \exp(-sY_1^{1/2})  \bigg\rVert_F}\, \dd s.
\end{align*}
Using \eqref{norm_exp_bound} to bound both exponential terms in this integrand, we find that 
$$
\lVert \psi(t) \rVert_F \le m \ \bigg\lVert Y_1^{1/2}-Y_2^{1/2} \bigg\rVert_F  \int^{t}_0{ \exp(-(t-s) \lambda_{\min}(Y_2^{1/2})) \exp(-s \lambda_{\min}(Y_1^{1/2}))}\, \dd s.
$$
Assuming that $\lambda_{\min}(Y_1^{1/2}) \neq \lambda_{\min}(Y_2^{1/2})$, we calculate the latter integral, obtaining 
\begin{align}
\label{bound_k}
\psi(t) &= m \ \bigg\lVert Y_1^{1/2}-Y_2^{1/2} \bigg\rVert_F \cdot \frac{ \exp(-t \lambda_{\min}(Y_1^{1/2})) -\exp(-t \lambda_{\min}(Y_2^{1/2}))}{\lambda_{\min}(Y_2^{1/2}) -\lambda_{\min}(Y_1^{1/2})}.
\end{align}
Combining (\ref{boundtrace_diff})-(\ref{bound_k}), we obtain 
\begin{align*}
\bigg\lVert & \nabla_{Y_1} (\tr Q Y_1^{1/2}) - \nabla_{Y_2} (\tr Q Y_2^{1/2}) \bigg\rVert_F \\
& \le c_2 \, \frac{\bigg\lVert Y_1^{1/2}-Y_2^{1/2} \bigg\rVert_F}{\lambda_{\min}(Y_2^{1/2}) -\lambda_{\min}(Y_1^{1/2})} 
\int^{\infty}_0 {\bigg[ \exp(-2t \lambda_{\min}(Y_1^{1/2})) - \exp(-2t \lambda_{\min}(Y_2^{1/2})) \bigg]}\, \dd t \\
&= c_2 \, \frac{\bigg\lVert Y_1^{1/2}-Y_2^{1/2} \bigg\rVert_F}{\lambda_{\min}(Y_1^{1/2}) \lambda_{\min}(Y_2^{1/2})}.
\end{align*}
By continuity, this result remains valid for $\lambda_{\min}(Y_1^{1/2}) = \lambda_{\min}(Y_2^{1/2})$.

Next, it follows from (\ref{bounddiff}) that 
\begin{align*}
\bigg\lVert \nabla_{Y_1} A_{\nu}(Y_1) & - \nabla_{Y_2} A_{\nu}(Y_2) \bigg\rVert_F \\
& \le 
c_3 \frac{\bigg\lVert Y_1^{1/2}-Y_2^{1/2} \bigg\rVert_F}{\lambda_{\min}(Y_1^{1/2}) \lambda_{\min}(Y_2^{1/2})} \\
& \quad + c_4 \int_{Q'Q < I_m} {|\sin\big(\tr(Y_1^{1/2}-Y_2^{1/2})Q\big)| \, \cdot \, \lVert \nabla_{Y_2} (\tr Q Y_2^{1/2}) \rVert_F }\, \dd \mu(Q). 
\end{align*}
By the Cauchy-Schwarz inequality,
\begin{align*}
|\sin\big(\tr(Y_1^{1/2}-Y_2^{1/2})Q\big)| &\le   | \tr(Y_1^{1/2}-Y_2^{1/2})Q| \\
& \le  \bigg\lVert Y_1^{1/2}-Y_2^{1/2} \bigg\rVert_F \cdot ( \tr(Q Q')) ^{1/2} 
\le  m^{1/2} \ \bigg\lVert Y_1^{1/2}-Y_2^{1/2} \bigg\rVert_F,
\end{align*}
and by (\ref{bound__nabla_trace_squareroot1}), 
$$
\lVert \nabla_{Y_2} (\tr Q Y_2^{1/2}) \rVert_F \le c \, (\lambda_\min (Y_2))^{-1/2}.
$$
Therefore, with $c_5 = m^{1/2} c_4 c \, \int_{Q'Q < I_m} {}\, \dd \mu(Q) < \infty$, we have derived 
\begin{align}
\bigg\lVert \nabla_{Y_1} A_{\nu}(Y_1) & - \nabla_{Y_2} A_{\nu}(Y_2) \bigg\rVert_F \nonumber \\
& \le c_3 \, \frac{\bigg\lVert Y_1^{1/2}-Y_2^{1/2} \bigg\rVert_F}{\lambda_{\min}(Y_1^{1/2}) \lambda_{\min}(Y_2^{1/2})}  + c_5 \bigg\lVert Y_1^{1/2}-Y_2^{1/2} \bigg\rVert_F \ (\lambda_\min (Y_2))^{-1/2} \nonumber \\
\label{nabla_diff_bessel_1matrixargument}
&= \bigg\lVert Y_1^{1/2}-Y_2^{1/2} \bigg\rVert_F \bigg[\frac{c_3}{\lambda_{\min}(Y_1^{1/2}) \lambda_{\min}(Y_2^{1/2})} + \frac{c_5}{\lambda_{\min}(Y_2^{1/2})} \bigg]. 
\end{align}
By (\ref{nabla_diffe_twoargument}), Minkowski's inequality, and the sub-multiplicative property of the Frobenius norm, we obtain
\begin{align*}
\bigg\lVert \nabla_{Z_1} A_{\nu} (T, Z_1) & - \nabla{Z_2} A_{\nu}(T, Z_2) \bigg\rVert_F \\
& \le \int_{O(m)} { \bigg\lVert H T^{1/2} H' \bigg[ \nabla_{Y_1} A_{\nu}(Y_1){\hskip-1.22pt} - {\hskip-1.22pt} \nabla_{Y_2} A_{\nu}(Y_2)\bigg]  H T^{1/2} H' \bigg\rVert_F} \dd H \\
& = \int_{O(m)} { \bigg\lVert H T H'  \bigg[ \nabla_{Y_1} A_{\nu}(Y_1) -  \nabla_{Y_2} A_{\nu}(Y_2)\bigg] \bigg\rVert_F }\, \, \dd H \\
& \le \int_{O(m)} { \lVert T \rVert_F \cdot \bigg \lVert \nabla_{Y_1} A_{\nu}(Y_1) -  \nabla_{Y_2} A_{\nu}(Y_2) \bigg\rVert_F}\, \, \dd H .
\end{align*}
Applying the bound (\ref{nabla_diff_bessel_1matrixargument}), we find that 
\begin{multline}
\bigg\lVert \nabla_{Z_1} A_{\nu} (T, Z_1) - \nabla_{Z_2} A_{\nu}(T, Z_2) \bigg\rVert_F \\
\label{nabla_bound_diff_2matrixargument_final}
\le \int_{O(m)} { \lVert T \rVert_F \cdot \bigg\lVert Y_1^{1/2}-Y_2^{1/2} \bigg\rVert_F \bigg[ \frac{c_3}{ \lambda_{\min}(Y_1^{1/2}) \lambda_{\min}(Y_2^{1/2})} + \frac{c_5}{ \lambda_{\min}(Y_2^{1/2})} \Bigg]} \, \dd H .
\end{multline}
By a result of Wihler \cite[Eq.~(3.2)]{wihler},
\begin{equation}
\label{wihler}
\bigg\lVert Y_1^{1/2}-Y_2^{1/2} \bigg\rVert_F \le  m^{1/4}  \ \lVert Y_1 -Y_2 \rVert^{1/2}_F.
\end{equation}
Since $M = H T H'$, \ $Y_1 = M^{1/2} Z_1 M^{1/2}$, and $Y_2 = M^{1/2} Z_2 M^{1/2}$, then we have
\begin{align}
\lVert Y_1 -Y_2 \rVert_F^{1/2} &=\bigg\lVert H T^{1/2} H' (Z_1-Z_2) H T^{1/2} H' \bigg\rVert^{1/2}_F \nonumber\\
& =  \lVert H T H' (Z_1-Z_2) \lVert^{1/2}_F\nonumber\\
& \le 
\label{diff_y1_y2}
\lVert T \rVert^{1/2}_F \cdot \ \rVert Z_1-Z_2 \lVert^{1/2}_F.
\end{align}
Also, for $j=1,2$, 
\begin{align}
\lambda_{\min}(Y_j^{1/2}) &= (\lambda_{\min}(Y_j))^{1/2} = (\lambda_{\min}(H T H' Z_j))^{1/2} \nonumber\\
& \ge  (\lambda_{\min} (H T H'))^{1/2} (\lambda_{\min} (Z_j))^{1/2} 
\label{min_eigen_bound}
= \lambda_{\min} (T^{1/2}) \lambda_{\min} (Z_j^{1/2}).
\end{align}

Combining (\ref{nabla_bound_diff_2matrixargument_final})-(\ref{min_eigen_bound}), and using the fact that $\dd H$ is normalized, we obtain 
\begin{multline*}
\bigg\lVert \nabla_{Z_1} A_{\nu} (T, Z_1) - \nabla_{Z_2} A_{\nu}(T, Z_2) \bigg\rVert_F \\
\le \lVert T \rVert^{3/2}_F \, { \frac{\rVert Z_1-Z_2 \lVert^{1/2}_F}{ \lambda_{\min} (T^{1/2}) \lambda_{\min} (Z_2^{1/2})} \, \bigg[ \frac{C_1}{ \lambda_{\min} (T^{1/2}) \lambda_{\min} (Z_1^{1/2})} + C_2 \bigg]},
\end{multline*}
which is identical with \eqref{norm_bound_nabladiff_bessel_2matrixargument}.
\end{proof}

\medskip

Let $X$ be a Wishart-distributed random matrix, $X \sim W_m (\alpha, I_m)$, and define for $m \times m$ positive definite matrices $T$ the matrix-valued function 
\begin{equation}
\label{gLambdadef}
g(T) = E \left[\alpha^{-1} X^{1/2} \, \nabla_Z A_{\nu}(T,Z) \, X^{1/2} \Bigg|_{Z=\alpha^{-1}X}\right].
\end{equation}

\begin{lemma} For $T > 0$,
\begin{equation}
\label{trgLambdaformula}
\tr g(T) = -\frac{\alpha^{-1}}{\Gamma_m(\alpha)} \ (\tr T) \ \etr(-\alpha^{-1} T).
\end{equation}
\end{lemma}

\Pro
We will establish this result by the method of Laplace transforms.  For $R > 0$, the Laplace transform of the function $(\det T)^{\nu} \, \tr g(T)$ is
\begin{align}
\label{laplace1}
\widehat{g}(R) &:= \int_{T > 0}  \etr (-T R) \, (\det T)^{\nu} \, \tr g(T) \, \dd T. 
\end{align}
We substitute \eqref{gLambdadef} into this integral, interchange the trace and expectation, apply Fubini's theorem to interchange the expectation and the integral, and verify the validity of interchanging derivatives and integrals; then we obtain 
\begin{align}
\widehat{g}(R) 
\label{laplace2}
=\alpha^{-1} \tr \, E \bigg[X \, \nabla_{Z} \int_{T > 0} { \etr (-T R) \ (\det T)^{\nu} \  A_{\nu}(T, Z) } \, \dd T  \bigg|_{Z=\alpha^{-1}X} \bigg]. 
\end{align}
Applying \eqref{bessel_2matrixargument} to write $A_{\nu}(T, Z)$ as an average of its single-matrix argument counterpart,  
and reversing the order of integration, we obtain 
\begin{multline}
\widehat{g}(R) 
\label{laplace3}
= \alpha^{-1} \tr \, E \bigg[X \, \nabla_{Z} \, \int_{O(m)} \ { \int_{T > 0} { \etr (-T R) \ (\det T)^{\nu} }} \\
\times {{A_{\nu}(H T H' Z)} \, \dd T } \, \dd H \bigg|_{Z=\alpha^{-1}X} \bigg].
\end{multline}
The inner integral with respect to $T$ is precisely the Laplace transform (\ref{besselintegral}); 
substituting the outcome of that calculation into (\ref{laplace3}), we obtain
$$
\widehat{g}(R) = \alpha^{-1} (\det R)^{-\alpha} \, \tr \ E  \left[X \ \nabla_{Z} \, \int_{O(m)} \etr (-H' Z H R^{-1}) \, \dd H \bigg|_{Z=\alpha^{-1}X} \right].
$$
Interchanging the gradient and the integral, and then the integral and the trace, noting that
\begin{align*}
\nabla_{Z} \etr (-H' Z H R^{-1})\bigg|_{Z=\alpha^{-1}X} 
&= (-H R^{-1} H') \etr (-\alpha^{-1} H' X H R^{-1}),
\end{align*}
we find that
\begin{align}
\widehat{g}(R) 
\label{laplace4}
&= (\det R)^{-\alpha} \, E \int_{O(m)} \, \tr (-\alpha^{-1} H' X H R^{-1} ) \etr (-\alpha^{-1} H' X H R^{-1}) \, \dd H
\end{align}
since the trace and the integral commute.  

Next, we have
\begin{multline}
\int_{O(m)} \tr (-\alpha^{-1} H' X H R^{-1} ) \etr (-\alpha^{-1} H' X H R^{-1}) \dd H \\
\label{laplace6}
= \frac{\dd}{\dd t} \int_{O(m)} \exp(-t \ \tr(\alpha^{-1} H' X H R^{-1}) \, \dd H \bigg|_{t=1},
\end{multline}
by interchanging integral and derivative.  By \cite[p.~279, Eq.~(41)]{muirhead},
$$
\int_{O(m)} \  \exp[-t \, \tr(\alpha^{-1} H' X H R^{-1})] \dd H =\sum_{k=0}^{\infty} \frac{(-t \alpha^{-1})^k} {k!} \ \sum_{|\kappa|=k} \frac{C_{\kappa}(X) C_{\kappa} (R^{-1})}{C_{\kappa} (I_m)};
$$
differentiating this series term-by-term and evaluating the outcome at $t=1$, we find that \eqref{laplace6} equals 
\begin{equation}
\label{laplace7}
\sum_{k=1}^{\infty} \frac{(-\alpha)^{-k}} {(k-1)!} \ \sum_{|\kappa|=k} C_{\kappa}(X) \frac{C_{\kappa} (R^{-1})}{C_{\kappa} (I_m)}.
\end{equation}
By (\ref{zonalintegral}), $E \, C_{\kappa}(X) = [\alpha]_\kappa C_\kappa(I_m)$; therefore, 
by combining (\ref{laplace4})-(\ref{laplace7}), we obtain 
\begin{align}
\widehat{g}(R)
\label{laplace8}
&= -\alpha^{-1} \, (\det R)^{-\alpha} \sum_{k=1}^{\infty} \frac{(-\alpha^{-1})^{k-1}} {(k-1)!} \ \sum_{|\kappa|=k} [\alpha]_{\kappa} C_{\kappa} (R^{-1}).
\end{align}
It is also known from \cite[p.~248]{muirhead} that 
\begin{align*}
(\det(I_m + tR^{-1}))^{-a} &= \sum_{k=0}^{\infty} \frac{(-t)^{k}} {k!} \ \sum_{|\kappa|=k} [a]_{\kappa} C_{\kappa} (R^{-1}),
\end{align*}
for $\lVert t R^{-1} \rVert < 1$, where $\lVert \cdot \rVert$ denotes the maximum of the absolute values of the eigenvalues of $tR^{-1}$.  Differentiating this series term-by-term with respect to $t$, we obtain
\begin{equation*}
-\frac{\dd}{\dd t} (\det(I_m + tR^{-1}))^{-\alpha} = \sum_{k=1}^{\infty} \frac{(-t)^{k-1}} {(k-1)!} \ \sum_{|\kappa|=k} [\alpha]_{\kappa} C_{\kappa} (R^{-1});
\end{equation*}
now setting $t = \alpha^{-1}$ and comparing the outcome with (\ref{laplace8}), we find that 
\begin{align}
\label{laplace_final}
\widehat{g}(R) 
=\alpha^{-1} \ \frac{\dd}{\dd t} (\det(R+tI_m))^{-\alpha} \bigg|_{t=\alpha^{-1}}.
\end{align}
Therefore, by (\ref{gammaintegral}),
\begin{align*}
\widehat{g}(R) 
&= \frac{\alpha^{-1}}{\Gamma_m(\alpha)} \frac{\dd}{\dd t} \int_{T>0} \etr(-T(R+tI_m)) \, (\det T)^{\nu} \dd T \bigg|_{t=\alpha^{-1}} \\
&= \frac{\alpha^{-1}}{\Gamma_m(\alpha)} \int_{T > 0} \etr(-TR) \, (\det T)^\nu \, (-\tr T) \, \etr(-\alpha^{-1} T) \, \dd T,
\end{align*}
evidently a Laplace transform.  Comparing this expression with \eqref{laplace1} then the conclusion follows from the uniqueness theorem for Laplace transforms. 
$\qed$

\medskip

\begin{lemma} For $T > 0$, 
\begin{equation}
\label{gLambdaformula}
g(T) = - \frac{\alpha^{-1}}{m \, \Gamma_m(\alpha)} (\tr T) \etr(-\alpha^{-1} T) \, I_m,
\end{equation}
\end{lemma}

\Pro
Define for $Y > 0$ the function 
\begin{equation}
\label{gradientAnu}
\phi(Y) := 
\nabla_Z A_{\nu}(T,Z)\bigg|_{Z=\alpha^{-1}Y} \, .
\end{equation}
By (\ref{gLambdadef}), $g(T) = E \left[\alpha^{-1} X^{1/2} \, \phi(X) \, X^{1/2} \right]$, where $X \sim W_m(\alpha,I_m)$.  
Since the distribution of $X$ is orthogonally invariant, i.e., $X \overset{d}{=} H'XH$ for all $H \in O(m)$, then 
\begin{align}
\label{HgLambdaH}
H g(T) H' &= H E \left[\alpha^{-1} (H'XH)^{1/2} \, \phi(H'XH) \, (H'XH)^{1/2} \right] H' \nonumber \\
&= E \left[\alpha^{-1} X^{1/2} H \, \phi(H'XH) \, H'X^{1/2}\right].
\end{align}

By (\ref{gradientAnu}), 
\begin{align*}
\phi(H'XH) &= \nabla_Z A_{\nu}(T,Z)\bigg|_{Z=\alpha^{-1}H'XH} \\
&= \nabla_{H'ZH} A_{\nu}(T,H'ZH) \bigg|_{H'ZH=\alpha^{-1}H'XH}.
\end{align*}
By Maass \cite[p.~64]{maass}, $\nabla_{H'ZH} = H' \nabla_Z H$; 
so it follows that 
\begin{align*}
\phi(H'XH) &= H' \nabla_Z H A_{\nu}(T,H'ZH) \bigg|_{Z=\alpha^{-1}X} \\
&= H' \, \nabla_Z A_{\nu}(T,H'ZH) \bigg|_{Z=\alpha^{-1}X} \, H.
\end{align*}
However, $A_{\nu}(T,H'ZH) = A_{\nu}(T,Z)$ for all $H \in O(m)$; therefore, 
\begin{align*}
\phi(H'XH) &= H' \, \nabla_Z A_{\nu}(T,Z) \bigg|_{Z=\alpha^{-1}X} \, H = H' \phi(X) H.
\end{align*}
Substituting this result into (\ref{HgLambdaH}) we obtain, for all $H \in O(m)$, 
\begin{align*}
H g(T) H' 
&= E \left[\alpha^{-1} X^{1/2} \phi(X) X^{1/2}\right] = g(T).
\end{align*}

Since $H g(T) H' = g(T)$ for all $H \in O(m)$ then, by {\it Schur's Lemma} \cite[p.~315]{shilov}, $g(T)$ is a scalar matrix, i.e., $g(T) = \gamma_1 I_m$ for some scalar $\gamma_1$.  By taking traces and by applying (\ref{trgLambdaformula}), we obtain 
$$
m \gamma_1 = \tr \gamma_1 I_m = \tr g(T) = -\frac{\alpha^{-1}}{\Gamma_m(\alpha)} \ (\tr T) \ \etr(-\alpha^{-1} T);
$$
therefore, 
$$
\gamma_1 = - \frac{\alpha^{-1}}{m \, \Gamma_m(\alpha)} (\tr \, T) \etr(-\alpha^{-1} T). 
$$
The proof is now complete. $\qed$

\medskip

The final preliminary result needed for the proof of Theorem \ref{limitingnulldistribution_matrixcase} is the following consequence of \cite[Lemma 7, Eq. (20)]{khatri}.

\begin{lemma} 
\label{integral_trace_lambdamin}
The integrals 
$$
\int_{T > 0 } {\frac{ (\tr T)^{2}}{ \lambda_{\min} (T)}}\, \dd P_0(T)\quad\text{and}\quad \int_{T > 0 } {\frac{ (\tr T)^{3}}{ \lambda_{\min} (T)}}\, \dd P_0(T) 
$$
are finite for all $\alpha > \tfrac12(m+1)$. Further, the integral 
$$
\int_{T > 0 } {\frac{ (\tr T)^{3}}{ [\lambda_{\min} (T)]^2}}\, \dd P_0(T)
$$
is finite for all $\alpha > \tfrac12(m+3)$.
\end{lemma}

\subsubsection{The proof of the limiting distribution}
\label{proofofasympdistn}


In what follows, we will use for various matrices $V$ the shorthand notation 
$$
\nabla A_\nu(T,V) \equiv \nabla_Z A_\nu(T,Z)\Bigg|_{Z=V}
$$

\medskip

\noindent{\it Proof of Theorem \ref{limitingnulldistribution_matrixcase}}.  
By \eqref{taylorexpansion1}, the Taylor expansion of the Bessel function $A_{\nu}(T,Z)$ at $(T,Z_0)$ is 
\begin{equation}
\label{taylorexpansion2}
A_{\nu}(T,Z)=A_{\nu}(T,Z_0) + \langle Z - Z_0, \nabla A_{\nu}(T, U) \rangle,
\end{equation}
where $U= t Z +(1-t) Z_0$, for some $t \in [0,1]$.  
Setting $Z = Z_j$ and $Z_0 = \alpha^{-1} X_j$, $j=1,\dotsc,n$, in (\ref{taylorexpansion2}), we have the Taylor expansion of order 1 of $A_{\nu}(T, Z_j)$ at $(T, \alpha^{-1} X_j)$:
\begin{align}
\label{taylorexp_matrix1}
A_{\nu}(T, Z_j) = A_{\nu}(T, \alpha^{-1} X_j) + \langle Z_j- \alpha^{-1}X_j, \nabla A_{\nu}(T, U_j) \rangle,
\end{align}
where $U_j= t Z_j +(1-t) \alpha^{-1} X_j$, for some $t \in [0,1]$. Define 
$$
M_n=\bar{X}_n^{-1/2} (\alpha I_m-\bar{X}_n) \bar{X}_n^{-1/2};
$$ 
then (\ref{taylorexp_matrix1}) reduces to 
$$
A_{\nu}(T, Z_j) = A_{\nu}(T, \alpha^{-1} X_j) + \langle \alpha^{-1} X_j^{1/2} M_n X_j^{1/2}, \nabla A_{\nu}(T, U_j) \rangle.
$$
Adding and subtracting the term 
$\langle \alpha^{-1} X_j^{1/2} M_n X_j^{1/2}, \nabla A_{\nu}(T, \alpha^{-1} X_j) \rangle$ 
on the right-hand side, we obtain
\begin{align}
A_{\nu} (T, Z_j) &= A_{\nu}(T, \alpha^{-1} X_j) + \langle \alpha^{-1} X_j^{1/2} M_n X_j^{1/2}, \nabla A_{\nu}(T, U_j) \rangle \nonumber \\
& \qquad + \big\langle \alpha^{-1} X_j^{1/2} M_n X_j^{1/2}, \nabla A_{\nu}(T, U_j)-  \nabla A_{\nu}(T, \alpha^{-1} X_j) \big\rangle \nonumber \\
\label{taylorexp_matrix}
&= A_{\nu}(T, \alpha^{-1} X_j) + \langle M_n, \alpha^{-1} X_j^{1/2} \, \nabla A_{\nu}(T, \alpha^{-1} X_j) \, X_j^{1/2} \rangle \nonumber\\
& \qquad + \big\langle M_n, \alpha^{-1} X_j^{1/2} \ \big[ \nabla A_{\nu}(T, U_j)-  \nabla A_{\nu}(T, \alpha^{-1}X_j) \big] \ X_j^{1/2} \big\rangle,
\end{align}
where the second equality is obtained by permuting terms cyclically in the inner product.  
For  $T > 0$ and $X_j > 0$, $j=1,\dotsc,n$, define the function
\begin{equation*}
g(T,X_j):= \alpha^{-1} \ X_j^{1/2} \nabla A_{\nu}(T, \alpha^{-1}X_j) X_j^{1/2}, 
\end{equation*}
We remark that as $X_1,\ldots,X_n$ are i.i.d. then $E_{X_j} g(T,X_j)$ does not depend on $j$; hence, 
\begin{equation*}
g(T) := E_{X_j} g(T,X_j) = E( \alpha^{-1} \ X_j^{1/2} \, \nabla A_{\nu}(T, \alpha^{-1} X_j) X_j^{1/2}).
\end{equation*}
is a function evaluated earlier; by (\ref{gLambdaformula}), 
$$
g(T)=  - \frac{\alpha^{-1}}{m \, \Gamma_m(\alpha)} (\tr T) \etr(-\alpha^{-1} T) \, I_m.
$$

Define the random fields $\mathcal{Z}_{n,1}(T)$, $\mathcal{Z}_{n,2}(T)$ and  $\mathcal{Z}_{n,3}(T)$, $T > 0$, by
\begin{align*}
\mathcal{Z}_{n,1}(T) &= \frac{1}{\sqrt{n}} \sum_{j=1}^{n} \bigg[ \Gamma_m(\alpha) A_{\nu}(T, \alpha^{-1}X_j) + \Gamma_m(\alpha) \langle M_n , g(T,X_j) \rangle - \etr(-\alpha^{-1} T) \bigg], \\
\mathcal{Z}_{n,2}(T) &= \frac{1}{\sqrt{n}}\sum_{j=1}^{n}\bigg[ \Gamma_m(\alpha) A_{\nu}(T, \alpha^{-1}X_j) + \Gamma_m(\alpha) \langle M_n, g(T) \rangle - \etr(-\alpha^{-1} T) \bigg],\\
\mathcal{Z}_{n,3}(T) &= \frac{1}{\sqrt{n}}\sum_{j=1}^{n}\bigg[ \Gamma_m(\alpha) A_{\nu}(T, \alpha^{-1}X_j) 
{\hskip-0.7pt}+{\hskip-0.7pt} \Gamma_m(\alpha) \langle \alpha^{-1}(\alpha I_m-X_j),g(T\rangle - \etr(-\alpha^{-1}T)\bigg].
\end{align*}

The random fields $\mathcal{Z}_{n,k}$, $k=1,2,3$ arise as follows. To define $\mathcal{Z}_{n,1}(T)$, we use the first two terms in (\ref{taylorexp_matrix}).  
To define $\mathcal{Z}_{n,2}(T)$, we use the same expression from $\mathcal{Z}_{n,1}(T)$ except that the term $g(T, X_j)$ is replaced by its expected value $g(T)$, which is given by (\ref{gLambdaformula}). To define $\mathcal{Z}_{n,3}(T)$, we replace the term $M_n$ in $\mathcal{Z}_{n,2}(T)$ by a constant multiple of $\alpha I_m-X_j$, the constant being obtained by applying the Law of Large Numbers to $\bar{X_{n}}^{-1/2}$. We will show that
\begin{eqnarray}
\label{nulldistr_1} &\mathcal{Z}_{n,3} \xrightarrow{d} \mathcal{Z} \ \ \text{in $L^2$},&  \\
\label{nulldistr_2} &\lVert \mathcal{Z}_{n}-\mathcal{Z}_{n,1} \rVert_{L^2} \xrightarrow{p} 0,&\\
\label{nulldistri_3} &\lVert \mathcal{Z}_{n,1}-\mathcal{Z}_{n,2} \rVert_{L^2} \xrightarrow{p} 0,&\\
\label{nulldistr_4} &\lVert \mathcal{Z}_{n,2}-\mathcal{Z}_{n,3} \rVert_{L^2} \xrightarrow{p} 0.
\end{eqnarray}
By writing $\mathcal{Z}_n$ as 
$$
\mathcal{Z}_n = \mathcal{Z}_{n} - \mathcal{Z}_{n,1} + \mathcal{Z}_{n,1}-\mathcal{Z}_{n,2} + \mathcal{Z}_{n,2}-\mathcal{Z}_{n,3} + \mathcal{Z}_{n,3},
$$
it will follow that $\mathcal{Z}_n \xrightarrow{d} \mathcal{Z}$ in $L^2$ (cf. Billingsley \cite[ p.~25, Theorem 4.1]{ref21}).

To establish (\ref{nulldistr_1}), define for $T > 0$, 
\begin{equation}
\label{zn31_matrixcase}
\mathcal{Z}_{n, 3, j }(T) := \Gamma_m(\alpha) A_{\nu}(T, \alpha^{-1}X_j) 
+ \Gamma_m(\alpha)\langle \alpha^{-1}(\alpha I_m-X_j) , g(T) \rangle - \etr(-\alpha^{-1} T),
\end{equation}
$j=1,\dotsc,n$.  Since $X_j \sim W_m(\alpha, I_m)$ then $E(X_j-\alpha I_m) = 0$ and therefore, since the trace and the expectation are linear operators, we deduce that
\begin{align*}
E \bigg[ \langle \alpha^{-1}(\alpha I_m-X_j), \, g(T) \rangle \bigg] =\tr\bigg[ \alpha^{-1} E(\alpha I_m-X_j) \cdot g(T)\bigg] =0.
\end{align*}
Also, by Example \ref{orthogonallyhankeltransformwishartdistn_example} and (\ref{kummer_matrixargument}), we have 
$
E \left[ \Gamma_m (\alpha) A_{\nu}(T, \alpha^{-1}X_j) \right] 
= \etr(-\alpha^{-1}T).
$ 
Therefore, $E(\mathcal{Z}_{n, 3, j}(T))=0$, for all $T > 0$ and $j=1, \dotsc, n$, and it is also clear that  $\mathcal{Z}_{n,3,1},\dotsc,\mathcal{Z}_{n,3,n}$ are independent and identically distributed random elements in $L^2$.  

We now show that $E(\lVert \mathcal{Z}_{n,3,j} \rVert^2_{L^2}) < \infty$ for $j=1,\dotsc,n$.  We have
\begin{align*}
E(\lVert \mathcal{Z}_{n,3,j} \rVert^2_{L^2}) &= E  \int_{T > 0} \mathcal{Z}^2_{n,3,j} (T) \dd P_0(T)  \\
&= E \int_{T > 0} \bigg[ \Gamma_m(\alpha) A_{\nu}(T,\alpha^{-1} X_j) \\
& \qquad\qquad\quad + \Gamma_m(\alpha)\langle \alpha^{-1}(\alpha I_m-X_j) , g(T) \rangle - \etr(-\alpha^{-1}T) \bigg]^2 \dd P_0(T).
\end{align*}
By the Cauchy-Schwarz inequality, $(a+b+c)^2 \le 3(a^2+b^2+c^2)$ for $a,b,c \in \mathbb{R}$; so to prove that $E(\lVert \mathcal{Z}_{n,3,j} \rVert^2_{L^2}) < \infty$, it suffices to prove that 
\begin{equation}
\label{limitingnullterm1_matrix}
E  \int_{T > 0} \bigg[ \Gamma_m(\alpha) A_{\nu}(T, \alpha^{-1}X_j) \bigg]^2 \dd P_0(T)  < \infty,
\end{equation}
\begin{equation}
\label{limitingnullterm2_matrix}
E \int_{T > 0} \bigg[\Gamma_m(\alpha)\langle \alpha^{-1}(\alpha I_m-X_j) , g(T) \rangle \bigg]^2 \dd P_0(T) < \infty,
\end{equation}
and 
\begin{equation}
\label{limitingnullterm3_matrix}
 E  \int_{T > 0} \etr(-2\alpha^{-1}T) \dd P_0(T) < \infty.
\end{equation}

To establish (\ref{limitingnullterm1_matrix}), we apply (\ref{2besselineq_matrixargument}) to obtain
$$
E \int_{T > 0} \bigg[ \Gamma_m(\alpha) A_{\nu}(T, \alpha^{-1}X_j) \bigg]^2 \dd P_0(T) \le E \int_{T > 0} 1 \cdot \dd P_0(T) = 1.
$$
To prove (\ref{limitingnullterm2_matrix}), write 
\begin{align*}
\langle (\alpha I_m-X_j) , g(T) \rangle ^2  &= (\tr[(\alpha I_m-X_j) \cdot g(T)] )^2 \\
&= \bigg( \frac{\alpha^{-1}}{m \Gamma_m(\alpha)} \bigg)^2 \, ( \tr(\alpha I_m-X_j))^2 \, (\tr T )^2 \, \etr(-2\alpha^{-1}T);
\end{align*}
therefore, the integral in \eqref{limitingnullterm2_matrix} is a constant multiple of 
\begin{align*}
E ( \tr(\alpha I_m-X_j))^2 \cdot \int_{T > 0} (\tr T )^2 \ \etr(-2\alpha^{-1}T) \ \dd P_0(T).
\end{align*}

Since $(\tr(\alpha I_m-X_j))^2$ is a polynomial in $X_j$, its expectation is finite because the moment-generating function of $X$ exists.  As for 
\begin{equation}
\label{gLambda_integral}
\int_{T > 0} (\tr T )^2 \ \etr(-2\alpha^{-1}T) \ \dd P_0(T), 
\end{equation}
again this integral is finite because $(\tr T)^2$ is a polynomial and $\etr(-2\alpha^{-1}T) \, \dd P_0(T)$, after normalization, is a Wishart measure.  For the same reason, (\ref{limitingnullterm3_matrix}) is valid.  

In summary, for $T > 0$ and $j=1, \dotsc, n$, $\mathcal{Z}_{n,3,1},\dotsc, \mathcal{Z}_{n,3,n}$ are i.i.d. random elements in $L^2$ with $E(\mathcal{Z}_{n, 3, j}(T)) = 0$ and $E(\lVert \mathcal{Z}_{n,3,j} \rVert^2_{L^2}) < \infty$.  Therefore, by the Central Limit Theorem in $L^2$, 
$$
\frac{1}{\sqrt{n}}\sum_{j=1}^{n} \mathcal{Z}_{n,3,j} \xrightarrow{d} \mathcal{Z},
$$
where $\mathcal{Z}:=(\mathcal{Z}(T), T > 0)$ is a centered Gaussian random element in $L^2$. Moreover, $\mathcal{Z}$ has the same covariance operator as $\mathcal{Z}_{n,3,1}$. 

It is well-known that the covariance operator of the random element $\mathcal{Z}_{n,3,1}$ is uniquely determined by the covariance function of the random field $\mathcal{Z}_{n,3,1}$; cf., G{\=\i}khman and Skorohod \cite[pp.~218-219]{ref20}. 

We now show that the function $K(S, T)$ in (\ref{covariancefn_matrixcase}) is the covariance function of $\mathcal{Z}_{n,3,1}$. Noting that $E[\mathcal{Z}_{n,3,1}(T)]=0$ for all $T > 0$, we obtain 
\begin{align}
\label{covariancefulldef_matrixcase}
K(S, T) &= \Cov [\mathcal{Z}_{n,3,1}(S), \mathcal{Z}_{n,3,1}(T)]\nonumber\\
&=\Cov [\mathcal{Z}_{n,3,1}(S) +\etr(-\alpha^{-1} S), \mathcal{Z}_{n,3,1}(T)+\etr(-\alpha^{-1} T) ] \nonumber\\
&= E [(\mathcal{Z}_{n,3,1}(S) +\etr(-\alpha^{-1}S) ) \cdot (\mathcal{Z}_{n,3,1}(T)+\etr(-\alpha^{-1} T))] \nonumber \\
& {\hskip 3truein} - \etr(-\alpha^{-1} (S+T)).
\end{align}
By (\ref{zn31_matrixcase}), 
\begin{align}
E [(\mathcal{Z}_{n,3,1}(S) & + \etr(-\alpha^{-1} S) ) \cdot (\mathcal{Z}_{n,3,1}(T)+\etr(-\alpha^{-1} T))] \nonumber\\
\label{covariancecalc_matrixcase}
&= E \bigg[ \Gamma_m(\alpha) A_{\nu}(S, \alpha^{-1}X_1) + \Gamma_m(\alpha)\langle \alpha^{-1}(\alpha I_m-X_1) , g(S) \rangle \bigg] \nonumber \\
& \qquad \times E \bigg[ \Gamma_m(\alpha) A_{\nu}(T, \alpha^{-1}X_1) + \Gamma_m(\alpha)\langle \alpha^{-1}(\alpha I_m-X_1) , g(T) \rangle \bigg],
\end{align}
so the calculation of $K(S, T)$ reduces to evaluating the four terms obtained by expanding the product on the right-hand side of (\ref{covariancecalc_matrixcase}).

The first term in the product in (\ref{covariancecalc_matrixcase}) is 
\begin{multline}
\label{boundredcov_matrixcase}
E  [ \Gamma_m(\alpha)]^2 A_{\nu}(S, \alpha^{-1}X_1) \ A_{\nu}(T, \alpha^{-1}X_1) \\
=\Gamma_m(\alpha) \int_{X > 0} {A_{\nu}(S, \alpha^{-1}X) \ A_{\nu}(T, \alpha^{-1}X) }
{(\det X)^{\nu} \etr(-X)}\, \dd X.
\end{multline}
By (\ref{besselproductintegral}), (\ref{bessel_2matrixargument}), and Fubini's theorem, we find that this term equals
\begin{align*}
\Gamma_m&(\alpha) \int_{O(m)} \int_{O(m)} \int_{X > 0} {A_{\nu} (\alpha^{-1} H S H' X) A_{\nu} (\alpha^{-1}K T K' X)} 
{(\det X)^{\nu}\etr(-X) }\, \dd X \, \dd H  \dd K \\
&= \Gamma_m(\alpha) \int_{O(m)} \int_{O(m)} \ A_{\nu} (-\alpha^{-2} H S H'  K T K') 
\etr(-\alpha^{-1}(H S H'+ K T K')) \, \dd H \dd K.
\end{align*}
Since $\etr(-\alpha^{-1}(H S H'+ K T K'))= \etr(-\alpha^{-1} (S +T))$, and 
\begin{align*}
\int_{O(m)} A_{\nu} (-\alpha^{-2} H S H'  K T K') \, \dd H = A_{\nu} (-\alpha^{-2} S, K T K') 
= A_{\nu} (-\alpha^{-2}S, T),
\end{align*}
we conclude that the first term equals 
\begin{equation}
\label{firsttermcovariance_matrix}
\Gamma_m(\alpha) \etr(-\alpha^{-1}(S+ T)) A_{\nu} (-\alpha^{-2}S, T).
\end{equation}

The second term in the product in (\ref{covariancecalc_matrixcase}) is 
\begin{align}
\alpha^{-1} [\Gamma_m&(\alpha)]^2 E \left[ A_{\nu}(S, \alpha^{-1}X_1) \cdot \langle (\alpha I_m-X_1) , g(T) \rangle\right] \nonumber\\
&= \frac{\Gamma_m(\alpha)}{\alpha^2 m} \ E \left[ A_{\nu}(S, \alpha^{-1}X_1) \cdot \langle (X_1-\alpha I_m) , (\tr T) \etr(-\alpha^{-1} T)I_m \rangle\right]\nonumber\\ 
\label{secondterm_covariance}
&=\frac{1}{\alpha^2 m} \Gamma_m(\alpha) (\tr T) \etr(-\alpha^{-1} T) E \left[\Big(\tr X_1) - 1\Big)\cdot A_{\nu}(S,\alpha^{-1}X_1)\right].
\end{align}
We have seen earlier that 
\begin{equation}
\label{secondterm_covfn1}
\Gamma_m(\alpha) \, E \, A_{\nu}(S, \alpha^{-1}X_1) = \etr(-\alpha^{-1} S).
\end{equation}
Also, by (\ref{bessel_2matrixargument}), 
\begin{align}
\label{secondterm_covfn}
E (\tr X_1) \, A_{\nu}(S, \alpha^{-1}X_1) = \int_{O(m)} \tr E \bigg( X_1\cdot A_{\nu}(\alpha^{-1} H S H' X_1) \bigg) \, \dd H .
\end{align}
Since 
$
\Gamma_m(\alpha) \, A_{\nu}(\alpha^{-1} H S H' X_1) = {_0}F_1(\alpha; -\alpha^{-1} H S H' X_1)
$ 
then, by \cite[p.~442]{muirhead}, 
the expectation
$
E \left(X_1\cdot A_{\nu}(\alpha^{-1} HSH'X_1)\right)
$ 
is a multiple of the expected value of a noncentral Wishart distributed random matrix $W_m(\alpha,I_m,\Omega)$, where $\Omega = -\alpha^{-1} H S H'$ is the matrix of noncentrality parameters.  Hence, 
\begin{align*}
E \left( (\tr X_1) \, A_{\nu}(\alpha^{-1}HSH'X_1) \right) &= \tr \, E \left(X_1 \, A_{\nu}(\alpha^{-1}HSH'X_1) \right) \\
&= \frac{1}{\Gamma_m(\alpha)} \, \tr (\alpha I_m - \alpha^{-1} \Omega) \, \etr(-\alpha^{-1} \Omega) \\
&= \frac{1}{\Gamma_m(\alpha)} \, (\alpha m - \alpha^{-1} \tr S) \, \etr(-\alpha^{-1}S) .
\end{align*}
Substituting this result into (\ref{secondterm_covfn}), we obtain
\begin{equation}
\label{secondterm_covfn2}
E \left( (\tr X_1) \, A_{\nu}(S, \alpha^{-1}X_1) \right) 
=\frac{1}{\Gamma_m(\alpha)} \ \etr(-\alpha^{-1} S) \ [\alpha m -\alpha^{-1} (\tr S)].
\end{equation}
Substituting (\ref{secondterm_covfn1}) and (\ref{secondterm_covfn2}) into (\ref{secondterm_covariance}), and simplifying the result, we find that the second term equals 
$$
-\frac{1}{\alpha^{3}m} (\tr S) (\tr T) \etr(-\alpha^{-1} (S+T)).
$$

The third term in the product in (\ref{covariancecalc_matrixcase}) is
$$
\alpha^{-1} \Gamma^2_m(\alpha) \ E \bigg[ A_{\nu}(T, \alpha^{-1}X_1) \cdot \langle (\alpha I_m-X_1) , g(S) \rangle\bigg],
$$
which is the same as the second term but with $S$ and $T$ interchanged. 

The fourth term in the product in (\ref{covariancecalc_matrixcase}) is
\begin{multline*}
[\alpha^{-1} \Gamma_m(\alpha)]^2 \, E \bigg[ \langle (\alpha I_m-X_1), g(S) \rangle \cdot \langle (\alpha I_m-X_1) , g(T) \rangle \bigg] \\
= [\alpha^{-1} \Gamma_m(\alpha)]^2 \ E \bigg[ \tr ((\alpha I_m-X_1) \cdot g(S)) \cdot \tr ((\alpha I_m-X_1) \cdot g(T)) \bigg].
\end{multline*}
Using the explicit formula for $g(T)$ from 
(\ref{trgLambdaformula}) and (\ref{gLambdaformula}), we obtain 
\begin{align}
E \bigg[ \tr & ((\alpha I_m-X_1) \cdot g(S))  \cdot \tr ((\alpha I_m-X_1) \cdot g(T)) \bigg]\nonumber\\
&= \frac{1}{[\Gamma_m(\alpha)]^2}  \bigg[ \ (\tr S) \etr(-\alpha^{-1} S) (\tr T)  \etr(-\alpha^{-1} T )\nonumber\\
& \qquad\qquad\qquad - 2 (m\alpha)^{-1}  \ (\tr S) \etr(-\alpha^{-1} S) (\tr T) \etr(-\alpha^{-1} T) \ E (\tr X_1) \nonumber\\
\label{fifthterm_covfn2}
& \qquad\qquad\qquad + (m\alpha)^{-2} \, (\tr S) \etr(-\alpha^{-1} S) (\tr T) \etr(-\alpha^{-1} T) \ E (\tr X_1)^2 \bigg].
\end{align}
By (\ref{trace_zonal}) and (\ref{zonalintegral}), it follows that
\begin{equation}
\label{expected_trace_fifthterm}
E(\tr X_1) = [\alpha]_{(1)} C_{(1)}(I_m) = \alpha m.
\end{equation}
Also, using \eqref{zonal_identitymatrix}, we obtain 
\begin{align}
\label{expected_trace_squared_fifthterm}
E[(\tr X_1)^2] &= \sum_{|\kappa|=2} [\alpha]_{\kappa} C_{\kappa} (I_m) \nonumber \\
&= [\alpha]_{(2)} C_{(2)}(I_m) + [\alpha]_{(11)} C_{(11)} (I_m) = \alpha m (\alpha m + 1). 
\end{align}
Substituting (\ref{expected_trace_fifthterm}) and (\ref{expected_trace_squared_fifthterm}) into (\ref{fifthterm_covfn2}), we deduce that the fourth term equals 
$$
\frac{1}{\alpha^{3}m} (\tr S) (\tr T) \etr(-\alpha^{-1} (S+T)).
$$
Combining all four terms, we obtain (\ref{covariancefn_matrixcase}).

To establish (\ref{nulldistr_2}), we begin by showing that 
$$
\tr \left[(\sqrt{n} \, M_n)^2 \right]=\rVert \bar{X}_n^{-1/2} \sqrt{n} (\alpha I_m-\bar{X}_n) \bar{X}_n^{-1/2} \lVert^2_F
$$
converges in distribution to a random variable with finite variance. By the multivariate Central Limit Theorem, $\sqrt{n} \vect(\alpha I_m-\bar{X}_n)$ converges in distribution to a multivariate normal random vector.  Also, by the Law of Large Numbers, $\bar{X}_n^{-1} \xrightarrow{p} \alpha^{-1} I_m$.  Therefore, by Slutsky's theorem, $\sqrt{n} \vect(M_n)$ converges in distribution to a multivariate normal random vector, so it follows from the Continuous Mapping Theorem that $\tr \left[(\sqrt{n} \, M_n)^2 \right]$ converges in distribution to a random variable which has finite variance.  

By the Taylor expansion (\ref{taylorexp_matrix}), 
\begin{align*}
\mathcal{Z}_n-\mathcal{Z}_{n,1} 
&=\frac{\Gamma_m(\alpha)}{\sqrt{n}} \sum_{j=1}^{n} \left\langle M_n , \alpha^{-1} X_j^{1/2} \left( \nabla A_{\nu}(T, U_j) - \nabla A_{\nu}(T, \alpha^{-1}X_j) \right) X_j^{1/2} \right\rangle \\
&=\frac{\alpha^{-1} \Gamma_m(\alpha)}{n} \sum_{j=1}^{n} \left\langle  \sqrt{n} M_n , X_j^{1/2} \left( \nabla A_{\nu}(T, U_j) - \nabla A_{\nu}(T, \alpha^{-1}X_j) \right) \ X_j^{1/2} \right\rangle.
\end{align*}
Define
$$
V_n := \frac{1}{n^2} \int_{T >0} {\tr \bigg[\sum_{j=1}^n  X_j^{1/2} \ \bigg( \nabla A_{\nu}(T, U_j)} - {\nabla A_{\nu}(T, \alpha^{-1} X_j) \bigg) \ X_j^{1/2}  \bigg]^2} \, \dd P_0(T).
$$
By the Cauchy-Schwarz inequality,
\begin{equation}
\label{zn_zn1}
\lVert \mathcal{Z}_n - \mathcal{Z}_{n,1} \rVert^2_{L^2} \le [\alpha^{-1} \Gamma_m(\alpha)]^2 \tr \left[(\sqrt{n} \, M_n)^2 \right] \cdot V_n ;
\end{equation}
so we will establish (\ref{nulldistr_2}) by proving that $V_n \xrightarrow{p} 0$.  

By the triangle inequality and the sub-multiplicative property of the Frobenius norm, we have
\begin{align*}
\tr \bigg[\sum_{j=1}^n  X_j^{1/2} & \bigg( \nabla A_{\nu}(T, U_j)- \nabla A_{\nu}(T, \alpha^{-1} X_j) \bigg) \ X_j^{1/2}  \bigg]^2 &&\nonumber\\
& = \bigg\lVert  \sum_{j=1}^n X_j^{1/2} \ \bigg( \nabla A_{\nu}(T, U_j)-  \nabla A_{\nu}(T, \alpha^{-1}X_j) \bigg) \ X_j^{1/2} \bigg\rVert^2_F &&\nonumber\\
& \le \bigg(\sum_{j=1}^n \rVert X_j \lVert_F \rVert \nabla A_{\nu}(T, U_j)-  \nabla A_{\nu}(T, \alpha^{-1} X_j) \rVert_F \bigg)^2. &&
\end{align*}
Applying (\ref{norm_bound_nabladiff_bessel_2matrixargument}), we obtain
\begin{multline*}
\rVert \nabla A_{\nu}(T, U_j) - \nabla A_{\nu}(T, \alpha^{-1}X_j) \rVert_F \\
\le \frac{\lVert T \rVert^{3/2}_F \, \rVert U_j -\alpha^{-1} X_j \lVert^{1/2}_F}{\lambda_{\min}(X_j^{1/2})} \, \bigg[ \frac{C_1}{ \lambda_{\min}(T) \, \lambda_{\min} (U_j^{1/2})} + \frac{C_2}{\lambda_{\min} (T^{1/2}) } \bigg].
\end{multline*}
Also, since $U_j = X_j^{1/2} [ \alpha^{-1} I_m +t (\bar{X}_n^{-1} -\alpha^{-1} I_m) ] X_j^{1/2}$, $t \in [0,1]$, then 
\begin{align*}
\rVert U_j-\alpha^{-1} X_j \lVert^{1/2}_F &= \| X_j^{1/2} [t (\bar{X}_n^{-1} -\alpha^{-1} I_m)] X_j^{1/2} \|^{1/2}_F \nonumber\\
&=\, \| t X_j (\bar{X}_n^{-1} -\alpha^{-1} I_m) \|^{1/2}_F \nonumber\\
&\le \| X_j \|^{1/2}_F \ \| \bar{X}_n^{-1} -\alpha^{-1} I_m \|^{1/2}_F.
\end{align*}
Define
$$
V_{n,1} := C_1^2 \lVert \bar{X}_n^{-1} -\alpha^{-1} I_m \rVert_F \ \bigg( \frac{1}{n} \sum_{j=1}^n \frac{\rVert X_j  \lVert^{3/2}_F} {\lambda_{\min}(X^{1/2}_j) \lambda_{\min} (U_j^{1/2})} \bigg)^2 \cdot \int_{T > 0 } {\frac{ \lVert T \rVert^{3}_F}{ [\lambda_{\min} (T)]^2}}\, \dd P_0(T),
$$
and
$$
V_{n,2} := C_2^2 \lVert \bar{X}_n^{-1} -\alpha^{-1} I_m \rVert_F \ \bigg( \frac{1}{n} \sum_{j=1}^n \frac{\rVert X_j  \lVert^{3/2}_F} {\lambda_{\min}(X^{1/2}_j)} \bigg)^2 \cdot \int_{T > 0 } {\frac{ \lVert T \rVert^{3}_F}{\lambda_{\min} (T)}}\, \dd P_0(T).
$$
By the Cauchy-Schwarz inequality, $V_n \le V_{n,1} +V_{n,2}$. Thus, it suffices to show that $V_{n,1}, V_{n, 2} \xrightarrow{p} 0$. 

We first establish that $V_{n,1} \xrightarrow{p} 0$. By the Cauchy-Schwarz inequality, 
\begin{align*}
\bigg(\frac{1}{n} \sum_{j=1}^n \frac{\rVert X_j  \lVert^{3/2}_F} {\lambda_{\min}(X^{1/2}_j) \lambda_{\min} (U_j^{1/2})} \bigg)^2 
& \le \frac{1}{n} \sum_{j=1}^n \frac{ \rVert X_j  \lVert^{3}_F} {\lambda_{\min}(X_j) \lambda_{\min} (U_j)} \\
& = \frac{1}{n} \sum_{j=1}^n \frac{ (\tr(X^2_j))^{3/2}} {\lambda_{\min}(X_j) \lambda_{\min} (U_j)}.
\end{align*}
By Weyl's inequality for the smallest eigenvalue of the sum of two symmetric matrices, 
\begin{align*}
\lambda_{\min} (U_j) 
& \ge t \lambda_{\min} (X_j \bar{X}_n^{-1}) + (1-t) \alpha^{-1} \lambda_{\min} (X_j) \\
& \ge t \lambda_{\min} (X_j) \lambda_{\min} (\bar{X}_n^{-1}) + (1-t) \alpha^{-1} \lambda_{\min} (X_j) \\
& \ge  \min \{  \lambda_{\min} (X_j) \lambda_{\min} (\bar{X}_n^{-1}), \alpha^{-1} \lambda_{\min} (X_j) \} \\
&= \lambda_{\min} (X_j) \ \min \{  \lambda_{\min} (\bar{X}_n^{-1}), \alpha^{-1} \};
\end{align*}
therefore, 
$$
\frac{1}{n} \sum_{j=1}^n \frac{ (\tr(X^2_j))^{3/2}} {\lambda_{\min}(X_j) \lambda_{\min} (U_j)} \le \ \frac{1} {\min \{  \lambda_{\min} (\bar{X}_n^{-1}), \alpha^{-1} \}} \ \frac{1}{n} \sum_{j=1}^n \frac{ (\tr(X^2_j))^{3/2}} {[\lambda_{\min}(X_j)]^2}.
$$
By the Law of Large Numbers and the Continuous Mapping Theorem, we have
$$
\frac{\lVert \bar{X}_n^{-1} -\alpha^{-1} I_m \rVert_F}{\min \{\lambda_{\min} (\bar{X}_n^{-1}), \alpha^{-1} \}} \xrightarrow{p} 0.
$$
Again by the Law of Large Numbers, 
$$
\frac{1}{n} \sum_{j=1}^n \frac{ (\tr(X^2_j))^{3/2}} {[\lambda_{\min}(X_j)]^2} \xrightarrow{p} E_{P_0} \bigg(\frac{ (\tr(X^2))^{3/2}} {[\lambda_{\min}(X)]^2}\bigg).
$$
Therefore, to complete the proof of $V_{n, 1} \xrightarrow{p} 0$, we need to establish that 
$$
\int_{T > 0 } {\frac{ \lVert T \rVert^{3}_F}{ [\lambda_{\min} (T)]^2}}\, \dd P_0(T) < \infty \quad\text{and}\quad E_{P_0} \bigg(\frac{ (\tr(X^2))^{3/2}} {[\lambda_{\min}(X)]^2}\bigg) < \infty.
$$
Since $\lVert T \rVert^{3}_F=(\tr T^2)^{3/2}$ then these criteria are the same, so we show that the first one holds.  
For $T > 0$, we have $ \tr T^2  \le (\tr T)^2$ and hence $(\tr T^2)^{3/2} \le (\tr T)^3$.  By Lemma \ref{integral_trace_lambdamin}, 
$$
\int_{T > 0 } {\frac{ (\tr T)^{3}}{ [\lambda_{\min} (T)]^2}}\, \dd P_0(T) < \infty,
$$
for $\alpha > \tfrac12(m+3)$, so it follows that $V_{n, 1} \xrightarrow{p} 0$.  

As for $V_{n, 2} \xrightarrow{p} 0$, the proof is similar.  By the Cauchy-Schwarz inequality, 
$$
\bigg(\frac{1}{n} \sum_{j=1}^n \frac{\rVert X_j  \lVert^{3/2}_F} {\lambda_{\min}(X^{1/2}_j)} \bigg)^2 \ \le \ \frac{1}{n} \sum_{j=1}^n \frac{ \rVert X_j  \lVert^{3}_F} {\lambda_{\min}(X_j)}=\frac{1}{n} \sum_{j=1}^n \frac{ (\tr X^2_j)^{3/2}} {\lambda_{\min}(X_j)}.
$$
Applying the Law of Large Numbers and the Continuous Mapping Theorem, we obtain $\lVert \bar{X}_n^{-1} -\alpha^{-1} I_m \rVert_F \xrightarrow{p} 0$ and
$$
\frac{1}{n} \sum_{j=1}^n \frac{ (\tr X^2_j)^{3/2}} {\lambda_{\min}(X_j)} \xrightarrow{p} E_{P_0} \bigg( \frac{ (\tr X^2)^{3/2}} {\lambda_{\min}(X)}\bigg).
$$
Thus, to complete the proof of $V_{n, 2} \xrightarrow{p} 0$, we need to establish that 
$$
\int_{T > 0 } {\frac{ \lVert T \rVert^{3}_F}{ \lambda_{\min} (T)}}\, \dd P_0(T) < \infty \quad\text{and}\quad E_{P_0} \bigg( \frac{ (\tr X^2)^{3/2}} {\lambda_{\min}(X)}\bigg) < \infty,
$$
which are identical criteria.  Since $\lVert T \rVert^{3}_F=(\tr T^2)^{3/2}$, it suffices to show that
$$
\int_{T > 0 } {\frac{ (\tr T^2)^{3/2}}{\lambda_{\min} (T)}}\, \dd P_0(T) < \infty.
$$
However, $\tr T^2 \le (\tr T)^2$ so $(\tr T^2)^{3/2} \le  (\tr T)^3$ so, by Lemma \ref{integral_trace_lambdamin},
$$
\int_{T > 0 } {\frac{ (\tr T)^{3}}{\lambda_{\min} (T)}}\, \dd P_0(T) < \infty
$$
for all $\alpha > \tfrac12(m+1)$.  Therefore, $V_{n,2} \xrightarrow{p} 0$ for all $\alpha > \tfrac12(2m-1)$. 

Since $0 \le V_n \le V_{n,1} +V_{n,2}$, we conclude that $V_n \xrightarrow{p} 0$ for all $\alpha > \max\{ \tfrac12(2m-1), \tfrac12(m+3)\}$.  By Slutsky's theorem,
$
[\alpha^{-1} \Gamma_m(\alpha)]^2 \tr \left[(\sqrt{n} \, M_n)^2 \right] \cdot V_n \xrightarrow{d} 0;
$ 
and therefore 
$
[\alpha^{-1} \Gamma_m(\alpha)]^2 \tr \left[(\sqrt{n} \, M_n)^2 \right] \cdot V_n \xrightarrow{p} 0.
$ 
Hence, by (\ref{zn_zn1}), 
$\lVert \mathcal{Z}_n - \mathcal{Z}_{n,1} \rVert_{L^2} \xrightarrow{p} 0$, for $\alpha > \max\{ \tfrac12(2m-1), \tfrac12(m+3)\}$.

To establish (\ref{nulldistri_3}), define $V_j := g(T,X_j) - \ g(T)$ for $T > 0$ and $j=1,\dotsc,n$. Then it is straightforward to verify that 
$$
\mathcal{Z}_{n,1}-\mathcal{Z}_{n,2} = \Gamma_m(\alpha) \bigg\langle M_n, \frac{1}{\sqrt{n}} \sum_{j=1}^{n} V_j \bigg\rangle
$$
and therefore
\begin{equation}
\label{bound_zn1_zn2}
\lVert \mathcal{Z}_{n,1}-\mathcal{Z}_{n,2} \rVert^2_{L^2}  \le [\Gamma_m(\alpha)]^2 \tr (M_n^2)  \cdot \int_{T > 0} \bigg[\tr \bigg( \frac{1}{\sqrt{n}} \sum_{j=1}^{n} V_j \bigg)^2 \ \bigg] \, \dd P_0(T).
\end{equation}
By the Law of Large Numbers and the Continuous Mapping theorem, $\tr (M_n^2) \xrightarrow{p} 0$.  Since $g(T)= E[g(T,X_j)]$ then $E(V_j)=0$, $j=1,\dotsc,n$; also, $V_1, \dotsc, V_n$ are i.i.d.  

We now show that $E_{X_j} E_{T} \rVert V_j \lVert^2_F < \infty$.  First, 
$$
E_{X_j} E_{T} ( \rVert V_j \lVert^2_F ) = E_{X_j} \bigg( \int_{T>0} \lVert g(T,X_j) - \ g(T) \rVert^2_F \ \dd P_0(T) \bigg).
$$
By the triangle inequality, 
\begin{align*}
\lVert g(T,X_j) - \ g(T) \rVert^2_F & \le  \bigg( \lVert g(T,X_j) \rVert_F + \lVert g(T) \rVert_F \bigg)^2 \\
 & \le  2 \bigg( \lVert g(T,X_j) \rVert^2_F + \lVert g(T) \rVert^2_F \bigg).
\end{align*}
Therefore, it suffices to show that
$E_{X_j} E_T \lVert g(T,X_j) \rVert^2_F$ 
and 
$E_T \lVert g(T) \rVert^2_F$ 
are finite.  

Applying the sub-multiplicative property of the Frobenius norm, and the inequality (\ref{norm_bound_nabla_bessel_2matrixargument}), we have 
\begin{align*}
\lVert g(T,X_j) \rVert^2_F &=\lVert X_j^{1/2} \, \nabla A_{\nu}(T, \alpha^{-1} X_j) X_j^{1/2} \rVert^2_F \\
&\le \, \lVert X_j \lVert^2_F \ \lVert \nabla A_{\nu}(T, \alpha^{-1} X_j) \rVert^2_F \\
&= 
c \, (\tr X^2_j) (\lambda_\min (X_j))^{-1} (\tr T^2) (\lambda_\min (T))^{-1},
\end{align*}
$c > 0$; therefore, 
$$
E_{X_j} E_T \lVert g(T,X_j) \rVert^2_F \le c \,
E_{X_j} \left[(\tr X^2_j) \, (\lambda_\min (X_j))^{-1}\right] \, E_T \left[(\tr T^2) (\lambda_\min (T))^{-1}\right].
$$
By Lemma \ref{integral_trace_lambdamin}, $E_T \left[(\tr T^2) (\lambda_\min (T))^{-1}\right] < \infty$ for $\alpha > \tfrac12(m+1)$.  Since $X_j \sim W_m(\alpha, I_m)$, $j=1,\dotsc,n$, then the same holds for $E_{X_j} \left[(\tr X^2_j) \, (\lambda_\min (X_j))^{-1}\right]$.  
Therefore, it follows that $E_{X_j} E_T \lVert g(T,X_j) \rVert^2_F < \infty$ for all $\alpha > \tfrac12(2m-1)$.  

To show that $E_T \lVert g(T) \rVert^2_F < \infty$, $T \sim W_m(\alpha,I_m)$, we observe that $\lVert g(T) \rVert^2_F = \tr[(g(T))^2]$ is a polynomial in $T$ and therefore its expectation is finite since the moment-generating function of $T$ exists.  


Next, we vectorize the matrices $V_1, \dotsc, V_n$ and denote the corresponding vectors by $\vect(V_1),\dotsc, \vect(V_n)$.  Then, $\vect(V_1),\dotsc, \vect(V_n)$ are i.i.d. zero-mean random vectors with finite covariance matrices.  By the multivariate Central Limit Theorem, $n^{-1/2} \sum_{j=1}^n \vect(V_j)$ converges in distribution to a multivariate normal random vector.  Define 
$$
\mathcal{V} (T):= \bigg\lVert \frac{1}{\sqrt{n}} \sum_{j=1}^{n} V_j \bigg\rVert_F, 
$$
for $T > 0$; we regard $\mathcal{V}$ as a random element in $L^2$. Since $ \lVert \cdot \rVert_F$ is a continuous function, it follows from the Continuous Mapping theorem that $\mathcal{V}$ converges to a random element in $L^2$ and also that 
\begin{align*}
\lVert \mathcal{V} \rVert^2_{L^2} & := \int_{T > 0} \bigg\lVert \frac{1}{\sqrt{n}} \sum_{j=1}^{n} V_j \bigg\rVert_F^2 \ \dd P_0(T) \\
&=\int_{T > 0} \tr \bigg(\frac{1}{\sqrt{n}} \sum_{j=1}^{n} V_j \bigg)^2 \ \dd P_0(T)
\end{align*}
converges in distribution to a random variable that has finite variance.  Since $\tr(M_n^2) \xrightarrow{p} 0$, by (\ref{bound_zn1_zn2}) then, by Slutsky's theorem, we obtain $\lVert \mathcal{Z}_{n,1}-\mathcal{Z}_{n,2} \rVert^2_{L^2} \xrightarrow{d} 0$; therefore $\lVert \mathcal{Z}_{n,1}-\mathcal{Z}_{n,2} \rVert_{L^2} \xrightarrow{p} 0$. 

To establish (\ref{nulldistr_4}), we observe that
\begin{align*}
\mathcal{Z}_{n,2}-\mathcal{Z}_{n,3}&= \frac{\Gamma_m(\alpha)}{\sqrt{n}} \sum_{j=1}^n \bigg( \langle M_n, g(T) \rangle - \langle \alpha^{-1}(\alpha I_m-X_j) ,g(T) \rangle \bigg)\\
&= \frac{\Gamma_m(\alpha)}{\sqrt{n}} \bigg( \langle n M_n, g(T) \rangle - \langle \alpha^{-1} \sum_{j=1}^n (\alpha I_m-X_j) ,g(T) \rangle \bigg)\\
&= \Gamma_m(\alpha) \tr \bigg[ \bigg( \bar{X}_n^{-1/2} \sqrt{n} (\alpha I_m-\bar{X}_n) \bar{X}_n^{-1/2} - \alpha^{-1} \sqrt{n}  (\alpha I_m-\bar{X}_n) \bigg) g(T) \bigg].
\end{align*}
Substituting the now-familiar explicit formula for $g(T)$ from \eqref{gLambdaformula}, 
we obtain
\begin{multline*}
\lVert \mathcal{Z}_{n,2}-\mathcal{Z}_{n,3} \rVert^2_{L^2} 
= \frac{1}{\alpha^2 m^2} \ \bigg[ \tr \bigg( \bar{X}_n^{-1/2} \sqrt{n} (\alpha I_m-\bar{X}_n) \bar{X}_n^{-1/2} - \alpha^{-1} \sqrt{n}  (\alpha I_m-\bar{X}_n) \bigg) \bigg]^2 \\
\times \int_{T > 0} (\tr T)^2 \etr(-2\alpha^{-1} T) \ \dd P_0(T),
\end{multline*}
and as we have seen before, the latter integral 
is finite.  

Now, we observe that
\begin{align*}
\bar{X}_n^{-1/2} \sqrt{n} (\alpha I_m-\bar{X}_n) \bar{X}_n^{-1/2} - \alpha^{-1} \sqrt{n} (\alpha I_m-\bar{X}_n) 
\equiv \sqrt{n}  (\alpha I_m-\bar{X}_n) (\bar{X}_n^{-1}-\alpha^{-1} I_m).
\end{align*}
By the multivariate Central Limit Theorem, $ \sqrt{n} \vect (\alpha I_m-\bar{X}_n)$ converges in distribution to a multivariate normal random vector; and by the Law of Large Numbers for random vectors, $\bar{X}_n^{-1} \xrightarrow{p} \alpha^{-1} I_m$.
By Slutsky's theorem, $\sqrt{n} (\alpha I_m-\bar{X}_n) (\bar{X}_n^{-1}-\alpha^{-1} I_m) \xrightarrow{d} 0$, and so $\sqrt{n} (\alpha I_m-\bar{X}_n) (\bar{X}_n^{-1}-\alpha^{-1} I_m) \xrightarrow{p} 0$. Hence, by the Continuous Mapping Theorem,
$$
\bigg[ \tr \bigg( \bar{X}_n^{-1/2} \sqrt{n} (\alpha I_m-\bar{X}_n) \bar{X}_n^{-1/2} - \alpha^{-1} \sqrt{n}  (\alpha I_m-\bar{X}_n) \bigg) \bigg]^2 \xrightarrow{p} 0;
$$
and so $\lVert \mathcal{Z}_{n,2}-\mathcal{Z}_{n,3} \rVert_{L^2} \xrightarrow{p} 0$. 

Finally, by the Continuous Mapping Theorem in $L^2$ (\cite[p.~67]{chowteicher}, \cite[p.~31]{ref21}), $\lVert \mathcal{Z}_{n} \rVert^2_{L^2} \xrightarrow{d} \lVert \mathcal{Z} \rVert^2_{L^2}$, i.e., 
$$
\boldsymbol{T}^2_{n}= \int_{T > 0} {\mathcal{Z}^2_{n}(T)}\, \dd P_0(T) \xrightarrow{d} \int_{T>0} {\mathcal{Z}^2(T)}\, \dd P_0(T).
$$
The proof now is complete.
$\qed$

\subsection{Eigenvalues and eigenfunctions of the covariance operator}
\label{sec:eigen_matrix}

The covariance operator $\mathcal{S}:L^2 \rightarrow L^2$ of the random element $\mathcal{Z}$ is defined for $S > 0$ and $f \in L^2$ by 
$$
\mathcal{S} f(S) = \int_{S > 0} K(S, T) f(T) \, \dd P_0(T),
$$
where $K(S, T)$ is the covariance function defined in equation (\ref{covariancefn_matrixcase}).  Let $\{\delta_k: k \ge 1\}$ be the positive eigenvalues, listed in non-increasing order according to their multiplicities, of $\mathcal{S}$; also, let $\{\chi^2_{1k}: k \ge 1\}$ be i.i.d. $\chi^2_{1}$-distributed random variables.  It is well-known that the integrated squared process, $\int_{T > 0} \mathcal{Z}^2(T) \, \dd P_0(T)$, has the same distribution as $\sum_{k=1}^\infty \delta_k\chi^2_{1k}$.  This result follows from the Karhunen-Lo\'eve expansion of the Gaussian random field $\mathcal{Z}(T)$; see Le Ma{\^\i}tre and Knio \cite[Chapter 2]{lemaitreknio} or Vakhania \cite[p.~58]{ref22}.  Therefore, the limiting null distribution of $\boldsymbol{T}^2_n$ is the same as $\sum_{k=1}^\infty \delta_k \chi^2_{1k}$.  Let us also denote by $\tilde{\delta}_k$, $k \ge 1$, an enumeration, listed in non-increasing order, of the {\it distinct} values of the eigenvalues $\delta_k$.  Further, we denote by $N(\tilde{\delta}_k)$ the corresponding multiplicities of the distinct eigenvalues $\tilde{\delta}_k$.  Then, $\boldsymbol{T}_n^2$ converges in distribution to $\sum_{k \ge 1} \tilde{\delta}_k \chi^2_{N(\tilde{\delta}_k)}$, where $\{\chi^2_{N(\tilde{\delta}_k)} \}$ are i.i.d. $\chi^2_{N(\tilde{\delta}_k)}$-distributed random variables.

For $S, T > 0$, define 
\begin{equation}
\label{K0kernel_matrixcase}
K_0(S, T)=\Gamma_m(\alpha) \ \etr(-\alpha^{-1}(S+T)) \ A_{\nu} (-\alpha^{-2} S, T), 
\end{equation}
the first term in the covariance function defined in equation (\ref{covariancefn_matrixcase}); by (\ref{boundredcov_matrixcase}) and (\ref{firsttermcovariance_matrix}),
\begin{align}
\label{K0integral_matrix}
K_0(S,T)= [\Gamma_m(\alpha)]^2 \int_{X >0} { A_{\nu}(S, \alpha^{-1}X) \ A_{\nu}(T, \alpha^{-1}X)}\, \dd P_0(X). 
\end{align}

We will first find the eigenvalues and eigenfunctions of the integral operator $\mathcal{S}_0: L^2 \rightarrow L^2$, defined for $S > 0$ and $f$ in $L^2$ by 
\begin{align}
\label{defeigen_matrixcase}
\mathcal{S}_0 f(S)=\int_{T > 0} {K_0(S,T) f(T)}\, \dd P_0(T).
\end{align}

Recall that $m \ge 2$ and $\alpha > \max\{ \tfrac12(2m-1), \tfrac12(m+3) \}$. Throughout the remainder of this work, we use the notation
\begin{equation}
\label{betaandbalpha_matrixcase}
\beta = \bigg(\frac{\alpha+4}{\alpha}\bigg)^{1/2} \quad \hbox{and} \quad b_{\alpha} = \big(1+\tfrac12\alpha(1-\beta)\big)^{1/2}.
\end{equation}
We also set 
\begin{equation}
\label{defrho_matrixcase}
\rho_{\kappa}=\alpha^{m\alpha} b_{\alpha}^{4|\kappa|+2m\alpha}
\end{equation}
for $\kappa$ ranging over all partitions, and
\begin{equation}
\label{deflen_matrixcase}
\textgoth{L}_{\kappa}^{(\nu)} (S) := \beta^{m\alpha/2} \ \etr \big( (1-\beta) S/2 \big)  \ \mathcal{L}_{\kappa}^{(\nu)} \big( \beta S \big).
\end{equation}

\begin{theorem}
\label{theoremreducedcov_matrixcase}
The collection $\{(\rho_\kappa,\textgoth{L}_{\kappa}^{(\nu)})\}$, where $\kappa$ ranges over the set of all partitions, is a complete enumeration of the eigenvalues and eigenfunctions, respectively, of the oerator $\mathcal{S}_0$.  Further, the eigenfunctions $\{\textgoth{L}_{\kappa}^{(\nu)}\}$, for $\kappa$ ranging over all partitions, form an orthonormal basis in $L^2$, and $\mathcal{S}_0$ is positive and of trace-class.
\end{theorem}

\begin{proof}
Recall from \cite[p.~290, Problem 7.21]{muirhead} the \textit{Poisson kernel}: For $r \in (0,1)$ and $X , Y > 0$, 
\begin{multline}
\label{poissonkernel_matrixcase}
A_{\nu}\bigg(-\frac{r}{(1-r)^2} X, Y \bigg) \\
= (1-r)^{m \alpha} \etr \Bigg(\frac{r}{1-r} (X+Y) \Bigg) 
\frac{1}{\Gamma_m(\alpha)} \ \sum_{k=0}^{\infty} \sum_{|\kappa|=k} \mathcal{L}_{\kappa}^{(\nu)} (X) \mathcal{L}_{\kappa}^{(\nu)} (Y) \ r^k.
\end{multline}
In this expansion, set 
\begin{equation}
\label{r_matrixcase}
r = b_{\alpha}^4 = \big(1+\tfrac12\alpha(1-\beta)\big)^2,
\end{equation}
so that $r \in (0,1)$.  Note that $r^{1/2} = 1+\tfrac12\alpha(1-\beta)$ satisfies the quadratic equation 
$$
r - (\alpha+2)r^{1/2} + 1 = 0
$$
and also that this equation is equivalent to the identity 
\begin{align}
\frac{1-r}{\alpha r^{1/2}} &= 1 + \frac{2}{\alpha}(1 - r^{1/2}). \nonumber \\
\intertext{On the right-hand side of this identity, substitute for $r^{1/2}$ in terms of $\alpha$ and $\beta$ to obtain }
\label{defbeta_matrixcase}
\frac{1-r}{\alpha r^{1/2}} &= 1 + \frac{2}{\alpha}\big[1 - \big(1+\tfrac12\alpha(1-\beta)\big)\big] = \beta.
\end{align}
In (\ref{poissonkernel_matrixcase}), also set 
\begin{equation*}
X = \frac{1-r}{\alpha r^{1/2}} S \equiv \beta S \quad \hbox{and} \quad Y = \frac{1-r}{\alpha r^{1/2}} T \equiv \beta T.
\end{equation*}
Then, 
\begin{equation}
\label{xyridentities_matrixcase}
\frac{r(X+Y)}{1-r} = \frac{(r^{1/2}-1)(S+T)}{\alpha} + \frac{(S+T)}{\alpha}.
\end{equation}
Applying (\ref{defrho_matrixcase}),(\ref{deflen_matrixcase}) and (\ref{r_matrixcase})-(\ref{xyridentities_matrixcase}) to (\ref{poissonkernel_matrixcase}), and substituting the result in (\ref{K0kernel_matrixcase}), we obtain for $S, T > 0$, the pointwise convergent series expansion, 
\begin{equation}
\label{reducedcovariancepoisson2_matrixcase}
K_0(S, T) = \sum_{k=0}^{\infty} \sum_{|\kappa|=k} \rho_{\kappa} \, \textgoth{L}_{\kappa}^{(\nu)} (S) \textgoth{L}_{\kappa}^{(\nu)} (T).
\end{equation}
By (\ref{laguerreorthog_matrixargument}), the generalized Laguerre polynomials $\{\mathcal{L}_{\kappa}^{(\nu)}\}$ form an orthonormal system; then it is straightforward to verify that the system $\{\textgoth{L}_{\kappa}^{(\nu)}\}$ also is orthonormal in $L^2$, for $\kappa$ ranging over all partitions, i.e., 
\begin{equation}
\label{orthonormal_matrixcase}
\int_{S>0} {\textgoth{L}_{\kappa}^{(\nu)} (S) \textgoth{L}_{\sigma}^{(\nu)} (S)}\, \dd P_0(S)=
\begin{cases}
    1, & \kappa=\sigma \\
    0, & \kappa \neq \sigma
  \end{cases}
\end{equation}

Now we verify that the series (\ref{reducedcovariancepoisson2_matrixcase}) converges in the separable tensor product Hilbert space $L^2 \otimes L^2:=L^2 (P_0 \times P_0)$. By the Cauchy criterion, it suffices to prove that for each $\epsilon > 0$, there exists $N \in \mathbb{N}$ such that 
\begin{align*}
\int_{\mathcal{P}_{+}^{m \times m} \times \mathcal{P}_{+}^{m \times m}} {\bigg[ \sum_{k=l_1}^{l_2}\sum_{|\kappa|=k} \rho_{\kappa} \textgoth{L}_{\kappa}^{(\nu)} (S) \textgoth{L}_{\kappa}^{(\nu)} (T) \bigg]^2}\, \dd (P_0 \otimes P_0) (S, T) < \epsilon,
\end{align*}
for all $l_1, l_2 \in \mathbb{N}$ such that $l_2 \ge l_1 \ge N$. By squaring the integrand, it suffices by Fubini's theorem to consider 
\begin{align*}
&\sum_{k=l_1}^{l_2}  \ \int_{S> 0} \int_{T>0} { \bigg[ \sum_{|\kappa|=k} \rho_{\kappa} \textgoth{L}_{\kappa}^{(\nu)} (S) \textgoth{L}_{\kappa}^{(\nu)} (T) \bigg]^2} \, \dd P_0(T) \, \dd P_0(S) \\
& + 2 \sum_{k_1=l_1}^{l_2-1} \sum _{k_2=l_1+1}^{l_2} \ \int_{S > 0} \int_{T>0} { \bigg[ \sum_{|\kappa|=k_1} \rho_{\kappa} \textgoth{L}_{\kappa}^{(\nu)} (S) \textgoth{L}_{\kappa}^{(\nu)} (T) \bigg]} \\
& \qquad\qquad\qquad\qquad\qquad\qquad \times 
{\bigg[\sum_{|\kappa|=k_2} \rho_{\kappa} \textgoth{L}_{\kappa}^{(\nu)} (S) \textgoth{L}_{\kappa}^{(\nu)} (T) \bigg] }\, \dd P_0 (T) \ \dd P_0(S).
\end{align*}
Since the system $\{\textgoth{L}_{\kappa}^{(\nu)}\}$ is orthonormal, the latter sum reduces to 
\begin{align*}
\sum_{k=l_1}^{l_2} \sum_{|\kappa|=k} \rho_{\kappa}^2 &= \alpha^{2m\alpha} b_{\alpha}^{4m\alpha} \, \sum_{k=l_1}^{l_2} b_{\alpha}^{8k}\sum_{|\kappa|=k} 1 \\
&= \alpha^{2m\alpha} b_{\alpha}^{4m\alpha} \ \sum_{k=l_1}^{l_2} b_{\alpha}^{8k} p_m(k),
\end{align*}
where $p_m(k)$ represents the number of partitions of $k$ into at most $m$ parts.  It is well-known that 
$$
\sum_{k=0}^{\infty} b_{\alpha}^{8k} p_m(k) = \prod_{k=1}^m (1-b_{\alpha}^{8k})^{-1}.
$$
Therefore, $\sum_{k=0}^{\infty} b_{\alpha}^{8k} p_m(k)$ is a convergent series. Since every convergent series in any metric space is Cauchy, it follows that for each $\epsilon >0$, there exists $N \in \mathbb{N}$ such that $\sum_{k=l_1}^{l_2} b_{\alpha}^{8k} p_m(k) < \epsilon$, for all $l_1, l_2 \in \mathbb{N}$ such that $l_2 \ge l_1 \ge N$. Therefore, the series (\ref{reducedcovariancepoisson2_matrixcase}) is Cauchy in $L^2 \otimes L^2$ and hence, 
$$
\lim_{l \rightarrow \infty} \int_{\mathcal{P}_{+}^{m \times m} \times \mathcal{P}_{+}^{m \times m}}  {\bigg[ K_0(S, T)} 
{-\sum_{k=0}^{l} \sum_{|\kappa|=k} \rho_{\kappa} \textgoth{L}_{\kappa}^{(\nu)} (S) \textgoth{L}_{\kappa}^{(\nu)} (T)  \bigg] ^2}\, \dd (P_0 \otimes P_0) (S, T) =0.
$$
By Fubini's theorem, the latter expression equals 
\begin{equation}
\label{seriesoconvarianceconv_matrixcase}
\lim_{l \rightarrow \infty} \int_{S>0} \int_{T>0}  {\bigg[ K_0(S, T)} 
-\sum_{k=0}^{l} \sum_{|\kappa|=k} \rho_{\kappa} \textgoth{L}_{\kappa}^{(\nu)}(S) 
\textgoth{L}_{\kappa}^{(\nu)}(T) \bigg] ^2 \, \dd P_0(T) \, \dd P_0(S) = 0.
\end{equation}
It follows from the orthonormality, (\ref{orthonormal_matrixcase}), of the system $\{\textgoth{L}_{\kappa}^{(\nu)}\}$ that, for $l \in \mathbb{N}$ and partition $\sigma$ such that $l \ge |\sigma|$, 
\begin{equation}
\label{propertyeigen_matrixcase}
\int_{T > 0} \sum_{k=0}^{l} \ \sum_{|\kappa|=k}  \rho_{\kappa} \textgoth{L}_{\kappa}^{(\nu)} (S) \textgoth{L}_{\kappa}^{(\nu)} (T) \  \textgoth{L}_{\sigma}^{(\nu)}(T) \, \dd P_0 (T) = \rho_{\sigma}  \textgoth{L}_{\sigma}^{(\nu)}(S).
\end{equation}
By (\ref{defeigen_matrixcase}) and (\ref{propertyeigen_matrixcase}),
\begin{align*}
\int_{S > 0} & { \bigg| \mathcal{S}_0\textgoth{L}_{\sigma}^{(\nu)}(S)-\rho_{\sigma}  \textgoth{L}_{\sigma}^{(\nu)}(S)\bigg| }\, \dd P_0(S) \nonumber\\
&= \int_{S > 0} {\bigg| \int_{T >0}{ \bigg[ K_0(S, T)- \sum_{k=0}^{l}  \ \sum_{|\kappa|=k} \rho_{\kappa} \textgoth{L}_{\kappa}^{(\nu)} (S) \textgoth{L}_{\kappa}^{(\nu)} (T) \bigg] }}
{{\textgoth{L}_{\sigma}^{(\nu)}(T)}\, \dd P_0 (T)  \bigg| }\, \dd P_0(S).\nonumber
\end{align*}
By the Cauchy-Schwarz inequality, this latter expression is bounded by 
\begin{multline}
\label{propertyeigen2_matrixcase}
\bigg( \int_{S > 0} { \int_{T > 0}{ \bigg| K_0(S, T)}} 
{{- \sum_{k=0}^{l} \sum_{|\kappa|=k} \rho_{\kappa}\textgoth{L}_{\kappa}^{(\nu)} (S) \textgoth{L}_{\kappa}^{(\nu)} (T) \bigg|^2}}\, \dd P_0(T) \ \dd P_0(S) \bigg)^{1/2} \\
\times  \bigg( \int_{S > 0} \int_{T > 0} {\bigg| \textgoth{L}_{\sigma}^{(\nu)}(T) \bigg|^2 }\, \dd P_0 (T) \ \dd P_0(S) \bigg)^{1/2}.
\end{multline}
By the orthonormality property (\ref{orthonormal_matrixcase}) and the fact that $P_0$ is a probability distribution, the second term in (\ref{propertyeigen2_matrixcase}) equals 1; therefore, 
\begin{multline}
\label{propertyeigen3_matrixcase}
\int_{S > 0} { \bigg| \mathcal{S}_0\textgoth{L}_{\sigma}^{(\nu)}(S)-\rho_{\sigma}  \textgoth{L}_{\sigma}^{(\nu)}(S)\bigg| }\, \dd P_0(S) \\
\le \bigg( \int_{S > 0} { \int_{T > 0}{ \bigg| K_0(S, T)}} 
{{- \sum_{k=0}^{l} \sum_{|\kappa|=k} \rho_{\kappa}\textgoth{L}_{\kappa}^{(\nu)} (S) \textgoth{L}_{\kappa}^{(\nu)} (T) \bigg|^2}}\, \dd P_0(T) \ \dd P_0(S) \bigg)^{1/2}.
\end{multline}
Since $l$ is arbitrary, we now let $l \rightarrow \infty$. By (\ref{seriesoconvarianceconv_matrixcase}), the right-hand side of (\ref{propertyeigen3_matrixcase}) converges to $0$, so we obtain 
$$
\int_{S > 0} { \bigg| \mathcal{S}_0\textgoth{L}_{\sigma}^{(\nu)}(S)-\rho_{\sigma}  \textgoth{L}_{\sigma}^{(\nu)}(S)\bigg| }\, \dd P_0(S)=0,
$$
which proves that $\mathcal{S}_0\textgoth{L}_{\sigma}^{(\nu)}(S)=\rho_\sigma \textgoth{L}_{\sigma}^{(\nu)}(S)$, for $P_0$-almost every $S$. Therefore, $\rho_{\kappa}$ is an eigenvalue of $S_0$ with corresponding eigenfunction $\textgoth{L}_{\kappa}^{(\nu)}$.

Since the kernel $K_0(S,T)$ is symmetric in $(S,T)$, it follows that $\mathcal{S}_0$ is symmetric. To show that $\mathcal{S}_0$ is positive, we observe that for $f \in L^2$, 
\begin{align*}
\langle \mathcal{S}_0 f, f \rangle_{L^2}&=\int_{S >0} {\mathcal{S}_0 f(S) \overline{f(S)}}\, \dd P_0(S) \nonumber\\
&=\int_{S>0} {\bigg[ \int_{T>0} {K_0(S, T) \ f(T)}\, \dd P_0(T)\bigg] \overline{f(S)}}\, \dd P_0(S).
\end{align*}
Substituting for $K_0(S,T)$ from (\ref{K0integral_matrix}), we obtain
\begin{multline*}
\langle \mathcal{S}_0 f, f \rangle_{L^2} = \int_{S >0} { \bigg[ \int_{T>0} { \bigg( \int_{X> 0} {[\Gamma_m(\alpha)]^2 A_{\nu}(S, \alpha^{-1}X) }}} \\
\times {{{A_{\nu}(T, \alpha^{-1}X)}\, \dd P_0(X) \bigg) f(T) } \dd P_0(T) \bigg] \overline{f(S)} }\, \dd P_0(S).
\end{multline*}
Applying Fubini's theorem to reverse the order of the integration, we find that the inner integrals with respect to $S$ and $T$ are complex conjugates of each other; therefore, 
\begin{align}
\label{innerproduct_matrixcase}
\langle \mathcal{S}_0 f, f \rangle_{L^2}&=\int_{X>0} { \bigg| \int_{S >0} { \Gamma_m(\alpha) A_{\nu}(S, \alpha^{-1}X) f(S) }\, \dd P_0(S)  \bigg|^2}\, \dd P_0(X),
\end{align}
which is positive. Thus, $\mathcal{S}_0$ is positive. 

Next, we prove that $\mathcal{S}_0$ is of trace-class. For $f \in L^2$, $S > 0$, it again follows by (\ref{K0integral_matrix}) and Fubini's theorem that
\begin{align}
\label{S0traceclass_matrix}
\mathcal{S}_0 f(S) &= \int_{T>0} {K_0(S, T) \ f(T)}\, \dd P_0(T) \nonumber\\
&=\int_{X >0} { \int_{T>0} { [\Gamma_m(\alpha)]^2 A_{\nu}(T, \alpha^{-1}X) f(T)}\, \dd P_0(T) } {A_{\nu}(S, \alpha^{-1}X)}\, \dd P_0(X).
\end{align}
Denote by $\mathcal{T}_0:L^2 \rightarrow L^2$ the integral operator,
$$
\mathcal{T}_0f(T)=\int_{X>0} {\Gamma_m(\alpha) \ A_{\nu}(T, \alpha^{-1}X) f(X)}\, \dd P_0(X),
$$
$T > 0$. By (\ref{2besselineq_matrixargument}), $|\Gamma_m(\alpha) \ A_{\nu}(T, \alpha^{-1}X)|\, \le 1$ and therefore 
$$
\left\lvert \Gamma_m(\alpha) \ A_{\nu}(T, \alpha^{-1}X)\right\rvert_{L^2 \otimes L^2}^2 < \infty,
$$
for $T, X > 0$.  By \cite[p.~93, Theorem 8.8]{ref17}, it follows that $\mathcal{T}_0$ is a Hilbert-Schmidt operator. Now, we can write (\ref{S0traceclass_matrix}) as
\begin{align*}
\mathcal{S}_0 f(S) &= \int_{X >0} \mathcal{T}_0 f(X) \left[\Gamma_m(\alpha) \, A_{\nu}(S, \alpha^{-1}X)\right] \, \dd P_0(X) \\
&= \mathcal{T}_0(\mathcal{T}_0 f)(S),
\end{align*}
$S > 0$, which proves that $\mathcal{S}_0$ is of trace-class.

To complete the proof, we now show that the set $\{\textgoth{L}_{\kappa}^{(\nu)}\}$ is complete. It is sufficient to show that if $f \in L^2$ with $\langle f, \textgoth{L}_{\kappa}^{(\nu)} \rangle _{L^2}=0$ for all partitions $\kappa$, then $f=0$ $P_0$-almost everywhere. First, we note that 
\begin{align}
\label{completeness_matrixcase}
\int_{S>0} & { \bigg| (\mathcal{S}_0 f)(S) \overline{f(S)}-\sum_{k=0}^{l} \sum_{|\kappa|=k} \rho_{\kappa} \langle f, \textgoth{L}_{\kappa}^{(\nu)} \rangle _{L^2} \ \textgoth{L}_{\kappa}^{(\nu)} (S) \overline{f(S)} \bigg| }\, \dd P_0(S)\nonumber\\
= \ & \int_{S > 0} {\bigg| \int_{T>0} { \bigg[ K_0(S, T)-\sum_{k=0}^{l} \sum_{|\kappa|=k} \rho_{\kappa} \textgoth{L}_{\kappa}^{(\nu)} (S) \ \textgoth{L}_{\kappa}^{(\nu)} (T)  \bigg]}} {{ f(T) \overline{f(S)}} \, \dd P_0(T) \bigg| }\, \dd P_0(S) \nonumber\\
\le \ & \bigg( \int_{S > 0} {\int_{T>0} { \bigg| K_0(S, T) }} {{-\sum_{k=0}^{l} \sum_{|\kappa|=k} \rho_{\kappa} \textgoth{L}_{\kappa}^{(\nu)} (S) \ \textgoth{L}_{\kappa}^{(\nu)} (T)  \bigg|^2}}\, \dd P_0(T) \dd P_0(S) \bigg)^{1/2} \nonumber\\
& \qquad\qquad\qquad\qquad\qquad \times  \bigg( \int_{S > 0} {\int_{T>0} {|f(S)|^2 \  |f(T)|^2}}\, \dd P_0(T) \dd P_0 (S) \bigg)^{1/2},
\end{align}
by the Cauchy-Schwarz inequality. Since $f \in L^2$, the second term on the right-hand side of (\ref{completeness_matrixcase}) is finite. Taking the limit on both sides of  (\ref{completeness_matrixcase}) as $l \rightarrow \infty$ and applying (\ref{seriesoconvarianceconv_matrixcase}), we obtain
\begin{equation}
\label{completeness2_matrixcase}
\lim_{l \rightarrow \infty} \int_{S>0} { \bigg| (\mathcal{S}_0 f)(S) \overline{f(S)} } 
- {\sum_{k=0}^{l} \sum_{|\kappa|=k} \rho_{\kappa} \langle f, \textgoth{L}_{\kappa}^{(\nu)} \rangle _{L^2} \ \textgoth{L}_{\kappa}^{(\nu)} (S) \overline{f(S)} \bigg| }\, \dd P_0(S)=0.
\end{equation}
Since $\langle f, \textgoth{L}_{\kappa}^{(\nu)} \rangle _{L^2}=0$ for all partitions $\kappa$ then (\ref{completeness2_matrixcase}) reduces to 
$$
\langle \mathcal{S}_0 f, f \rangle_{L^2}=\int_{S>0} {(\mathcal{S}_0 f)(S) \overline{f(S)}}\, \dd P_0(S)=0.
$$

Therefore, by (\ref{innerproduct_matrixcase}), we obtain for $P_0$-almost every $X$, 
\begin{align}
\label{completenesseigen2_matrixcase}
\int_{S >0} { \Gamma_m(\alpha) A_{\nu}(S, \alpha^{-1}X) f(S) }\, \dd P_0(S)=0.
\end{align}
Since the function $\Gamma_m(\alpha) A_{\nu}(S, \alpha^{-1}X)$ is continuous for all $X >0$ and fixed $S >0$ and by (\ref{2besselineq_matrixargument}),
$$
|\Gamma_m(\alpha) A_{\nu}(S, \alpha^{-1}X)| \ \le \ 1, 
$$
for $X, S > 0$, then by the Dominated Convergence Theorem, the integral on the left-hand side of (\ref{completenesseigen2_matrixcase}) is a continuous function of $X$. If two continuous functions are equal $P_0$-almost everywhere then they are equal everywhere; hence (\ref{completenesseigen2_matrixcase}) holds for all $X > 0$. 

Henceforth, without loss of generality, we assume that $f$ is real-valued. Let $f^{+}$ and $f^{-}$ denote the positive and negative parts of $f$, respectively. Then, $f=f^{+}-f^{-}$, $f^{+}$ and $f^{-}$ are nonnegative, and since $f \in L^2$ then by the Cauchy-Schwarz inequality, $f^{+}$ and $f^{-}$ are $P_0$-integrable. Also, by (\ref{completenesseigen2_matrixcase}), 
$$
\int_{S >0} { \Gamma_m(\alpha) A_{\nu}(S, \alpha^{-1}X) f^{+}(S) }\, \dd P_0(S) = \int_{S >0} { \Gamma_m(\alpha) A_{\nu}(S, \alpha^{-1}X) f^{-}(S) }\, \dd P_0(S),
$$
$X > 0$. By the Uniqueness Theorem for orthogonally invariant Hankel transforms, Theorem \ref{uniqueness_hankeltilde}, we notice that there are only two possible cases. Either
$$
\int_{S>0}{f^{+} (S)}\, \dd P_0(S)=\int_{S>0}{f^{-} (S)}\, \dd P_0(S)=0,
$$
or
$$
\int_{S>0}{f^{+} (S)}\, \dd P_0(S)=\int_{S>0}{f^{-} (S)}\, \dd P_0(S)=C > 0.
$$
For the first case, we have $f^{+}=f^{-}=0$ and so $f=0$ $P_0$-almost everywhere. 
As for the second case, we have
$$
\int_{S >0} { 
A_{\nu}(S, \alpha^{-1}X) C^{-1} f^{+}(S)}\, \dd P_0(S) 
=\int_{S >0} { 
A_{\nu}(S, \alpha^{-1}X) C^{-1} f^{-}(S)}\, \dd P_0(S),
$$
$X > 0$. By the Uniqueness Theorem for orthogonally invariant Hankel transforms, we obtain $f^{+}=f^{-}$ and hence $f=0$ $P_0$-almost everywhere. This proves that the orthonormal set $\{\textgoth{L}_{\kappa}^{(\nu)}\}$ is complete, and therefore, it forms a basis in the separable Hilbert space $L^2$. 
\end{proof}

\medskip

The proof of the following theorem is similar to the proof of Theorem \ref{theoremreducedcov_matrixcase}, and the
complete details are provided by Hadjicosta \cite{hadjicosta19}.

\begin{theorem} 
\label{properties_covariance_operator_matrix}
Let $\mathcal{S}:L^2 \rightarrow L^2$ be the covariance operator of the random element $\mathcal{Z}$ defined as
$$
\mathcal{S} f(S)=\int_{T > 0} {K(S,T) f(T)}\, \dd P_0(T)
$$
for all $S > 0$ and for all functions $f$ in $L^2$, where $K(S, T)$ is the covariance function defined in equation (\ref{covariancefn_matrixcase}). Then, $\mathcal{S}$ is positive and of trace-class.
\end{theorem}

Recall here that a non-trivial function $\phi \in L^2$ is an \textit{eigenfunction} of $\mathcal{S}$ if there exists an \textit{eigenvalue} $\delta \in \bC$ such that $\mathcal{S}\phi = \delta \phi$.  As $\mathcal{S}$ is self-adjoint and positive, its eigenvalues are real and nonnegative.  In the next result, we find the positive eigenvalues (that are not eigenvalues of $\mathcal{S}_0$) and corresponding eigenfunctions of the operator $\mathcal{S}$, and we will show in Subsection \ref{decayrate_matrix} that $0$ is not an eigenvalue of $\mathcal{S}$.  

\begin{theorem}
\label{thmeigenS_matrix}
Let $\delta \in \mathbb{R}$ with $\delta \neq \rho_{\kappa}$ for any partition $\kappa$. Also, denote by $\tilde{\rho}_k$, $k \ge 1$, an enumeration, listed in non-increasing order, of the {\it distinct} values of the eigenvalues $\rho_\kappa$ and define the functions
\begin{align*}
A(\delta)&=1- \beta^{m\alpha} m \ \sum_{k=0}^{\infty} \frac{ (m\alpha)_k}{k!( \tilde{\rho}_k-\delta)} \tilde{\rho}_k^2,\\
B(\delta)&=1-\alpha \beta^{m\alpha} m \sum_{k=0}^{\infty} \frac{ (m\alpha)_k}{ k! (\tilde{\rho}_k -\delta)} \tilde{\rho}_k^2 \ (b_{\alpha}^{2} -m^{-1} k \beta)^2,\\
\intertext{and}
D(\delta)&=\alpha^{2} \beta^{m\alpha} m \sum_{k=0}^{\infty}  \frac{ (m\alpha)_k}{k ! (\tilde{\rho}_k -\delta)} \tilde{\rho}_k^2 \ ( b_{\alpha}^{2} -m^{-1} k\beta).
\end{align*}
Then, the positive eigenvalues of $\mathcal{S}$ are the positive roots of $G(\delta) = \alpha^3 A(\delta)B(\delta) - D^2(\delta)$. The eigenfunction corresponding to an eigenvalue $\delta$ has Fourier-Laguerre expansion
$$
\beta^{m\alpha/2} \ \sum_{k=0}^{\infty} \frac{ \tilde{\rho}_k }{\sqrt{k!}(\tilde{\rho}_k -\delta)} \big(C_1+C_2 \alpha^{-1}(b_{\alpha}^{2} -m^{-1} k \beta)\big) \sum_{|\kappa|=k}  \big( C_{\kappa} (I_m) \ \big[ \alpha \big]_{\kappa}  \big)^{1/2}  \ \textgoth{L}_{\kappa}^{(\nu)},
$$
where $C_1C_2 \neq 0$, $\alpha^3 C_1A(\delta)=C_2 D(\delta)$ and  $C_2B(\delta)=C_1D(\delta)$.
\end{theorem}

\begin{proof}
Since the set $\{\textgoth{L}_{\kappa}^{(\nu)}\}$, for $\kappa$ ranging over all partitions, is an orthonormal basis for $L^2$, the eigenfunction $\phi \in L^2$ corresponding to an eigenvalue $\delta$ can be written as
$$
\phi=\sum_{k=0}^{\infty} \sum_{|\kappa|=k} \langle \phi, \textgoth{L}_{\kappa}^{(\nu)} \rangle_{L^2} \ \textgoth{L}_{\kappa}^{(\nu)}.
$$
We restrict ourselves temporarily to eigenfunctions for which this series is pointwise convergent.  Substituting this series into the equation $\mathcal{S} \phi =\delta \phi$, we obtain 
\begin{equation}
\label{eigen_matrixcase}
\int_{T>0} {K(S,T) \sum_{k=0}^{\infty} \sum_{|\kappa|=k} \langle \phi, \textgoth{L}_{\kappa}^{(\nu)} \rangle_{L^2} \textgoth{L}_{\kappa}^{(\nu)} (T) }\, \dd P_0(T) 
= \delta \ \sum_{k=0}^{\infty} \sum_{|\kappa|=k} \langle \phi, \textgoth{L}_{\kappa}^{(\nu)} \rangle_{L^2} \textgoth{L}_{\kappa}^{(\nu)}(S).
\end{equation}
Substituting the covariance function $K(S,T)$ in the left-hand side of (\ref{eigen_matrixcase}), writing $K$ in terms of $K_0$, and assuming that we can interchange the order of integration and summation, we obtain
\begin{align}
\label{equationeigen_matrixcase}
\sum_{k=0}^{\infty} & \sum_{|\kappa|=k} \langle \phi, \textgoth{L}_{\kappa}^{(\nu)} \rangle_{L^2} \nonumber\\ 
\times & \int_{T>0} { \bigg[ K_0(S,T) -\etr(-\alpha^{-1} (S+T)) (\alpha^{-3} m^{-1} (\tr S) (\tr T) +1) \bigg]} 
{\textgoth{L}_{\kappa}^{(\nu)} (T) }\, \dd P_0(T) \nonumber\\
&=\delta \ \sum_{k=0}^{\infty} \sum_{|\kappa|=k} \langle \phi, \textgoth{L}_{\kappa}^{(\nu)} \rangle_{L^2} \textgoth{L}_{\kappa}^{(\nu)}(S).
\end{align}

By Theorem \ref{theoremreducedcov_matrixcase}, 
$$
\int_{T> 0} {K_0(S,T) \textgoth{L}_{\kappa}^{(\nu)} (T)}\, \dd P_0(T)=\rho_{\kappa} \textgoth{L}_{\kappa}^{(\nu)} (S). 
$$
On writing $\textgoth{L}_{\kappa}^{(\nu)}$ in terms of $L_{\kappa}^{(\nu)}$, the generalized Laguerre polynomial, applying (\ref{laguerreintegral_matrixcase}) for the Laplace transform of $L_{\kappa}^{(\nu)}$, and making use of (\ref{r_matrixcase}) and (\ref{defbeta_matrixcase}), we obtain
\begin{align}
\langle \etr(-\alpha^{-1} T), \textgoth{L}_{\kappa}^{(\nu)} \rangle_{L^2} &:= \int_{T >0} {\etr(-\alpha^{-1} T) \textgoth{L}_{\kappa}^{(\nu)} (T) }\, \dd P_0(T) \nonumber\\
\label{integral1_eigenS_matrixcase}
&= \bigg(\frac{ C_{\kappa} (I_m) \ \big[ \alpha \big]_{\kappa} }{| \kappa | !} \bigg)^{1/2} \ \beta^{m\alpha/2} \rho_{\kappa}.
\end{align}

Again writing $\textgoth{L}_{\kappa}^{(\nu)}$ in terms of $L_{\kappa}^{(\nu)}$, applying \eqref{integrallemmalaguerre_matrixcase}, and making use of (\ref{r_matrixcase}) and (\ref{defbeta_matrixcase}), we obtain
\begin{align}
\langle \etr(-\alpha^{-1} T) (\tr T), \textgoth{L}_{\kappa}^{(\nu)} \rangle_{L^2} &: =\int_{T>0} {\etr(-\alpha^{-1} T) (\tr T) \textgoth{L}_{\kappa}^{(\nu)} (T) }\, \dd P_0(T)\nonumber\\
\label{secondintegral2_matrixcase}
&= \bigg( \frac{ C_{\kappa} (I_m) \ \big[ \alpha \big]_{\kappa} }{|\kappa!|} \bigg)^{1/2}  \ \alpha^2 \beta^{m \alpha/2} \rho_{\kappa} (m b_{\alpha}^{2} -|\kappa| \beta).
\end{align}
In summary, (\ref{equationeigen_matrixcase}) reduces to
\begin{multline*}
\sum_{k=0}^{\infty} \sum_{|\kappa|=k} \rho_{\kappa} \langle \phi, \textgoth{L}_{\kappa}^{(\nu)} \rangle_{L^2}  \bigg[ \textgoth{L}_{\kappa}^{(\nu)}(S) \\
-\etr(-\alpha^{-1} S)  \bigg( \frac{C_{\kappa} (I_m) \ \big[ \alpha \big]_{\kappa} }{|\kappa!|} \bigg)^{1/2} \beta^{m\alpha/2} (\alpha^{-1}m^{-1} (\tr S) (m b_{\alpha}^{2} -|\kappa| \beta) +1 )\bigg]\\
=\delta \ \sum_{k=0}^{\infty} \sum_{|\kappa|=k} \langle \phi, \textgoth{L}_{\kappa}^{(\nu)} \rangle_{L^2} \textgoth{L}_{\kappa}^{(\nu)}(S).
\end{multline*}

By applying (\ref{integral1_eigenS_matrixcase}), we obtain the Fourier-Laguerre expansion of $\etr(-\alpha^{-1}S)$ with respect to the orthonormal basis $\{ \textgoth{L}_{\kappa}^{(\nu)} \}$; indeed,
\begin{align*}
\etr(-\alpha^{-1} S) &=\sum_{k=0}^{\infty} \sum_{|\kappa|=k} \langle \etr(-\alpha^{-1} S), \textgoth{L}_{\kappa}^{(\nu)} \rangle_{L^2} \ \textgoth{L}_{\kappa}^{(\nu)} (S)\\
&= \beta^{m\alpha/2} \sum_{k=0}^{\infty} \sum_{|\kappa|=k} \bigg( \frac{C_{\kappa} (I_m) \ \big[ \alpha \big]_{\kappa} } {| \kappa | !} \bigg)^{1/2} \ \rho_{\kappa} \ \textgoth{L}_{\kappa}^{(\nu)} (S).
\end{align*}
Similarly, by applying (\ref{secondintegral2_matrixcase}), we have 
\begin{align*}
\etr(-\alpha^{-1} S) (\tr S)
&= \sum_{k=0}^{\infty} \sum_{|\kappa|=k} \langle \etr(-\alpha^{-1} S) (\tr S), \textgoth{L}_{\kappa}^{(\nu)} \rangle_{L^2} \ \textgoth{L}_{\kappa}^{(\nu)} (S)\nonumber\\
&=\alpha^2 \beta^{m \alpha/2} \sum_{k=0}^{\infty} \sum_{|\kappa|=k} \bigg( \frac{C_{\kappa} (I_m) \ \big[ \alpha \big]_{\kappa} } {| \kappa | !} \bigg)^{1/2} \ \rho_{\kappa} (m b_{\alpha}^{2} -|\kappa| \beta) \ \textgoth{L}_{\kappa}^{(\nu)} (S).
\end{align*}

Let
\begin{align}
\label{C1_matrixcase}
C_1&=\int_{T>0} {\etr(-\alpha^{-1} T) \phi(T)}\, \dd P_0(T) \nonumber\\
&=\beta^{m\alpha/2} \sum_{k=0}^{\infty} \sum_{|\kappa|=k} \langle \phi, \textgoth{L}_{\kappa}^{(\nu)} \rangle_{L^2} \ \bigg( \frac{C_{\kappa} (I_m) \ \big[ \alpha \big]_{\kappa} } {| \kappa | !} \bigg)^{1/2} \ \rho_{\kappa},
\end{align}
and
\begin{align}
\label{C2_matrixcase}
C_2&:=\int_{T>0} {\etr(-\alpha^{-1} T) (\tr T) \phi(T) }\, \dd P_0(T)\nonumber\\
&=\alpha^2 \beta^{m \alpha/2} \sum_{k=0}^{\infty} \sum_{|\kappa|=k} \langle \phi, \textgoth{L}_{\kappa}^{(\nu)} \rangle_{L^2}  \ \bigg( \frac{C_{\kappa} (I_m) \ \big[ \alpha \big]_{\kappa} } {| \kappa | !} \bigg)^{1/2} \ \rho_{\kappa} (m b_{\alpha}^{2} -|\kappa| \beta).
\end{align}

Combining (\ref{equationeigen_matrixcase})-(\ref{C2_matrixcase}), we find that (\ref{eigen_matrixcase}) reduces to 
\begin{align*}
& \delta \ \sum_{k=0}^{\infty} \sum_{|\kappa|=k} \langle \phi, \textgoth{L}_{\kappa}^{(\nu)} \rangle_{L^2} \textgoth{L}_{\kappa}^{(\nu)}(S) \nonumber\\
& \quad = \sum_{k=0}^{\infty} \sum_{|\kappa|=k} \rho_{\kappa}  \bigg[ \langle  \phi, \textgoth{L}_{\kappa}^{(\nu)} \rangle_{L^2} \nonumber\\
& \qquad\qquad\qquad\qquad - \beta^{m\alpha/2} \bigg( \frac{C_{\kappa} (I_m) \ \big[ \alpha \big]_{\kappa} } {| \kappa | !} \bigg)^{1/2}  (C_1 +C_2 \alpha^{-1}(b_{\alpha}^{2} -m^{-1} |\kappa| \beta)) \bigg]  \ \textgoth{L}_{\kappa}^{(\nu)} (S),
\end{align*}
and by comparing the coefficients of $\textgoth{L}_{\kappa}^{(\nu)} (S)$, we obtain
$$
\delta \langle \phi, \textgoth{L}_{\kappa}^{(\nu)} \rangle_{L^2} 
=\rho_{\kappa}  \bigg[ \langle  \phi, \textgoth{L}_{\kappa}^{(\nu)} \rangle_{L^2}- \beta^{m\alpha/2} \bigg( \frac{C_{\kappa} (I_m) \ \big[ \alpha \big]_{\kappa} } {| \kappa | !} \bigg)^{1/2}  (C_1 +C_2 \alpha^{-1}(b_{\alpha}^{2} -m^{-1} |\kappa| \beta)) \bigg],
$$
for all partitions $\kappa$. Since we have assumed that $\delta \neq \rho_{\kappa}$ for any $\kappa$ then we can solve the equation for $\langle \phi, \textgoth{L}_{\kappa}^{(\nu)} \rangle_{L^2}$ to obtain
\begin{equation}
\label{coefficients_eigen2_matrixcase}
\langle \phi, \textgoth{L}_{\kappa}^{(\nu)} \rangle_{L^2} 
= \beta^{m\alpha/2} \frac{\rho_{\kappa}}{\rho_{\kappa}-\delta} \bigg( \frac{C_{\kappa} (I_m) \ \big[ \alpha \big]_{\kappa} } {| \kappa | !} \bigg)^{1/2}  (C_1 +C_2 \alpha^{-1}(b_{\alpha}^{2} -m^{-1} |\kappa| \beta)).
\end{equation}

Substituting (\ref{coefficients_eigen2_matrixcase}) into (\ref{C1_matrixcase}), and applying Lemma \ref{sum_zonalfactorial}, we get
\begin{align*}
C_1 &= C_1 \beta^{m\alpha} \ \sum_{k=0}^{\infty} \sum_{|\kappa|=k} \frac{ C_{\kappa} (I_m) \ \big[ \alpha \big]_{\kappa}}{| \kappa | ! (\rho_{\kappa} -\delta)}\rho_{\kappa}^2 \\
& \qquad
+ C_2 \alpha^{-1} \beta^{m\alpha} \ \sum_{k=0}^{\infty} \sum_{|\kappa|=k} \frac{ C_{\kappa} (I_m) \ \big[ \alpha \big]_{\kappa}}{| \kappa | ! (\rho_{\kappa}-\delta)} \rho_{\kappa}^2 \ (b_{\alpha}^{2} -m^{-1} |\kappa| \beta)\\
&= C_1(1-m^{-1}A(\delta))+C_2 \alpha^{-3} m^{-1} D(\delta);
\end{align*}
therefore,
\begin{align}
\label{C1C2_eq_matrixcase}
\alpha^3 C_1 A(\delta)=C_2 D(\delta).
\end{align}

Similarly, by substituting (\ref{coefficients_eigen2_matrixcase}) into (\ref{C2_matrixcase}) and applying Lemma \ref{sum_zonalfactorial}, we get
\begin{align*}
C_2 &= C_1 \ \alpha^{2} \beta^{m\alpha} m \sum_{k=0}^{\infty} \sum_{|\kappa|=k} \frac{ C_{\kappa} (I_m) \ \big[ \alpha \big]_{\kappa}}{| \kappa | ! (\rho_{\kappa} -\delta)} \rho_{\kappa}^2 \ ( b_{\alpha}^{2} -m^{-1} |\kappa| \beta) \\
& \quad\quad\quad\quad +C_2 \ \alpha \beta^{m\alpha} m \sum_{k=0}^{\infty} \sum_{|\kappa|=k} \frac{ C_{\kappa} (I_m) \ \big[ \alpha \big]_{\kappa}}{| \kappa | ! (\rho_{\kappa}-\delta)} \rho_{\kappa}^2 \ (b_{\alpha}^{2} -m^{-1} |\kappa| \beta)^2\\
&=C_1 D(\delta)+C_2(1-B(\delta));
\end{align*}
hence
\begin{align}
\label{C1C2_eq2_matrixcase}
C_2 B(\delta)=C_1 D(\delta).
\end{align}

Suppose $C_1=C_2=0$; then it follows from (\ref{coefficients_eigen2_matrixcase}) that $\langle \phi, \textgoth{L}_{\kappa}^{(\nu)} \rangle_{L^2}=0$ for all partitions $\kappa$, which implies that $\phi=0$, which is a contradiction since $\phi$ is a non-trivial eigenfunction. Hence, $C_1$ and $C_2$ cannot be both equal to 0. 

Combining (\ref{C1C2_eq_matrixcase}) and (\ref{C1C2_eq2_matrixcase}, and using the fact that $C_1$ and $C_2$ are not both $0$, it is straightforward to establish that $\alpha^3 A(\delta) B(\delta)= D^2(\delta)$ : If $C_1 \neq 0$ and $C_2 \neq 0$, then we obtain $\alpha^3 C_1 C_2 A(\delta) B(\delta)= C_1 C_2 D^2(\delta)$ so $\alpha^3 A(\delta) B(\delta)= D^2(\delta)$. If $C_1=0$ and $C_2 \neq 0$, then we obtain $D(\delta)=B(\delta)=0$ and again $\alpha^3 A(\delta) B(\delta)= D^2(\delta)$ is true. If $C_1 \neq 0$ and $C_2=0$, then we obtain $D(\delta)=A(\delta)=0$ and again $\alpha^3 A(\delta) B(\delta)= D^2(\delta)$ is true. Therefore, if $\delta$ is a positive eigenvalue of $\mathcal{S}$ then it is a positive root of the function 
$G(\delta) = \alpha^3 A(\delta) B(\delta)-D^2(\delta)$.  

Conversely, suppose that $\delta$ is a positive root of $G(\delta)$ with $\delta \neq \rho_{\kappa}$ for any partition $\kappa$. Define
\begin{align}
\label{gamman_matrixcase}
\gamma_{\kappa}:= \beta^{m\alpha/2} \bigg( \frac{C_{\kappa} (I_m) \ \big[ \alpha \big]_{\kappa} } {| \kappa | !} \bigg)^{1/2} \frac{ \rho_{\kappa}}{\rho_{\kappa} -\delta}  \bigg(C_1+C_2 \alpha^{-1}(b_{\alpha}^{2} -m^{-1} |\kappa| \beta)\bigg),
\end{align}
where $C_1$ and $C_2$ are real constants that are not both equal to 0 and which satisfy (\ref{C1C2_eq_matrixcase}) and (\ref{C1C2_eq2_matrixcase}). That such constants exist can be shown by following a case-by-case argument similar to \cite[p.~48]{ref27}: If $D \neq 0$, $A \neq 0$, and $B \neq 0$, then we can choose $C_2$ to be any non-zero number then set $C_1= C_2 B/D$. If $D=0$, $A=0$, and $B \neq 0$, then we can choose $C_1$ to be any non-zero number and then set $C_2=0$. If $D=0$, $A \neq 0$, and $B=0$, then we can choose $C_2$ to be any non-zero number and then set $C_1=0$. Last, if $D=0$, $A=0$, and $B=0$, then we can choose $C_1$ and $C_2$ to be any non-zero numbers.

Now define, for $S > 0$, the function 
\begin{align}
\label{eigenfunctionfull_matrixcase}
\tilde{\phi} (S) =\sum_{k=0}^{\infty} \sum_{|\kappa|=k} \gamma_{\kappa} \ \textgoth{L}_{\kappa}^{(\nu)} (S). 
\end{align}
By applying the ratio test, we obtain $\sum_{k=0}^{\infty} \sum_{|\kappa|=k}  \gamma_{\kappa}^2 < \infty$; therefore $\tilde{\phi} \in L^2$. 

We also verify that the series (\ref{eigenfunctionfull_matrixcase}) converges pointwise. By (\ref{laguerre2_matrixcase}) and (\ref{deflen_matrixcase}),
$$
|\textgoth{L}_{\kappa}^{(\nu)} (S) | = \beta^{m\alpha/2} \ \etr ((1-\beta) S/2) \ ( | \kappa | ! \ C_{\kappa} (I_m) \ \big[ \alpha \big]_{\kappa} )^{-1/2} \ | L^{\nu}_{\kappa} (\beta S )|,
$$
$S > 0$. By inequality (\ref{boundlaguerre_matrixcase}), 
$$
| L^{(\nu)}_{\kappa} ( \beta S )| \ \le \etr(\beta S) \ C_{\kappa}(I_m) \ \big[ \alpha \big]_{\kappa},
$$
$S > 0$. Therefore,  
\begin{align}
\label{ineqeigen_matrixcase}
| \textgoth{L}_{\kappa}^{(\nu)} (S) | &\le \ \beta^{m\alpha/2} \ \etr ((1+\beta)S/2) \ \bigg( \frac{C_{\kappa}(I_m) \ \big[ \alpha \big]_{\kappa}}{ | \kappa | !} \bigg)^{1/2}.
\end{align}
Thus, to establish the pointwise convergence of the series (\ref{eigenfunctionfull_matrixcase}), we need to show that
\begin{align}
\label{seriesconvpointwise_matrixcase}
\sum_{k=0}^{\infty} \sum_{|\kappa|=k} \bigg( \frac{C_{\kappa}(I_m) \ \big[ \alpha \big]_{\kappa}}{ | \kappa | !} \bigg)^{1/2} \ |\gamma_{\kappa}| \ < \infty .
\end{align}
The convergence of the above series follows from the ratio test. 

Next, we justify the interchange of summation and integration in our calculations. By a corollary to Theorem 16.7 in Billingsley \cite[p.~224]{billi2}, we need to verify that
\begin{align}
\label{interchange_sumint_matrixcase}
\sum_{k=0}^{\infty} \sum_{|\kappa|=k} | \gamma_{\kappa} | \int_{S>0} {  K(S,T) \ | \textgoth{L}_{\kappa}^{(\nu)} (T) | }\, \dd P_0(T) < \infty.
\end{align}
First, we find a bound for $K_0(S,T)$. By (\ref{2besselineq_matrixargument}), 
$|\Gamma_m(\alpha) A_{\nu} (-\alpha^{-2}S, T) | \ \le 1$ for
$S, T > 0$. Thus, by (\ref{K0kernel_matrixcase}),
\begin{align}
\label{ineqK0_matrixcase}
0 \le K_0(S,T) \le \etr(-\alpha^{-1}(S+ T))
\end{align}
By the triangle inequality and by (\ref{ineqK0_matrixcase}), we have 
\begin{align*}
0 \le K(S,T) & \le K_0(S,T) + \etr(-\alpha^{-1}(S+ T))(\alpha^{-3}m^{-1} (\tr S) (\tr T) +1) \\
& \le \etr(-\alpha^{-1}(S+T)) \ ( 2+ \alpha^{-3}m^{-1} (\tr S) (\tr T)).
\end{align*}
Thus, to prove (\ref{interchange_sumint_matrixcase}), we need to establish that 
$$
\sum_{k=0}^{\infty} \sum_{|\kappa|=k} | \gamma_{\kappa} | 
\int_{T >0} { \etr(-\alpha^{-1}T) \, ( 2+ \alpha^{-3}m^{-1} (\tr S) (\tr T)) \ | \textgoth{L}_{\kappa}^{(\nu)} (T) | }\, \dd P_0(T) < \infty.
$$
By applying the bound (\ref{ineqeigen_matrixcase}), we see that it suffices to prove that 
$$
\sum_{k=0}^{\infty} \sum_{|\kappa|=k} \bigg( \frac{C_{\kappa}(I_m) \ \big[ \alpha \big]_{\kappa}}{ | \kappa | !} \bigg)^{1/2} \ | \gamma_{\kappa} | \int_{T>0} { \etr(-\alpha^{-1}T)\ \etr ((1+\beta)T/2) }\, \dd P_0(T) < \infty,
$$
and 
$$
\sum_{k=0}^{\infty} \sum_{|\kappa|=k} \bigg( \frac{C_{\kappa}(I_m) \ \big[ \alpha \big]_{\kappa}}{ | \kappa | !} \bigg)^{1/2} \ | \gamma_{\kappa} | 
\int_{T>0} { (\tr T) \etr(-\alpha^{-1}T) \ \etr ((1+\beta)T/2) }\, \dd P_0(T) < \infty.
$$
As these integrals are finite, the convergence of both series follows from (\ref{seriesconvpointwise_matrixcase}).

To calculate $\mathcal{S} \tilde{\phi}(S)$ from (\ref{eigenfunctionfull_matrixcase}), we follow the same steps as before to obtain
\begin{align*}
\mathcal{S} \tilde \phi (S)&=\int_{S>0} {K(S,T) \sum_{k=0}^{\infty} \sum_{|\kappa|=k}\gamma_{\kappa} \textgoth{L}_{\kappa}^{(\nu)} (T)}\, \dd P_0(T) \\
&=\sum_{k=0}^{\infty} \sum_{|\kappa|=k} \rho_{\kappa} \gamma_{\kappa} \textgoth{L}_{\kappa}^{(\nu)}(S) 
- C_1 \beta^{m\alpha/2} \sum_{k=0}^{\infty} \sum_{|\kappa|=k} \bigg( \frac{C_{\kappa} (I_m) \ \big[ \alpha \big]_{\kappa} } {| \kappa | !} \bigg)^{1/2} \ \rho_{\kappa} \textgoth{L}_{\kappa}^{(\nu)}(S) \\
& \qquad - C_2 \alpha^{-1} \beta^{m \alpha/2} \sum_{k=0}^{\infty} \sum_{|\kappa|=k} \bigg( \frac{C_{\kappa} (I_m) \ \big[ \alpha \big]_{\kappa} } {| \kappa | !} \bigg)^{1/2} \ \rho_{\kappa} (b_{\alpha}^{2} -m^{-1} |\kappa| \beta) \textgoth{L}_{\kappa}^{(\nu)} (S).
\end{align*}
By the definition (\ref{gamman_matrixcase}) of $\gamma_{\kappa}$, and noting that 
$$
\frac{\rho_{\kappa}}{\rho_{\kappa} -\delta} -1 =\frac{\delta}{\rho_{\kappa} -\delta},
$$
we have 
\begin{align*}
\mathcal{S} \tilde \phi (S)&=\beta^{m \alpha/2} \sum_{k=0}^{\infty} \sum_{|\kappa|=k}  \bigg[\frac{\rho_{\kappa}}{\rho_{\kappa} -\delta} -1 \bigg]  \bigg(\frac{C_{\kappa} (I_m) \ \big[ \alpha \big]_{\kappa} } {| \kappa | !} \bigg)^{1/2} \rho_{\kappa} \\
& \qquad\qquad\qquad\qquad \times \bigg(C_1+C_2 \alpha^{-1}(b_{\alpha}^{2} -m^{-1} |\kappa| \beta)\bigg) \textgoth{L}_{\kappa}^{(\nu)}(S)\\
&=\beta^{m \alpha/2} \delta \sum_{k=0}^{\infty} \sum_{|\kappa|=k} \frac{\rho_{\kappa}^2}{\rho_{\kappa} -\delta}  \bigg(\frac{C_{\kappa} (I_m) \ \big[ \alpha \big]_{\kappa} } {| \kappa | !} \bigg)^{1/2} 
\bigg(C_1+C_2 \alpha^{-1}(b_{\alpha}^{2} -m^{-1} |\kappa| \beta)\bigg) \textgoth{L}_{\kappa}^{(\nu)}(S)\\
&=\delta \sum_{k=0}^{\infty} \sum_{|\kappa|=k}  \gamma_{\kappa} \textgoth{L}_{\kappa}^{(\nu)}(S)\\
&=\delta \tilde{\phi} (S).
\end{align*}
Therefore, $\delta$ is an eigenvalue of $\mathcal{S}$ with corresponding eigenfunction $\tilde{\phi}$. 
\end{proof}

\begin{remark}
{\rm
In \cite{hadjicostarichards}, where we studied goodness-of-fit testing for the gamma distributions, we have conjectured that the eigenvalues of $\mathcal{S}$ are not eigenvalues of $\mathcal{S}_0$.  However, as shown in the next subsection, this is not valid in the case of the Wishart distributions.
}
\end{remark}

\subsection{An interlacing property of the eigenvalues}
\label{decayrate_matrix}

A difficulty of the eigenvalues $\delta_k$ is that they have no closed form expression; hence there is no simple formula for $N$, the number of terms in the truncated series $\sum_{k=1}^N  \delta_k \chi^2_{1k}$ that should be used in practice to approximate the asymptotic distribution, $\sum_{k=1}^{\infty} \delta_k \chi^2_{1k}$, of the test statistic $\boldsymbol{T}_n^2$.  

Since $\mathcal{S}_0$ is of trace-class then, by \cite[p.~237, Corollary 3.2]{brislawn}, $Tr(\mathcal{S}_0)$ can be calculated by integrating the kernel $K_0$ or by evaluating the sum of all eigenvalues $\rho_{\kappa}$: 
\begin{align}
\label{TraceS0_matrix}
\int_{S > 0} { K_0(S,S)}\, \dd P_0(S) &= Tr(\mathcal{S}_0) \nonumber \\
& = \sum_{k=0}^{\infty} \sum_{|\kappa|=k} \rho_{\kappa} = \alpha^{m\alpha} b_{\alpha}^{2m\alpha} \prod_{k=1}^{m} (1-b_{\alpha}^{4k})^{-1}.
\end{align}
Since $\mathcal{S}$ also is of trace-class then 
\begin{align}
\label{TraceS1_matrix}
\sum_{k=1}^{\infty} \delta_k &= Tr(\mathcal{S}) 
= \int_{S > 0} K(S, S) \, \dd P_0(S) \nonumber \\
&= \int_{S >0} \bigg [ K_0(S, S) - (\alpha^{-3} m^{-1}(\tr S)^2 +1) \etr(-2 \alpha^{-1} S)  \bigg] \, \dd P_0(S) \nonumber \\
&=\alpha^{m\alpha} b_{\alpha}^{2m\alpha} \prod_{k=1}^{m} (1-b_{\alpha}^{4k})^{-1}-\alpha^{-3} m^{-1} \sum_{|\kappa|=2} \int_{S >0} {\etr(-2 \alpha^{-1} S) C_{\kappa} (S) }\, \dd P_0(S) \nonumber\\
& \quad\quad -\int_{S >0} {\etr(-2 \alpha^{-1} S)}\, \dd P_0(S).
\end{align}
All of these integrals can be evaluated using (\ref{zonalintegral}) and (\ref{gammaintegral}), and the resulting sum can be simplified using Lemma \ref{sum_zonalfactorial}.  Consequently, we obtain 
\begin{align}
\label{sum_of_delta_k}
\sum_{k=1} ^{\infty} \delta_k 
&= \alpha^{m\alpha} b_{\alpha}^{2m\alpha} \prod_{k=1}^{m} (1-b_{\alpha}^{4k})^{-1} - \bigg(\frac{\alpha}{\alpha+2}\bigg)^{m\alpha} \bigg(1 + \frac{m\alpha+1}{(\alpha+2)^2} \bigg).
\end{align}

To determine the number of terms in the truncated series $\sum_{k=1}^N \delta_k \chi^2_{1k}$ that should be used in practice to approximate the asymptotic distribution of $\boldsymbol{T}_n^2$, we derive bounds for the eigenvalues $\delta_k$ in terms of the $\rho_\kappa$ and then obtain a general formula for $N$ as a function of $\alpha$. We refer to the ratio $(\sum_{k=1}^N  \delta_k)/Tr(\mathcal{S})$ as the $N$th {\it scree ratio} for $\boldsymbol{T}_n^2$.  

Since the operator $\mathcal{S}$ is compact and positive then the set of all its eigenvalues is countable and contains only nonnegative values \cite[Theorem 8.12, p.~98]{ref17}. The next result shows that the eigenvalues indeed are positive.

\bigskip

\begin{proposition}
\label{injectiveoperators}
The operators $\mathcal{S}$ and $\mathcal{S}_0$ are injective; that is, $\mathcal{S}f = \mathcal{S}g$ if and only if $f = g$, and the same holds for $\mathcal{S}_0$.  In particular, $0$ is not an eigenvalue of $\mathcal{S}$ or $\mathcal{S}_0$.  
\end{proposition}

\Pro
By linearity, it suffices to assume that $g = 0$.  So, suppose that $\mathcal{S}f = 0$, that is, 
$$
\int_{S > 0} K(S, T) f(T) \, \dd P_0(T) = 0
$$
for all $S > 0$.  Then for $U > I_m$, by Fubini's theorem,
\begin{align}
\label{eq:injective1}
0 &= \int_{ S > 0} \etr(-(U-I_m) S/\alpha) (\det S)^{\alpha -\tfrac12(m+1)} \int_{T > 0} K(S, T) f(T) \, \dd P_0(T) \, \dd S \nonumber \\
&= \int_{ T > 0} f(T) \bigg[ \int_{S > 0} \etr(-(U-I_m) S/\alpha) (\det S)^{\alpha -\tfrac12(m+1)} K(S, T) \dd S \bigg] \, \dd P_0(T),
\end{align}
By the definition of the covariance function $K$ in (\ref{covariancefn_matrixcase}), 
\begin{align*}
\int_{S > 0} & \etr(-(U-I_m) S/\alpha) (\det S)^{\alpha -\tfrac12(m+1)} K(S, T) \dd S \nonumber\\
& = \etr( -T/\alpha) \int_{S > 0} {\etr(-US/\alpha) (\det S)^{\alpha -\tfrac12(m+1)} }\nonumber\\
& \qquad\qquad\qquad\quad\quad \times{ \bigg [\Gamma_m(\alpha) A_{\nu} (-\alpha^{-2} S, T)-\alpha^{-3}m^{-1} (\tr S) (\tr T) -1 \bigg]}\, \dd S.
\end{align*}
By (\ref{bessel_2matrixargument}), (\ref{besselintegral}), and Fubini's theorem, we have 
\begin{align*}
\int_{S > 0} \etr&(-US/\alpha) (\det S)^{\alpha -\tfrac12(m+1)}  A_{\nu} (-\alpha^{-2} S, T) \, \dd S\\
&=\int_{O(m)} {\int_{S > 0} {\etr(-US/\alpha) (\det S)^{\alpha -\tfrac12(m+1)}  A_{\nu} (-\alpha^{-2} H S H' T)}\, \dd S}\, \, \dd H \\
&=\int_{O(m)} {\etr(\alpha^{-1} H' T H U^{-1}) (\det (\alpha^{-1} U))^{-\alpha}}\, \dd H\\
&= \alpha^{m\alpha} (\det U)^{-\alpha} \int_{O(m)} {\etr(\alpha^{-1} H' T H U^{-1})}\, \dd H.
\end{align*}
Also, by (\ref{trace_zonal}) and (\ref{zonalintegral}), we have
\begin{align*}
\int_{S > 0} {\etr(-US/\alpha)  (\det S)^{\alpha -\tfrac12(m+1)} (\tr S)}\, \dd S 
&= \alpha^{m\alpha+2} \Gamma_m(\alpha) (\det U)^{-\alpha} \tr(U^{-1}),
\end{align*}
and, by (\ref{gammaintegral}), 
\begin{equation*}
\int_{S > 0} {\etr(-US/\alpha)  (\det S)^{\alpha -\tfrac12(m+1)}}\, \dd S=\alpha^{m\alpha} \Gamma_m(\alpha) (\det U)^{-\alpha}.
\end{equation*}

Substituting these results into (\ref{eq:injective1}) and discarding extraneous factors, we obtain
\begin{multline}
\label{eq:injective3}
\int_{T >0} \bigg[\int_{O(m)} {\etr(\alpha^{-1} H' T H U^{-1})}\, \dd H - \alpha^{-1}m^{-1} \tr(U^{-1}) (\tr T)-1 \bigg] \\
\times \etr(-T/\alpha) f(T) \, \dd P_0(T) = 0.
\end{multline}
Replacing $U$ by $U^{-1}$, we find that (\ref{eq:injective3}) is equivalent to 
\begin{multline}
\label{eq:injective4}
\int_{T >0} \bigg[ \int_{O(m)} {\etr(\alpha^{-1} H' T H U)}\, \dd H - 1 \bigg] \etr(-T/\alpha) f(T) \, \dd P_0(T) \\
= \alpha^{-1} m^{-1} (\tr U) \int_{T > 0} (\tr T) \etr(-T/\alpha) f(T) \, \dd P_0(T).
\end{multline}
Differentiating both sides of (\ref{eq:injective4}) with respect to $U$, we obtain 
\begin{multline*}
\int_{O(m)} \int_{T > 0} {\etr(\alpha^{-1} H' T H U) (\alpha^{-1} H' T H) f(T)}\, \dd P_0(T) \, \, \dd H \\
=\alpha^{-1} m^{-1} I_m \int_{T > 0} (\tr T) \etr(-T/\alpha) f(T) \, \dd P_0(T).
\end{multline*}
Since $T \stackrel{d}{=} H T H'$ for all $H \in O(m)$, and $f(HTH')=f(T)$, then 
$$
\int_{T > 0} {\etr(\alpha^{-1} H' T H U) (\alpha^{-1} H' T H) f(T)}\, \dd P_0(T)
=\int_{T > 0} {\etr(\alpha^{-1} U T) (\alpha^{-1} T ) f(T)}\, \dd P_0(T).
$$
Therefore, 
\begin{multline}
\label{eq:injective5}
\int_{T > 0} {\etr(\alpha^{-1} U T) (\alpha^{-1} T ) f(T)}\, \dd P_0(T) \\
= \alpha^{-1} m^{-1} I_m \int_{T > 0} (\tr T) \etr(-T/\alpha) f(T) \, \dd P_0(T).
\end{multline}
Differentiating both sides of (\ref{eq:injective5}) with respect to $U$, we find that
\begin{equation*}
\alpha^{-2} \int_{T > 0} {\etr(\alpha^{-1} U T) \ (T \otimes T) \ f(T)}\, \dd P_0(T)=0.
\end{equation*}
As this latter integral is a Laplace transform, we obtain $f = 0$, $P_0$-almost everywhere.  Also, the same argument may be used in the case of $\mathcal{S}_0$.  

Consequently, $0$ is not an eigenvalue of $\mathcal{S}$.
$\qed$

\medskip

We now derive an interlacing property of the eigenvalues $\delta_k$ and $\rho_{\kappa}$.   To state this property, denote by $\xi_k$, $k=1,2,3\ldots$ the partitions of all nonnegative integers, listed in increasing lexicographic order, e.g., $\xi_1 = (0)$, $\xi_2= (1)$, $\xi_3= (2)$, $\xi_4 = (1^2)$, $\xi_5 = (3)$, $\xi_6 = (21)$, $\xi_7 = (1^3), \ldots$  

\begin{proposition}
\label{interlacing_matrix}
For all $k \ge 1$, $\rho_{\xi_k} \ge \delta_k \ge \rho_{\xi_{k+2}}$.  
%
%
Further, for $k \ge 3$, every eigenvalue of $\mathcal{S}_0$ is an eigenvalue of $\mathcal{S}$ with multiplicity $p_m(k)-2$, $p_m(k)-1$, or $p_m(k)$.
\end{proposition}

\Pro
Define the kernels $k_0(S, T) = - \etr(-(S+T)/\alpha)$ and $$k_1(S, T) = - \alpha^{-3} m^{-1} \etr(-(S+T)/\alpha) (\tr S) (\tr T),$$ where $S, T > 0$.  Also, define on $L^2$ the corresponding integral operators, 
$$
\mathcal{U}_j f(S) = \int_{T > 0} k_j(S, T) f(T) \dd P_0(T),
$$
$j=0,1$, $S > 0$.  Then it follows from (\ref{covariancefn_matrixcase}) that $\mathcal{S} = \mathcal{S}_0 + \mathcal{U}_0 + \mathcal{U}_1$.  

It is clear that each $\mathcal{U}_j$ is self-adjoint and of rank one, i.e., the range of $\mathcal{U}_j$ is a one-dimensional subspace of $L^2$.  Also, $\mathcal{S}_0 + \mathcal{U}_0$ is self-adjoint, and by following the same steps as in Theorem \ref{properties_covariance_operator_matrix}, we see that it is positive and compact. 

By the same argument as in the proof of Proposition \ref{injectiveoperators}, we find that the operator $\mathcal{S}_0 + \mathcal{U}_0$ is injective; hence, the eigenvalues of $\mathcal{S}_0 + \mathcal{U}_0$ are positive.  

Denote by $\omega_k$, $k  \ge 1$, the eigenvalues of $\mathcal{S}_0 + \mathcal{U}_0$, where $\omega_1 \ge \omega_2 \ge \cdots$, repeated according to their multiplicities. Since $\mathcal{S}_0$ is compact, self-adjoint, and injective, and since $\mathcal{U}_0$ is self-adjoint and of rank one, it follows from Hochstadt \cite{hochstadt} or Dancis and Davis \cite{dancisdavis} that the eigenvalues of $\mathcal{S}_0$ interlace the eigenvalues of $\mathcal{S}_0 + \mathcal{U}_0$, i.e., $\rho_{\xi_1} \ge \omega_1 \ge \rho_{\xi_2} \ge \omega_2 \ge \rho_{\xi_3} \ge \omega_3 \ge \rho_{\xi_4} \ge \dotsc$.   Further, by Hochstadt \cite{hochstadt}, every eigenvalue of multiplicity $p_m(k)$, $k \ge 2$, of $\mathcal{S}_0$, where $p_m(k)$ denotes the number of partitions of $k$ in at most $m$ parts, is an eigenvalue of $\mathcal{S}_0 + \mathcal{U}_0$ with multiplicity $p_m(k)$ or $p_m(k)-1$. 

Since $\mathcal{U}_1$ is self-adjoint and of rank one then by applying again Hochstadt's, or Dancis and Davis', theorem we find that the eigenvalues of $\mathcal{S}_0 + \mathcal{U}_0$ interlace the eigenvalues of $\mathcal{S}_0 + \mathcal{U}_0 + \mathcal{U}_1 \equiv \mathcal{S}$, i.e, $\omega_k \ge \delta_k \ge \omega_{k+1}$ for all $k \ge 1$. 

Combining the conclusions of the preceding paragraphs, we deduce that $
\rho_{\xi_k} \ge \delta_k \ge \rho_{\xi_{k+2}}$, $k \ge 1$. Further, by Hochstadt \cite{hochstadt}, we have for $k \ge 3$, every eigenvalue of $\mathcal{S}_0$ is an eigenvalue of $\mathcal{S}$ with multiplicity $p_m(k)-2$, $p_m(k)-1$, or $p_m(k)$.
$\qed$

\bigskip

For $\epsilon \in (0,1)$, we can now determine a value for $N$ such that the $N$th scree ratio of $\boldsymbol{T}_n^2$ exceeds $1-\epsilon$.  Applying the interlacing inequalities for $\delta_k$, we obtain 
$\sum_{k=1}^N \delta_k \ge \sum_{2 \le |\kappa| \le r} \rho_\kappa$, where $N=\sum_{k=2}^r p_m(k)$. Since $Tr(\mathcal{S}_0) > Tr(\mathcal{S})$, we advise that $N$ be chosen so that 
$$
\sum_{0 \le |\kappa| \le r} \rho_\kappa \ge (1-\epsilon) Tr(\mathcal{S}_0).
$$
This criterion leads to a value for $N$ that is readily applicable in the analysis of data.  Substituting $\rho_\kappa = \alpha^{m\alpha} b_\alpha^{4|\kappa|+2m\alpha}$ and the value of $Tr(\mathcal{S}_0)$ from (\ref{TraceS0_matrix}), we obtain 
\begin{align}
\label{noeigenvalues_matrix}
\alpha^{m\alpha} b_\alpha^{2m\alpha} \sum_{k=0}^r b_{\alpha}^{4k} p_m(k) &\ge (1-\epsilon) Tr(\mathcal{S}_0) \nonumber \\
&= (1-\epsilon) \alpha^{m\alpha} b_\alpha^{2m\alpha} \prod_{k=1}^m (1-b_{\alpha}^{4k})^{-1}.
\end{align}

For $m=2,3$ and $\epsilon = 10^{-10}$, which represents accuracy to ten decimal places, we present in Tables \ref{screeratiotable_matrix1} and \ref{screeratiotable_matrix2} the values of the lower bounds on $r$ and $N$ for various values of $\alpha$.

\medskip

\begin{table}[!ht]
\caption{Values of the lower bounds on $r$ and $N$ for $m=2$.}
\medskip
\label{screeratiotable_matrix1}
\centering
\begin{tabular}{|r|rrrrrrr|}
  \hline
  $\alpha$ & 2.5 &  3 &  5 & 10 & 20 & 50 & 100 \\
  $r$      &   8 &  7 &  6 &  4 &  3 &  3 &   2 \\
  $N$      &  23 & 18 & 14 &  7 &  4 &  4 &   2 \\
  \hline
  \end{tabular}
\end{table}

\begin{table}[!ht]
\caption{Values of the lower bounds on $r$ and $N$ for $m=3$.}
\medskip
\label{screeratiotable_matrix2}
\centering
\begin{tabular}{|r|rrrrrrr|}
  \hline
  $\alpha$ &  3 &  4 &  5 & 10 & 20 & 50 & 100 \\
  $r$      &  8 &  7 &  6 &  4 &  3 &  3 &   2 \\
  $N$      & 39 & 29 & 21 &  9 &  5 &  5 &   2 \\
  \hline
  \end{tabular}
\end{table}

\medskip

As indicated by Tables \ref{screeratiotable_matrix1} and \ref{screeratiotable_matrix2}, fewer eigenvalues appear to be needed to approximate the distribution of $\mathcal{S}$ as $\alpha$ increases.  As we show in the following result, which is partly a consequence of the interlacing property of the eigenvalues, all but one of the $\delta_k$ and $\rho_\kappa$ converge to $0$ as $\alpha \to \infty$, a result that is consistent with the decreasing values of $r$ and $N$ in the tables.

\begin{corollary}
\label{limiting_deltas_and_rhos}
As $\alpha \to \infty$, $\rho_\kappa \to 0$ for all $\kappa \neq (0)$, $\delta_k \to 0$ for all $k \ge 2$, and $\delta_1 \to e^{-m}(1-e^{-m})$.
\end{corollary}

\Pro
By \eqref{betaandbalpha_matrixcase}, $\beta = (1 + 4\alpha^{-1})^{1/2}$.  Expanding this expression as a power series in $\alpha^{-1}$, we obtain  
$$
\alpha b_\alpha^2 = \alpha (1 + \tfrac12\alpha(1-\beta)) = 1 - \alpha^{-1} + O(\alpha^{-2}).
$$
Therefore, $(\alpha b_\alpha^2)^\alpha \to e^{-1}$ and $b_\alpha \to 0$ as $\alpha \to \infty$.  By \eqref{defrho_matrixcase}, $\rho_\kappa = (\alpha b_\alpha^2)^{m\alpha} b_\alpha^{4|\kappa|}$, so it follows that if $\kappa \neq 0$ then $\rho_\kappa \to 0$.

By Proposition \ref{interlacing_matrix}, $\delta_2 \le \rho_{(1)}$, so it follows that $\delta_2 \to 0$ as $\alpha \to \infty$.  Since the $\delta_k$ are nonnegative and listed in non-increasing order then it follows that, as $\alpha \to \infty$, $\delta_k \to 0$ for all $k \ge 2$.

Finally, the limiting value of $\delta_1$ is obtained by taking limits in (\ref{sum_of_delta_k}).
$\qed$


\subsection{An application to financial data}

In applying our test to a financial data set, we follow in part an example given by Haff, et al. \cite[Example 5.3]{minimaxest}.  Let us denote by $S_{j,k}$, for $k=1,2,3$ the daily closing stock prices of Johnson \& Johnson (JNJ), Berkshire Hathaway Inc., Class B (BRK-B), and JPMorgan Chase \& Co. (JPM) respectively, from November 26, 2017 to November 23, 2018. If a day were a trading holiday, we repeated the observation of the previous day; thus we had 260 observations in total. Then, we computed the daily logarithmic returns $\log (S_{j+1, k}/S_{j, k})$, for $j=1, \dotsc, 260$ and $k=1,2,3$; graphs of these logarithmic returns are given in Figure \ref{logreturns}.  
Finally, we partitioned the daily logarithmic returns into biweekly periods and calculated the $3 \times 3$ covariance matrix for each biweekly period, resulting in the matrices $X_1, \dotsc, X_{26}$. 

\begin{figure}[!ht]
\captionsetup{width=0.8\textwidth}
\includegraphics[width=\textwidth]{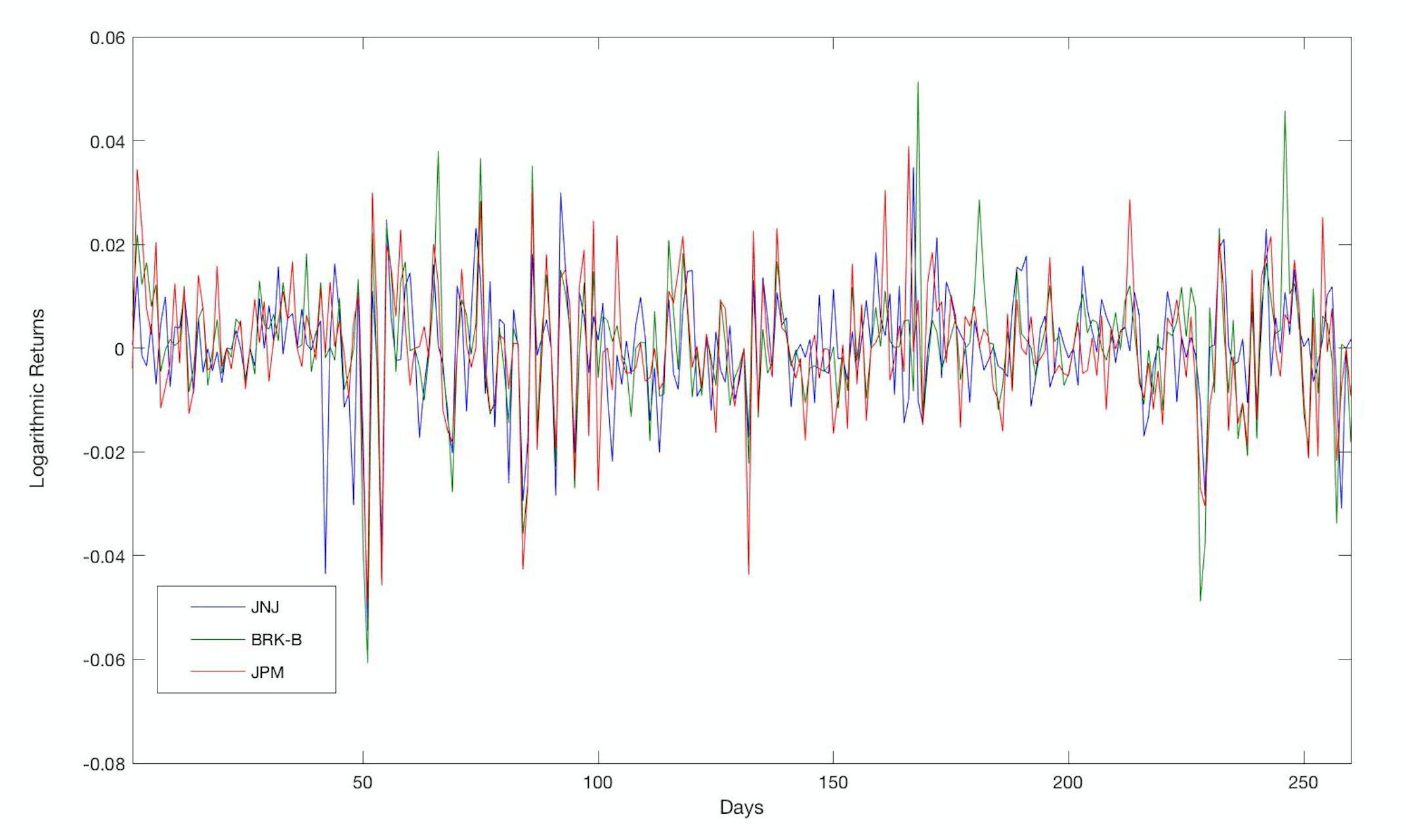}
\caption{Graphs of the logarithmic returns from the stock prices of {\textcolor{blue}{JNJ}}, {\textcolor{green}{BRK-B}}, and {\textcolor{red}{JPM}} over the period November 26, 2017 to November 23, 2018.}
\label{logreturns}
\end{figure}

A common assumption in the literature on stochastic volatility models is that the three-dimensional vectors of daily logarithmic returns, 
$$(\log(S_{j+1,1}/S_{j,1}), \log(S_{j+1,2}/S_{j,2}), \log(S_{j+1,3}/S_{j,3})),$$ $j=1,\ldots,260$, are mutually independent and identically distributed from a trivariate normal distribution. If this assumption were valid then the corresponding biweekly covariance matrices would be independent and identically distributed with Wishart distributions.  Thus, we will test the hypothesis that the biweekly covariance matrices are Wishart-distributed with $9$ degrees-of-freedom, i.e., $\alpha=4.5$.

To apply the test statistic $\boldsymbol{T}^2_n$ to test the hypothesis that the data are drawn from a Wishart distribution with $9$ degrees of freedom and unspecified scale matrix $\Sigma$, we use an algorithm developed by Koev and Edelman \cite{koevedelman} in \textsc{Matlab} \cite{matlab}
to evaluate the Bessel functions of two matrix arguments. Applying that algorithm to the data on the stock prices, we find that the observed value of the test statistic $\boldsymbol{T}^2_n$ is $0.127$.

We conducted a simulation study to approximate $\boldsymbol{T}^2_{n\,;\,0.05}$\,, the 95th-percentile of the null distribution of $\boldsymbol{T}_n^2$.  We generated $10,000$ random samples of size $n=26$ from the Wishart distribution with $\alpha=4.5$ and scale matrix $\Sigma = I_3$, calculated the value of $\boldsymbol{T}^2_n$ for each sample, and recorded the 95th-percentile of all 10,000 simulated values of $\boldsymbol{T}_n^2$.   We repeated this process a total of ten times, finally approximating $\boldsymbol{T}^2_{n\,;\,0.05}$ as the mean of all 10 simulated 95th-percentiles, viz., $\boldsymbol{T}^2_{n\,;\,0.05} = 0.002$.  Since  the observed value of $\boldsymbol{T}_n^2$ exceeds the critical value then we reject the null hypothesis that the random matrices $X_1, . . . . ., X_{26}$ are Wishart-distributed at the 5\% level of significance. Moreover, we derived from our simulation study an approximate P-value of $0.000$ for the test. Therefore, we have strong evidence that the three-dimensional vectors of logarithmic returns, $(\log(S_{j+1,1}/S_{j,1}), \log(S_{j+1,2}/S_{j,2}), \log(S_{j+1,3}/S_{j,3}))$, $j=1,\ldots,260$, do not have a trivariate normal distribution or are not mutually independent.

For an alternative approach to approximating $\boldsymbol{T}^2_{n\,;\,0.05}$, one can use the limiting null distribution of $\boldsymbol{T}^2_n$. For $\alpha = 4.5$, from (\ref{noeigenvalues_matrix}), we obtain the approximation $\boldsymbol{T}_n^2 \approx \sum_{k=1}^{21} \delta_k \chi^2_{1k}$.  This requires that we first calculate the $\delta_k$ (that are not equal to $\rho_\kappa$) and their multiplicities, numerically, using the results of Theorem \ref{thmeigenS_matrix}, and then we would apply the results of Kotz, et al. \cite{KJB} to derive the distribution of $\sum_{k=1}^{21} \delta_k \chi^2_{1k}$ and carry out the test.  We recommend in practice the one-term approximation \cite[Eqs.~(71),~(79)]{KJB},
$$
P\bigg(\sum_{k=1}^M \delta_k \chi^2_{1k} \ge t\bigg) \simeq P\big(\chi^2_M \ge 2t/(\delta_1 + \delta_M)\big)
$$
which leads to the explicit expression, $\boldsymbol{T}^2_{n\,;\,0.05} \simeq \frac12 (\delta_1+\delta_M)\chi^2_{M\,;\,0.05}$, for an approximate critical value of $\boldsymbol{T}_n^2$.  

As an alternative to calculating $\delta_1, \dotsc, \delta_M$, we can apply the interlacing inequalities in Proposition \ref{interlacing_matrix} to obtain a stochastic upper bound, $\sum_{k=1}^{M} \delta_k \chi^2_{1k} \le \sum_{0 \le |\kappa| \le r} \rho_\kappa \chi^2_{1 \kappa}$.  If we carry out the test by using the upper bound, $\sum_{0 \le |\kappa| \le r} \rho_\kappa \chi^2_{1 \kappa}$, with its exact distribution or a one-term approximation obtained from Kotz, et al. \cite[loc.~cit.]{KJB}, we will obtain a conservative test of the null hypothesis, i.e., with a level of significance at most 5\%.

\subsection{Consistency of the test}
\label{sec: consistency_matrix}

Before stating the theorem, we provide a lemma which will be helpful for establishing consistency of the test.  The proof of the following result is similar to the proof of Lemma \ref{lemma_norm_bound_nabladiff_bessel_2matrixarg}.

\begin{lemma}
\label{lemma_lipschitz_bessel}
For $T > 0$, $Y_1 > 0$, and $Y_2 > 0$, 
\begin{equation}
\label{lipschitz_bessel_2matrixargument}
\Gamma_m(\alpha) \, \bigg\lVert A_{\nu}(T,Y_1) - A_{\nu}(T,Y_2) \bigg\rVert_F \le 
2 m^{3/4} \, \lVert T \rVert^{1/2}_F \ \rVert Y_1-Y_2\lVert^{1/2}_F. 
\end{equation}
\end{lemma}

\smallskip

\begin{theorem}
\label{consistency_matrixcase}
Let $X_1,X_2,\dotsc$ be a sequence of $m \times m$ positive-definite, i.i.d. random matrices with mean $\mu $. Assume also that the p.d.f. of $X_1$ is of the form:
\begin{equation}
\label{assumption1_alternatives}
f(X_1)=f_0({\mu}^{-1/2} X_1 {\mu}^{-1/2}),
\end{equation}
where $f_0$ is orthogonally invariant. Let $\gamma \in (0,1)$ denote the level of significance of the test and $c_{n,\gamma}$ be the $(1-\gamma)$-quantile of the test statistic $\boldsymbol{T}^2_n$ under $H_0$. If $X_1,X_2,\dotsc$ are not Wishart-distributed then  
$$
\lim_{n \rightarrow \infty} P(\boldsymbol{T}^2_n > c_{n,\gamma})=1.
$$
\end{theorem}
\medskip

\Pro
By the definition (\ref{statistic_wishart}) of the test statistic and \eqref{remark_sameeigen}, we have
\begin{align*}
n^{-1} \boldsymbol{T}^2_n&=\int_{T > 0} {\bigg[ \frac{1}{n}\sum_{j=1}^{n} \Gamma_m(\alpha) A_{\nu} (T, Z_j)-\etr(-\alpha^{-1} T) \bigg]^2}\, \dd P_0(T),
\end{align*}
where $Z_j=X_{j}^{1/2} \bar{X}_n^{-1} X_{j}^{1/2}$. By subtracting and adding the quantity 
$$
\frac{1}{n}\sum_{j=1}^{n} \Gamma_m(\alpha) A_{\nu} (T, X_{j}^{1/2} \mu^{-1} X_{j}^{1/2})
$$
inside the squared term, and then expanding the integrand, we obtain
\begin{align}
n^{-1}  \boldsymbol{T}^2_n 
\label{consistency1_matrixcase}
&=\int_{T > 0} {\bigg[ \ \frac{\Gamma_m(\alpha)}{n}\sum_{j=1}^{n} A_{\nu} (T, X_{j}^{1/2} \mu^{-1} X_{j}^{1/2})-\etr(-\alpha^{-1} T) \bigg]^2}\, \dd P_0(T) \\
\label{consistency2_matrixcase}
& \quad + \int_{T > 0} {\bigg[ \frac{\Gamma_m(\alpha)}{n} \sum_{j=1}^n \bigg(A_{\nu} (T, Z_j)-A_{\nu} (T, X_{j}^{1/2} \mu^{-1} X_{j}^{1/2})\bigg) \bigg]^2}\, \dd P_0(T) \\
\label{consistency3_matrixcase}
& \quad + 2 \int_{T > 0} { \bigg[ \frac{\Gamma_m(\alpha)}{n} \sum_{j=1}^n \bigg( A_{\nu} (T, Z_j)-A_{\nu} (T, X_{j}^{1/2} \mu^{-1} X_{j}^{1/2})\bigg) \bigg]}\nonumber\\
& \qquad\qquad \times {\bigg[ \frac{\Gamma_m(\alpha)}{n}\sum_{j=1}^{n} A_{\nu} (T, X_{j}^{1/2} \mu^{-1} X_{j}^{1/2})-\etr(-\alpha^{-1} T) \bigg] }\, \dd P_0(T).
\end{align}

We begin by proving that the integral (\ref{consistency2_matrixcase}) converges almost surely to $0$. By (\ref{lipschitz_bessel_2matrixargument}), there exists a constant $C > 0$ such that 
\begin{align*}
\frac{\Gamma_m(\alpha)}{n}  \sum_{j=1}^n \bigg| A_{\nu} (T, Z_j) & - A_{\nu} (T, X_{j}^{1/2} \mu^{-1} X_{j}^{1/2})\bigg| \nonumber\\
& \le C \, \lVert T \lVert_F^{1/2} \ \frac{1}{n} \sum_{j=1}^n  \lVert Z_j - X_{j}^{1/2} \mu^{-1} X_{j}^{1/2} \rVert_F^{1/2} \nonumber\\
& = C \, \lVert T \lVert_F^{1/2} \ \frac{1}{n} \sum_{j=1}^n  \lVert X_{j}^{1/2} (\bar{X}^{-1}_n-\mu^{-1}) X_{j}^{1/2} \rVert_F^{1/2} \nonumber\\
& \le C \, \lVert T \lVert_F^{1/2} \ \lVert \bar{X}^{-1}_n-\mu^{-1} \rVert_F^{1/2} \ \frac{1}{n} \sum_{j=1}^n \lVert X_j \rVert_F^{1/2},
\end{align*}
since the Frobenius norm is sub-multiplicative. By the triangle inequality, we conclude that the integral (\ref{consistency2_matrixcase}) is bounded above by 
$$
C^2 \ \lVert \bar{X}^{-1}_n-\mu^{-1} \rVert_F \ \bigg( \frac{1}{n} \sum_{j=1}^n \lVert X_j \rVert_F^{1/2} \bigg)^2 \ \int_{T > 0} {\lVert T \lVert_F}\, \dd P_0(T).
$$
By the Cauchy-Schwarz inequality, 
$$
\bigg( n^{-1} \sum_{j=1}^n \lVert X_j \rVert_F^{1/2} \bigg)^2 \le n^{-1} \sum_{j=1}^n \lVert X_j \rVert_F.
$$
Since $T > 0$, then $(\tr T^2) \le (\tr T)^2$, so we have 
\begin{align*}
\int_{T > 0} {\lVert T \rVert_F }\, \dd P_0(T) &= \int_{T > 0} {(\tr T^2 )^{1/2} }\, \dd P_0(T) \\
& \le \int_{T > 0} {(\tr T) }\, \dd P_0(T)
< \infty,
\end{align*}
by (\ref{trace_zonal}) and (\ref{zonalintegral}).  

Moreover, by the Strong Law of Large Numbers and the Continuous Mapping Theorem, $\lVert \bar{X}^{-1}_n-\mu^{-1} \rVert_F \rightarrow 0$, almost surely. Also, again by the Strong Law of Large Numbers, $n^{-1} \sum_{j=1}^n \lVert X_j \rVert_F \rightarrow E \lVert X_1 \rVert_F$, almost surely. It is elementary to verify that $E \lVert X_1 \rVert_F < \infty$. Since $X_1 > 0$ and $\mu=E(X_1) > 0$, we have $ \lVert X_1 \rVert_F \le \tr X_1$ and so $E \lVert X_1 \rVert_F \le E(\tr X_1) =\tr\mu < \infty$. Therefore, (\ref{consistency2_matrixcase}) converges to 0, almost surely.

Second, we show that (\ref{consistency3_matrixcase}) tends to 0, almost surely. By (\ref{2besselineq_matrixargument}), the fact that $ \etr(-\alpha^{-1} T) \le 1$ for $T > 0$, and the triangle inequality, we have 
$$
\bigg| \frac{\Gamma_m(\alpha)}{n}\sum_{j=1}^{n} A_{\nu} (T, X_{j}^{1/2} \mu^{-1} X_{j}^{1/2})-\etr(-\alpha^{-1} T) \bigg| \le 2. 
$$
Further, by the triangle inequality, the absolute value of (\ref{consistency3_matrixcase}) is less than or equal to
\begin{align}
\label{cons3_matrixcase2}
& 2 \int_{T > 0} { \bigg| \frac{\Gamma_m(\alpha)}{n} \sum_{j=1}^n \bigg(A_{\nu} (T, Z_j)-A_{\nu} (T, X_{j}^{1/2} \mu^{-1} X_{j}^{1/2})\bigg) \bigg| }\, \dd P_0(T). 
\end{align}
By the Cauchy-Schwarz inequality and the fact that $\int_{T > 0} \dd P_0(T)=1$, (\ref{cons3_matrixcase2}) is seen to be less than or equal to 
\begin{align*}
2 \bigg( \int_{T > 0} { \bigg[ \frac{\Gamma_m(\alpha)}{n} \sum_{j=1}^n \bigg(A_{\nu} (T, Z_j)-A_{\nu} (T, X_{j}^{1/2} \mu^{-1} X_{j}^{1/2})\bigg) \bigg]^2 }\, \dd P_0(T) \bigg)^{1/2}.
\end{align*}
Following the same argument as for integral (\ref{consistency2_matrixcase}), we conclude that integral (\ref{consistency3_matrixcase}) converges to $0$, almost surely. 

Since 
$
A_{\nu} (T, X_{j}^{1/2} \mu^{-1} X_{j}^{1/2})=A_{\nu} (T, \mu^{-1/2} X_j \mu^{-1/2}),
$ 
we see that the integral (\ref{consistency1_matrixcase}) equals  
$$
\int_{T > 0} {\bigg[ \frac{\Gamma_m(\alpha)}{n}\sum_{j=1}^{n} A_{\nu} (T, \mu^{-1/2} X_j \mu^{-1/2})-\etr(-\alpha^{-1} T) \bigg]^2}\, \dd P_0(T).
$$
We subtract and add inside the squared term the orthogonally invariant Hankel transform of $\mu^{-1/2} X_1 \mu^{-1/2}$, i.e., the quantity 
$
E [\Gamma_m(\alpha) A_{\nu} (T, \mu^{-1/2} X_1 \mu^{-1/2})],
$ 
and expand the integrand. Then we find that (\ref{consistency1_matrixcase}) equals 
\begin{align}
\label{term1_consistency1_matrixcase}
& \int_{T > 0} \bigg[\frac{\Gamma_m(\alpha)}{n} \sum_{j=1}^{n} A_{\nu} (T, \mu^{-1/2} X_j \mu^{-1/2}) 
{- E [\Gamma_m(\alpha) A_{\nu} (T, \mu^{-1/2} X_1 \mu^{-1/2})] \bigg]^2}\, \dd P_0(T)\\
& + \int_{T > 0} {\bigg[ E [\Gamma_m(\alpha) A_{\nu} (T, \mu^{-1/2} X_1 \mu^{-1/2})] -\etr(-\alpha^{-1} T) \bigg]^2}\, \dd P_0(T) \nonumber\\
& + 2 \int_{T > 0} {\bigg[ \frac{\Gamma_m(\alpha)}{n}\sum_{j=1}^{n} A_{\nu} (T, \mu^{-1/2} X_j \mu^{-1/2})} 
{- E [\Gamma_m(\alpha) A_{\nu} (T, \mu^{-1/2} X_1 \mu^{-1/2})] \bigg] }\nonumber\\
\label{term3_consistency1_matrixcase}
& \qquad\quad \times {\bigg[ E [\Gamma_m(\alpha) A_{\nu} (T, \mu^{-1/2} X_1 \mu^{-1/2})] -\etr(-\alpha^{-1} T) \bigg]}\, \dd P_0(T).
\end{align}
By the Strong Law of Large Numbers in $L^2$ \cite[p.~189, Corollary 7.10]{ref19}, we conclude that the term (\ref{term1_consistency1_matrixcase}) converges to 0, almost surely. 

Next, we show that (\ref{term3_consistency1_matrixcase}) converges to 0, almost surely. By (\ref{2besselineq_matrixargument}) and the fact that $\etr(-\alpha^{-1} T) \le 1$ for $T >0$, we have 
$$
\bigg| E [\Gamma_m(\alpha) A_{\nu} (T, \mu^{-1/2} X_1 \mu^{-1/2})] -\etr(-\alpha^{-1} T) \bigg| \le 2.
$$
Therefore, the absolute value of the integral (\ref{term3_consistency1_matrixcase}) is less than or equal to
\begin{align*}
2 \int_{T > 0} \ & {\bigg| \frac{\Gamma_m(\alpha)}{n}\sum_{j=1}^{n} A_{\nu} (T, \mu^{-1/2} X_j \mu^{-1/2})} {- E [\Gamma_m(\alpha) A_{\nu} (T, \mu^{-1/2} X_1 \mu^{-1/2})] \ \bigg| }\, \dd P_0(T)\\
& \le \ 2 \bigg( \int_{T > 0} {\bigg[ \frac{\Gamma_m(\alpha)}{n}\sum_{j=1}^{n} A_{\nu} (T, \mu^{-1/2} X_j \mu^{-1/2})} \\
& \qquad\qquad\qquad\qquad\qquad {- E [\Gamma_m(\alpha) A_{\nu} (T, \mu^{-1/2} X_1 \mu^{-1/2})] \ \bigg]^2 }\, \dd P_0(T) \bigg)^{1/2},
\end{align*}
where the latter bound follows from the Cauchy-Schwarz inequality. Again, by the Strong Law of Large Numbers in $L^2$, we conclude that the integral (\ref{term3_consistency1_matrixcase}) converges to 0, almost surely. 

We have now shown that 
\begin{align}
\label{almostsurelyconv_matrixcase}
\frac{1}{n}\boldsymbol{T}^2_n \xrightarrow{a.s.} \int_{T > 0} {\bigg[ E [\Gamma_m(\alpha) A_{\nu} (T, \mu^{-1/2} X_1 \mu^{-1/2})] -\etr(-\alpha^{-1} T) \bigg]^2}\, \dd P_0(T).
\end{align}
Denote by $\Delta$ the right-hand side of (\ref{almostsurelyconv_matrixcase}); then $\Delta \ge 0$. Suppose that $\Delta=0$, then 
$$
E [\Gamma_m(\alpha) A_{\nu} (T, \mu^{-1/2} X_1 \mu^{-1/2})] -\etr(-\alpha^{-1} T)=0,
$$
equivalently, $\mathcal{\widetilde{H}}_{\mu^{-1/2} X_1 \mu^{-1/2}}(T)-\etr(-\alpha^{-1} T)=0$, $P_0$-almost everywhere. By continuity, we obtain $\mathcal{\widetilde{H}}_{\mu^{-1/2} X_1 \mu^{-1/2}}(T)-\etr(-\alpha^{-1} T)=0$ for all $T > 0$. By the Uniqueness Theorem for orthogonally invariant Hankel transforms, it follows that $\mu^{-1/2} X_1 \mu^{-1/2}$ has a Wishart distribution. By Muirhead \cite[p.~92, Theorem 3.2.5]{muirhead}, $X_1$ has also a Wishart distribution, which contradicts the assumption that $X_1$ does not have a Wishart distribution. Therefore, $\Delta > 0$. 

Under $H_0$, $n^{-1} \boldsymbol{T}^2_n \xrightarrow{a.s.} 0$, and therefore $n^{-1} \boldsymbol{T}^2_n\xrightarrow{p} 0$, i.e., for any $\epsilon >0$,
$$
\lim_{n \rightarrow \infty} P_{H_0} \big( n^{-1} \boldsymbol{T}^2_n \ge \epsilon \big)=0.
$$
Thus, for any $\epsilon > 0$ and $\gamma>0$, there exists $n_0(\epsilon,\gamma) \in \mathbb{N}$ such that
$$
P_{H_0} \big( n^{-1} \boldsymbol{T}^2_n \ge \epsilon \big) \le \gamma,
$$
for all $n \ge n_0(\epsilon,\gamma)$.  Let $c_{n,\gamma}$ be the $(1-\gamma)$-quantile of the test statistic $\boldsymbol{T}^2_n$ under $H_0$.  Then $0 \le c_{n,\gamma} \le n\epsilon$ for all $n \ge n_0(\epsilon)$ since, by definition, $c_{n,\gamma} := \inf\{ x \ge 0: P_{H_0}(\boldsymbol{T}^2_n > x) \le \gamma\}$.  Therefore, $0 \le n^{-1} c_{n,\gamma} \le \epsilon$ for all $n \ge n_0 (\epsilon)$.  In summary, for any $\epsilon > 0$, there exists $n_0(\epsilon) \in \bN$ such that $n^{-1} c_{n,\gamma} \le \epsilon$ for all $n \ge n_0(\epsilon)$, i.e., 
\begin{eqnarray}
\label{quantileconv_matrixcase}
\lim_{n \rightarrow \infty} n^{-1} c_{n,\gamma} =0.
\end{eqnarray}

By (\ref{almostsurelyconv_matrixcase}) and (\ref{quantileconv_matrixcase}), we have $n^{-1} \boldsymbol{T}^2_n-n^{-1} c_{n,\gamma} \xrightarrow{a.s.} \Delta$, and therefore $n^{-1} \boldsymbol{T}^2_n-n^{-1}c_{n,\gamma} \xrightarrow{p} \Delta$. Thus, by Severini \cite[ p.~340, Corollary 11.3 (i)]{ref12}), we conclude that $n^{-1} \boldsymbol{T}^2_n-n^{-1}c_{n,\gamma} \xrightarrow{d} \Delta$.
Further, 
\begin{align*}
\lim_{n \rightarrow \infty} P(\boldsymbol{T}^2_n > c_{n,\gamma})&=\lim_{n \rightarrow \infty} P\big( n^{-1}\boldsymbol{T}^2_n -n^{-1} c_{n,\gamma} > 0 \big)\\
&=1-\lim_{n \rightarrow \infty} P\big( n^{-1}\boldsymbol{T}^2_n -n^{-1}c_{n,\gamma} \le 0 \big).
\end{align*}
Since the distribution function of the constant positive random variable $\Delta$ is continuous at 0, we conclude that 
$$
\lim_{n \rightarrow \infty} P(\boldsymbol{T}^2_n > c_{n,\gamma})=1-P(\Delta \le 0)=1-0=1.
$$
This concludes the proof.
$\qed$

\begin{remark}
\rm{
We show that the assumption (\ref{assumption1_alternatives}), made in Theorem \ref{consistency_matrixcase}, holds for two alternative distributions. 

First, the \textit{matrix $F$-distribution} \cite[Section 4, part (c)]{khatri} or \cite[Eqs. (65), (72)]{james}:  Let $X$ be a positive-definite random matrix with p.d.f.
$$
f(X)=\frac{\Gamma_m(a+b)}{\Gamma_m(a)\Gamma_m(b)} \ (\det X)^{a-(m+1)/2} \ (\det (I_m+X))^{-(a+b)},
$$
where $a > \tfrac12 (m-1)$ and $b > \tfrac12 (m +1)$.   Since $f(X)$ is orthogonally invariant then, by Schur's Lemma, there exists a constant $c$ such that $\mu = E(X) = c I_m$.  

Last, a \textit{linear combination of two Wishart matrices}: 
Let $X$ be a positive-definite random matrix with p.d.f.
$$
f(X)=\frac{\delta^{mb} (\delta-1)^{ma}}{\Gamma_m(a+b)} \ (\det X)^{a+b-(m+1)/2} \ \etr(-\delta X) \ {_1}F_1(a; a+b; X),
$$
where $a > \tfrac12 (m-1)$, $b > \tfrac12 (m-1)$, and $\delta > 1$. 
By \cite[Section 4.4]{guptarichards2}, it is known that $X$ is equal in distribution to $X_1+ \delta^{-1}X_2$, where $X_1$ and $X_2$ are independent, $X_1 \sim W_m(a,I_m)$ and $X_2 \sim W_m(b,I_m)$.  Again, the distribution of $X$ is orthogonally invariant, therefore it satisfies (\ref{assumption1_alternatives}).
}\end{remark}

\section{Contiguous Alternatives to the Null Hypothesis}
\label{contiguousmatrix}

In this section, we derive the limiting distribution of the test statistic under a sequence of contiguous alternatives. 

\subsection{Assumptions}
\label{assumptionscontiguous_matrixcase}

For $n \in \mathbb{N}$ and $m \ge 2$, let $X_{n1},\dotsc,X_{nn}$ be a triangular array of row-wise independent $m \times m$ random matrices. As usual, let $P_0=W_m(\alpha,I_m)$, $\alpha > \max\{ \tfrac12(2m-1), \tfrac12(m+3) \}$, and let $Q_{n1}$ be a probability measure dominated by $P_0$. 

We wish to test the hypothesis
$$
H_0: \ \text{The marginal distribution of each} \ X_{ni}, \ i=1,\dotsc,n, \ \text{is} \ P_0
$$
against the alternative
$$
H_1: \ \text{The marginal distribution of each} \ X_{ni}, \ i=1,\dotsc,n, \ \text{is} \ Q_{n1}.
$$ 

We write the Radon-Nikodym derivative of $Q_{n1}$ with respect to $P_0$ in the form
\begin{align}
\label{radon_nikodym_matrixcase}
\frac{\dd Q_{n1}}{\dd P_0}=1+n^{-1/2}h_n.
\end{align}
We will need two assumptions in the sequel.

\begin{assumption}
\label{2assumptions_matrixcase}
{\rm
We assume that: 
\begin{itemize}
\item[(A1)] The functions $\{h_n: n \in \bN \}$ form a sequence of $P_0$-integrable functions converging pointwise, $P_0$-almost everywhere, to a function $h$, and 
\item[(A2)] $\sup_{n \in \bN} E_{P_0} |h_n|^4 < \infty$. 
\end{itemize}
}\end{assumption}

Note that since $ \int (\dd Q_{n1}/ \dd P_0) \, \dd P_0 = 1$ then we also have $\int {h_n}\, \dd P_0 = 0$, for all $n \in \bN$.  Denote the indicator function of an event $A$ by $I(A)$. By applying (A2), we deduce the uniform integrability of $|h_{n}|^2$: 
\begin{align*}
\lim_{k \rightarrow \infty} \sup_{n \in \mathbb{N}} E_{P_0} \big( |h_n|^2 I( |h_n|^2 >k) \big)&=\lim_{k \rightarrow \infty} \sup_{n \in \mathbb{N}} \int {|h_n|^2 I( |h_n|^2 >k)}\, \dd P_0\\
&\le \lim_{k \rightarrow \infty} \sup_{n \in \mathbb{N}} \int { k^{-1} |h_n|^4}\, \dd P_0\\ 
&=\lim_{k \rightarrow \infty} k^{-1} \sup_{n \in \mathbb{N}} E_{P_0} |h_n|^4=0.
\end{align*}
By Bauer \cite[p.~95, Theorem 2.11.4]{bauer}, the $P_0$-almost everywhere convergence of $h_n$ to $h$ implies the $P_0$-stochastic convergence of $h_n$ to $h$. Again by Bauer \cite[p.~104, Theorem 2.12.4]{bauer}, the uniform integrability of $|h_n|^2$ along with the $P_0$-stochastic convergence of $h_n$ to $h$ imply the convergence of $h_n$ in mean square, i.e., 
$$
\lim_{n \rightarrow \infty} \int {|h_n-h|^2}\, \dd P_0=0,
$$
and therefore
\begin{align*}
\lim_{n \rightarrow \infty} \int {|h_n|^2}\, \dd P_0 = \int {|h|^2}\, \dd P_0.
\end{align*}
Since convergence in mean square implies convergence in mean, we have
$$
\lim_{n \rightarrow \infty} \int {|h_n-h|}\, \dd P_0=0,
$$
and thus, 
$$
\lim_{n \rightarrow \infty} \int {h_n}\, \dd P_0=\int {h}\, \dd P_0.
$$
Now, due to the fact that $\int {h_n}\, \dd P_0=0$ for all $n \in \mathbb{N}$, we obtain
$$
\lim_{n \rightarrow \infty} \int {h_n}\, \dd P_0=\int {h}\, \dd P_0=0.
$$

\subsection{Examples}

In this subsection, we verify that Assumptions \ref{2assumptions_matrixcase} are valid for a broad collection of sequences of contiguous alternatives.  

\subsubsection{Wishart alternatives with contiguous scale matrices}

Let $Q_{n1}:=W_m(\alpha, \Sigma_n)$ with $\alpha > \max\{ \tfrac12(2m-1), \tfrac12(m+3) \}$ and $\Sigma_n=(1+\frac{1}{\sqrt{n}})I_m$. Then,
\begin{align*}
\frac{\dd Q_{n1}}{\dd P_0}&=\left( 1+n^{-1/2} \right)^{m\alpha}  \etr( -n^{-1/2} X),
\end{align*}
$X>0$. We equate the Radon-Nikodym derivative to $1+n^{-1/2}h_n(X)$, obtaining
$$
h_n(X)=n^{1/2} \bigg[\left( 1+n^{-1/2} \right)^{m\alpha}  \etr( -n^{-1/2}X) -1\bigg], 
$$
for $X > 0$. By applying L'Hospital's rule, we obtain
\begin{align*}
h(X) &:=\lim_{n \rightarrow \infty} h_n(X)=m\alpha -\tr X,
\end{align*}
for $X > 0$. Next, we find $E_{P_0} |h_n^4|$. Define
\begin{align}
\label{remainderterm_matrixcase}
R_n(X)&=\etr(-n^{-1/2}X)-( 1-n^{-1/2} (\tr X)\\
&=\sum_{k=2}^{\infty} \frac{1}{k!} \big(-n^{-1/2} (\tr X) \big)^{k} \nonumber,
\end{align}
the remainder term of the Taylor series expansion of $\etr(-n^{-1/2}X)$, $X > 0$. Then, by elementary algebraic manipulations, we obtain
\begin{align*}
h_n(X) &= n^{1/2} (1+n^{-1/2})^{m\alpha} \etr (-n^{-1/2} X ) - n^{1/2} \\
&= n^{1/2} (1+n^{-1/2})^{m\alpha} \big(R_n(X) + 1-n^{-1/2} (\tr X) \big) - n^{1/2} \\
&= (1+n^{-1/2})^{m\alpha-1} \big[1+(1+n^{1/2})R_n(X) - (1+n^{-1/2}) (\tr X) \big] \\
& \hskip 2.5truein + n^{1/2}[(1+n^{-1/2})^{m\alpha-1} - 1].
\end{align*}
By (\ref{remainderterm_matrixcase}), the triangle inequality, and the Lipschitz continuity of the exponential function, we have
\begin{align*}
|R_n(X)| &\le n^{1/2} |R_n(X) | 
\nonumber \\
&\le n^{1/2} \big[ |\etr(-n^{-1/2} X) - 1| + n^{-1/2} (\tr X) \big] \nonumber \\
&\le n^{1/2} \big[n^{-1/2} (\tr X) + n^{-1/2}  (\tr X)\big] 
= 2 \tr X,
\end{align*}
$X > 0$.  Therefore, 
\begin{align*}
|h_n(X)| &\le (1+n^{-1/2})^{m\alpha-1} \bigg[1+(1+n^{1/2})|R_n(X)| + (1+n^{-1/2}) (\tr X) \bigg] \\
& \hskip 2.5truein + \big| n^{1/2}\big((1+n^{-1/2})^{m\alpha-1}  - 1\big)\big| \\
&\le (1+n^{-1/2})^{m\alpha-1} (1+4\tr X + 2\tr X ) + \big| n^{1/2}\big((1+n^{-1/2})^{m\alpha-1}  - 1\big)\big| \\
&= (1+n^{-1/2})^{m\alpha-1} (1+6 \tr X) + \big| n^{1/2}\big((1+n^{-1/2})^{m \alpha-1}  - 1\big)\big|.
\end{align*}
It is elementary that 
$
(1+n^{-1/2})^{m\alpha-1} \to 1
$ 
and 
$
n^{1/2}\big((1+n^{-1/2})^{m\alpha-1} - 1\big) \to m\alpha-1
$ 
as $n \to \infty$; therefore, there exists a positive constant $M$ such that 
$
(1+ n^{-1/2})^{m\alpha-1} \le M
$  
and 
$
\big| n^{1/2}\big((1+n^{-1/2})^{m\alpha-1} - 1\big) \big| \le M
$ 
for all $n$. Therefore, $|h_n(X)| \le  M(1 + 6 \tr X) + M = M(2+6 \tr X)$, $X > 0$, so we obtain 
\begin{align*}
E_{P_0}|h_n|^4 \le M^4 \int_{X > 0} (2 + 6 \tr X)^4 \, \dd P_0(X), 
\end{align*}
and this bound does not depend on $n$. By (\ref{trace_zonal}) and (\ref{zonalintegral}), the above integral is finite; thus, $\sup_{n \in \bN} E_{P_0} |h_n|^4 < \infty$.

\subsubsection{Wishart alternatives with contiguous shape parameters} 

Let $Q_{n1}:=W_m(\alpha_n, I_m)$ with $\alpha_n=\alpha+\frac{1}{\sqrt{n}}$, $\alpha > \max\{ \tfrac12(2m-1), \tfrac12(m+3) \}$. We have
$$
\frac{\dd Q_{n1}}{\dd P_0}=\frac{\Gamma_m(\alpha)}{\Gamma_m(\alpha_n)} (\det X)^{1/\sqrt{n}},
$$
$ X > 0$. Following (\ref{radon_nikodym_matrixcase}), we equate this Radon-Nikodym derivative to $1+n^{-1/2}h_n(X)$, obtaining
$$
h_n(X)=n^{1/2} \bigg( \frac{\Gamma_m(\alpha)}{\Gamma_m(\alpha_n)} (\det X)^{1/\sqrt{n}} -1 \bigg ), 
$$
for $X > 0$. Recall the \textit{multivariate digamma function} 
$$
\psi_m(z):=\frac{\dd}{\dd z} \log \Gamma_m(z)=\frac{\Gamma_m'(z)}{\Gamma_m(z)},
$$
$z > 0$. Applying L'Hospital's rule, we obtain
\begin{align*}
h(X) &: =\lim_{n \rightarrow \infty} h_n(X) \\
&= \lim_{n \rightarrow \infty} n^{1/2} \bigg( \frac{\Gamma_m(\alpha)}{\Gamma_m(\alpha + n^{-1/2})} (\det X)^{1/\sqrt{n}} -1 \bigg)\\
&= \log (\det X)-\psi_m(\alpha), 
\end{align*}
$X > 0$. To calculate $E_{P_0} |h_n|^4$, we apply the binomial expansion, obtaining 
\begin{multline*}
\bigg| n^{1/2} \bigg( \frac{\Gamma_m(\alpha)}{\Gamma_m(\alpha + n^{-1/2})} (\det X)^{1/\sqrt{n}}-1 \bigg) \bigg|^4 \\
= n^{2} \sum_{j=0}^4 (-1)^j \binom{4}{j} \bigg(\frac{\Gamma_m(\alpha)}{\Gamma_m(\alpha+n^{-1/2})} \bigg)^j (\det X)^{j/\sqrt{n}} ,
\end{multline*}
thus, 
\begin{align}
\label{fourthmoment_matrixcase}
E_{P_0} |h_n|^4 &=n^{2} \sum_{j=0}^4 (-1)^j \binom{4}{j} \bigg(\frac{\Gamma_m(\alpha)}{\Gamma_m(\alpha+n^{-1/2})} \bigg)^j \, \frac{\Gamma_m(\alpha + jn^{-1/2})}{\Gamma_m(\alpha)}.
\end{align}
Next, the Taylor expansion of $\Gamma_m(\alpha)/\Gamma_m(\alpha+n^{-1/2})$ for sufficiently large values of $n$ is 
\begin{eqnarray}
\label{taylorexpwishartalter}
\frac{\Gamma_m(\alpha)}{\Gamma_m(\alpha+n^{-1/2})}= \sum_{j=0}^4 a_j n^{-j/2} + o(n^{-2}),
\end{eqnarray}
where $a_0 = 1$.  

After lengthy but straightforward calculations, we obtain
\begin{align*}
a_1&=-\psi_m(\alpha),\\ 
a_2&=\frac{1}{2} \psi_m^2(\alpha)-\frac{1}{2}\psi'_m(\alpha),\\
a_3&=-\frac{1}{6}\psi_m^3(\alpha)  -\frac{1}{6}\psi''_m(\alpha)+\frac{1}{2}\psi_m(\alpha)\psi'_m(\alpha),\\
a_4&=-\frac{\psi'''_m(\alpha)}{24}+\frac{1}{8}(\psi'_m(\alpha))^2+\frac{1}{6}\psi_m(\alpha)\psi''_m(\alpha)-\frac{1}{4}\psi^2_m(\alpha)\psi'_m(\alpha)+\frac{1}{24}\psi^4_m(\alpha).
\end{align*}
Next, we substitute the Taylor expansion (\ref{taylorexpwishartalter}) in (\ref{fourthmoment_matrixcase}) and then take the limit as $n \to \infty$. By applying L'Hospital's rule four times then, after some lengthy but straightforward calculations, we obtain 
$$
\lim_{n \to \infty} E_{P_0} |h_n|^4 = 9a_1^4 + 24a_2^2 + 24a_1a_3 - 36a_1^2a_2 - 24a_4.
$$
Thus, $E_{P_0} |h_n^4|$ is a bounded sequence, and therefore $\sup_{n \in \bN} E_{P_0} |h_n|^4 < \infty$.

\subsubsection{Contaminated Wishart models}

Consider the contamination model, 
\begin{equation}
\label{Qn1contamination_matrixcase}
Q_{n1}:=(1-n^{-1/2}) P_0 + n^{-1/2} W_m(2\alpha,I_m),
\end{equation}
where, as usual, $ \alpha > \max\{ \tfrac12(2m-1), \tfrac12(m+3) \}$. We note that contaminated Wishart models appear also in the analysis of diffusion tensor images \cite{jian2}.

We have
\begin{align*}
\frac{\dd Q_{n1}}{\dd P_0}&=n^{-1/2} \bigg( \frac{\Gamma_m(\alpha)}{\Gamma_m(2\alpha)} (\det X)^{\alpha}-1 \bigg) +1,
\end{align*}
for $X > 0$. Following (\ref{radon_nikodym_matrixcase}), we equate this Radon-Nikodym derivative to $1+n^{-1/2}h_n(X)$, obtaining
$$
h_n(X)=\frac{\Gamma_m(\alpha)}{\Gamma_m(2\alpha)} (\det X)^{\alpha}-1,
$$
for $X > 0$. Thus, 
$$
h(X):=\lim_{n \rightarrow \infty} h_n(X)=\frac{\Gamma_m(\alpha)}{\Gamma_m(2\alpha)} (\det X)^{\alpha}-1,
$$
$X > 0$. Since 
\begin{align*}
E_{P_0} |h_n^4| &=\int_{X>0} {  \bigg( \frac{\Gamma_m(\alpha)}{\Gamma_m(2\alpha)} (\det X)^{\alpha}-1  \bigg)^4}\, \dd P_0(X)\\
\end{align*}
clearly is finite and does not depend on $n$ then $\sup_{n \in \bN} E_{P_0} |h_n|^4 < \infty$.

We note also that the model (\ref{Qn1contamination_matrixcase}) is a special case of the contamination model 
$$
Q_{n2} = (1-n^{-1/2}) P_0 + n^{-1/2} P_1,
$$
where $P_1$ is a probability measure dominated by $P_0$, and $\int {(\dd P_1/\dd P_0)^4}\, \dd P_0 < \infty$.  The preceding calculations can also be done for many choices of $P_1$. 

For example, consider the case in which $P_1$ is the probability measure corresponding to the matrix generalized inverse Gaussian distribution \cite{butler} with density function
$$
f_1(X)=c_1 \ (\det X)^{b-\tfrac{1}{2}(m+1)} \etr(-\Phi X^{-1} -\Psi X),
$$
$X > 0$, where $c_1$ is the normalizing constant, $\Phi$ and $\Psi$ are symmetric non-negative definite matrices, and  $b \in \mathbb{R}$. Then
\begin{align*}
\int {(\dd P_1/\dd P_0)^4}\, \dd P_0 &= \int_{X > 0} \frac{c_1^4 (\det X)^{4b-\tfrac{4}{2}(m+1)} \etr(-4\Phi X^{-1} -4\Psi X)}{c_0^3 (\det X)^{3\alpha-\tfrac{3}{2}(m+1)} \etr(-3X)} \dd X \\
&=c \ \int_{X > 0} (\det X)^{4b-3\alpha-\tfrac{1}{2}(m+1)} \etr(-4\Phi X^{-1}-(4 \Psi-3 I_m)X) \dd X,
\end{align*}
where $c_0=1/\Gamma_m(\alpha)$ is the normalizing constant of $W_m(\alpha, I_m)$ and $c = c_1^4/c_0^3$.  By \cite[p.~506]{herz} and \cite[Eq.~(2)]{butler}, we deduce that $\int {(\dd P_1/\dd P_0)^4}\, \dd P_0 < \infty$ in the following cases:
\begin{itemize}
\item[(i)]$\Phi \ge 0$, $\Psi-\tfrac{3}{4} I_m > 0$, $b \ge \tfrac{1}{4}(3\alpha + \tfrac{1}{2}m)$
\item[(ii)]$\Phi > 0$, $\Psi-\tfrac{3}{4} I_m > 0$, $b \in \mathbb{R}$
\item[(iii)]$\Phi > 0$, $\Psi-\tfrac{3}{4} I_m \ge 0$, $b < \tfrac{1}{4}(3\alpha - \tfrac{1}{2}(m-1))$
\end{itemize}
Therefore, we deduce that the Assumptions \ref{2assumptions_matrixcase} also hold for broad classes of the model $Q_{n2}$.

\subsection{The distribution of the test statistic under contiguous alternatives}

Let $P_0=W_m(\alpha,I_m)$, $\alpha > \max\{ \tfrac12(2m-1), \tfrac12(m+3) \}$; and denote by $\boldsymbol{P_n}= P_0 \otimes \cdots \otimes P_0$ and $\boldsymbol{Q_n}=Q_{n1} \otimes \cdots \otimes Q_{n1}$ the $n$-fold product probability measures of $P_0$ and $Q_{n1}$, respectively. 
 
\begin{theorem}
\label{theoremcontiguous_matrixcase}
Let $m \ge 2$ and $X_{n1},\dotsc, X_{nn}$, $n \in \mathbb{N}$, be a triangular array of $m \times m$ positive-definite row-wise i.i.d. random matrices, where $X_{nj}=X_j$, $j=1,\dotsc,n$. We assume that the distribution of $X_{nj}$ is $Q_{n1}$, for every $j=1,\dotsc,n$. Further, let $\mathcal{Z}_n=(\mathcal{Z}_n(T), T > 0)$ be a random field with 
$$
\mathcal{Z}_{n}(T)=\frac{1}{\sqrt{n}} \sum_{j=1}^n  \bigg[ \Gamma_m(\alpha) A_{\nu}(T, X_{nj}^{1/2} \bar{X}_n^{-1} X_{nj}^{1/2})-\etr(-T/\alpha) \bigg],
$$
$T > 0$. Under the Assumptions \ref{2assumptions_matrixcase}, 
there exists a centered Gaussian field $\mathcal{Z}:=(\mathcal{Z}(T), T > 0)$ with sample paths in $L^2$ and the covariance function $K(S, T)$ in (\ref{covariancefn_matrixcase}), and a function
\begin{multline*}
c(T)=\int_{X>0} { \bigg[ \Gamma_m(\alpha)A_{\nu}(T, \alpha^{-1} X)+\frac{\tr(X-\alpha I_m)}{\alpha^{2}} (\tr T) \etr(-\alpha^{-1} T) }\\
{-\etr(-\alpha^{-1} T) \bigg]  h(X)}\, \dd P_0(X),
\end{multline*}
$ T > 0$, such that $\mathcal{Z}_n \xrightarrow{d} \mathcal{Z} + c$ in $L^2$. Moreover, as $n \rightarrow \infty$,
$$ 
\boldsymbol{T}^2_{n} \xrightarrow{d} \int_{T > 0} {\big( \mathcal{Z}(T)+c(T) \big)^2}\, \dd P_0(T).
$$ 
\end{theorem}
 
\bigskip

We note that the proof of this theorem and the subsequent results can be obtained by following the approach in \cite[pp.~79--91]{ref27} and Theorem 4.3 in \cite{hadjicostarichards}.  In order to maintain a relatively self-contained presentation, we provide some of the details here.

Before proceeding to those details, we will present some preliminary results.  Consider the log-likelihood ratio,
\begin{align*}
\Lambda_n(X_{n1},\dotsc,X_{nn}) := \log \frac{ \dd \boldsymbol{Q}_n(X_{n1},\dotsc,X_{nn})}{\dd\boldsymbol{P}_n(X_{n1},\dotsc,X_{nn})}.
\end{align*}
From the definition of $\boldsymbol{P}_n$ and $\boldsymbol{Q}_n$, we obtain 
\begin{align*}
\Lambda_n(X_{n1},\dotsc,X_{nn}) &= \log \prod_{j=1}^n (1+n^{-1/2}h_n(X_{nj})) \\
&=\sum_{j=1}^n \log(1+n^{-1/2}h_n(X_{nj})).
\end{align*}
Since $\Lambda_n=-\infty$ if and only if $1+n^{-1/2}h_n(X_{nj}) = 0$ for some $j$, we obtain 
\begin{align*}
\boldsymbol{P}_n (\Lambda_n = -\infty) &= \boldsymbol{P}_n\bigg(\bigcup_{j=1}^n \{1+n^{-1/2}h_n(X_{n1})=0\}\bigg) \\
&\le n \boldsymbol{P}_n(1+n^{-1/2}h_n(X_{n1})=0) \\
&= n \boldsymbol{P}_n(h_n(X_{n1})=-n^{1/2}).
\end{align*}
Since $h_n(X_{n1})=-n^{1/2}$ if and only if $|n^{-1/2}h_n(X_{n1})|^4=1$ then 
\begin{align*}
n \boldsymbol{P}_n(h_n(X_{n1}) = -n^{1/2}) &= n E_{P_0}  ( |n^{-1/2}h_n(X_{n1})|^4 I(h_n(X_{n1})=-n^{-1/2}))\\
&\le n^{-1} E_{P_0}  ( |h_n(X_{n1})|^4)\\
&\le n^{-1} \sup_{n \in \mathbb{N}} E_{P_0} ( |h_n|^4).
\end{align*}
Under the assumption that $\sup_{n \in \bN} E_{P_0} (|h_n|^4) < \infty$, we obtain 
$$
n \boldsymbol{P}_n(h_n(X_{n1}) = -n^{1/2}) \to 0,
$$
as $n \to \infty$.  Therefore, without loss of generality, we shall assume that $\Lambda_n > -\infty$ and $1+n^{-1/2}h_n(X_{nj}) > 0$ for all $j=1,\dotsc,n$ and $n \ge 1$ (see \cite[p.~140, Appendix D.2]{ref27} or \cite[p.~303, Example 6.118]{wittingmuller}).

The Taylor expansion of order $2$ of the function $\log(1+n^{-1/2}h_n(X_{nj}))$, at $1$ is 
$$
\log(1+n^{-1/2}h_n(X_{nj}))=n^{-1/2}h_n(X_{nj})-2^{-1}n^{-1}h_n^2(X_{nj}) + R(h_n(X_{nj})),
$$ 
with remainder term 
$$
R(h_n(X_{nj})) = \frac13 n^{-3/2}\big(1 + n^{-1/2} t_{nj}(X_{nj}) h_n(X_{nj})\big)^{-3} h_n^3(X_{nj}),
$$
where $t_{nj} : \mathcal{P}_{+}^{m \times m} \to [0,1]$ is a measurable function.  Therefore, 
$$
\Lambda_n(X_{n1},\dotsc,X_{nn})=\sum_{j=1}^{n} \big[ n^{-1/2}h_n(X_{nj})-2^{-1}n^{-1}h_n^2(X_{nj}) + R(h_n(X_{nj})) \big].
$$
In the following result, we use the notation $\sigma^2:=\int {|h|^2}\, \dd P_0$. 

\smallskip

\begin{lemma}
\label{lemmacontiguous_matrixcase} As $n \to \infty$, 
\begin{itemize}
\item[(i)]
$n^{-1/2} \sum_{j=1}^{n} h_n(X_{nj}) \xrightarrow{d} \mathcal{N}(0,\sigma^2)$ in $\boldsymbol{P}_n$-distribution.
\item[(ii)]
$n^{-1} \sum_{j=1}^{n} h_n^2(X_{nj}) \rightarrow \sigma^2$ in $\boldsymbol{P}_n$-probability.
\item[(iii)]
$\sum_{j=1}^{n} R(h_n(X_{nj})) \rightarrow 0$ in $\boldsymbol{P}_n$-probability.
\end{itemize}
\end{lemma}

\medskip

\noindent
The proofs of these results are given in \cite[pp.~80-83]{ref27} and in \cite{hadjicosta19}. Combining these three results, we conclude that under $\boldsymbol{P}_n$, 
\begin{equation}
\label{convln_contiguous_matrixcase}
\Lambda_n(X_{n1}, \dotsc, X_{nn}) \xrightarrow{d} \mathcal{N}(-\tfrac12 \sigma^2,\sigma^2).
\end{equation}

We introduced in Section \ref{sectionlimitingnull_matrixcase} the random field
$$
\mathcal{Z}_{n}(T)=\frac{1}{\sqrt{n}} \sum_{j=1}^n \big[ \Gamma_m(\alpha) A_{\nu}(T, X_{j}^{1/2} \bar{X}_n^{-1} X_{j}^{1/2})-\etr(-T/\alpha) \big], 
$$
$T > 0$. Also, we introduced in Theorem \ref{limitingnulldistribution_matrixcase}, the centered random field
\begin{multline*}
\mathcal{Z}_{n,3}(T)=\frac{1}{\sqrt{n}}\sum_{j=1}^{n} \big[ \Gamma_m(\alpha) A_{\nu}(T, \alpha^{-1}X_j)\\
+ \Gamma_m(\alpha) \langle \alpha^{-1}(\alpha I_m-X_j) , g(T)  \rangle - \etr(-\alpha^{-1} T) \big], 
\end{multline*}
$T > 0$, where 
$$
g(T)=- \frac{\alpha^{-1}}{m \, \Gamma_m(\alpha)} (\tr T) \etr(-\alpha^{-1} T) \, I_m.
$$
We proved that there exists a centered Gaussian field $\mathcal{Z}:=(\mathcal{Z}(T), T > 0)$ with sample paths in $L^2$ and with covariance function $K(S, T)$ given in (\ref{covariancefn_matrixcase}) such that, under $\boldsymbol{P_n}$, $|| \mathcal{Z}_n -\mathcal{Z}_{n,3}||_{L^2} \xrightarrow{p} 0$ and $\mathcal{Z}_{n,3} \xrightarrow{d} \mathcal{Z}$ in $L^2$. For $k \in \mathbb{N}$ and $T_1,\dotsc, T_k \in \mathcal{P}_{+}^{m \times m}$, it follows from the multivariate Central Limit Theorem that $\big(\mathcal{Z}_{n,3}(T_1),\dotsc,\mathcal{Z}_{n,3}(T_k)\big)' \xrightarrow{d} \mathcal{N}_k \big(0,\Sigma\big)$ under $\boldsymbol{P_n}$, where $\Sigma= \big( K(T_i, T_j) \big)_{1 \le i,j \le k}$ is the $k \times k$ positive definite matrix with $(i, j)$th entry $K(T_i, T_j)$. 

Let $\|\cdot\|_{\bR^{k+1}}$ denote the standard Euclidean norm on $\bR^{k+1}$.  Then, by Lemma \ref{lemmacontiguous_matrixcase}(iii),
\begin{align}
\bigg\|\big(&\mathcal{Z}_{n,3}(T_1),\dotsc,\mathcal{Z}_{n,3}(T_k),\Lambda_n\big)' \nonumber\\
& \quad - \bigg(\mathcal{Z}_{n,3}(T_1),\dotsc,\mathcal{Z}_{n,3}(T_k), \sum_{j=1}^n \big[ n^{-1/2} h_n(X_{nj})-2^{-1} n^{-1}h_n^2(X_{nj})\big] \bigg)' \bigg\|^2_{\mathbb{R}^{k+1}}\nonumber\\
\label{convprobzn3ln_matrixcase}
&=\bigg( \Lambda_n - \sum_{j=1}^n \big[ n^{-1/2} h_n(X_{nj})-2^{-1} n^{-1}h_n^2(X_{nj})\big] \bigg)^2 \nonumber\\
& =\bigg(\sum_{j=1}^{n} R(h_n(X_{nj})\bigg)^2 \rightarrow 0,
\end{align}
in $\boldsymbol{P}_n$-probability.  
 
\medskip

\begin{lemma}
\label{lemma4contiguous_matrixcase}
For $T > 0$, define 
\begin{align}
\label{cofti_matrixcase}
c(T)=\lim_{n \rightarrow \infty} \Cov \bigg(\mathcal{Z}_{n,3}(T), \sum_{j=1}^n \big[ n^{-1/2}  h_n(X_{nj})-\frac{1}{2}n^{-1}h_n^2(X_{nj})\big] \bigg), 
\end{align}
and set $\boldsymbol{c}=\big( c(T_1),\dotsc, c(T_k) \big)'$. Then, under $\boldsymbol{P_n}$,
\begin{multline}
\label{lemma4contiguouseq_matrixcase}
\bigg(\mathcal{Z}_{n,3}(T_1),\dotsc,\mathcal{Z}_{n,3}(T_k), \sum_{j=1}^n \big[ n^{-1/2} h_n(X_{nj})-2^{-1} n^{-1}h_n^2(X_{nj})\big] \bigg)' \\
\xrightarrow{d} \mathcal{N}_{k+1} \left( (0,\dotsc,0,-2^{-1}\sigma^2)', \begin{bmatrix} \Sigma & \boldsymbol{c} \\ \boldsymbol{c'} & \sigma^2 \end{bmatrix} \right)
\end{multline}
and 
\begin{align}
\label{convz3ln_contiguous_matrixcase}
\big(&\mathcal{Z}_{n,3}(T_1),\dotsc,\mathcal{Z}_{n,3}(T_k),\Lambda_n\big)' \xrightarrow{d} \mathcal{N}_{k+1} \left( (0,\dotsc,0,-2^{-1}\sigma^2)', \begin{bmatrix} \Sigma & \boldsymbol{c} \\ \boldsymbol{c'} & \sigma^2 \end{bmatrix} \right).
\end{align}
\end{lemma}

\begin{proof}
Substituting for $\mathcal{Z}_{n,3}$ in (\ref{cofti_matrixcase}), applying Assumptions \ref{2assumptions_matrixcase}, and carrying out some straightforward calculations, we obtain 
\begin{multline*}
c(T)=\int_{X>0} { \bigg[ \Gamma_m(\alpha)A_{\nu}(T, \alpha^{-1} X)+\frac{\tr(X-\alpha I_m)}{\alpha^{2}} (\tr T) \etr(-\alpha^{-1} T) }\\
{-\etr(-\alpha^{-1} T) \bigg]  h(X)}\, \dd P_0(X),
\end{multline*}
for $T > 0$. Letting 
\begin{multline*}
w(T_i,X_{nj}) :=\Gamma_m(\alpha)A_{\nu}(T_i, \alpha^{-1} X_{nj})\\
-\alpha^{-2} (\tr T_i) \etr(-\alpha^{-1} T_i) \tr(\alpha I_m-X_{nj})-\etr(-\alpha^{-1} T_i); 
\end{multline*}
then 
$$
\mathcal{Z}_{n,3}(T_i)=\frac{1}{\sqrt{n}} \sum_{j=1}^n w(T_i, X_{nj}).
$$
To establish (\ref{lemma4contiguouseq_matrixcase}), we will apply the Cram\'er-Wold device. Then it suffices to establish that for every $\boldsymbol{u}=(u_1,\dotsc, u_{k+1})' \in \mathbb{R}^{k+1}$,
\begin{multline}
\frac{1}{\sqrt{n}} \sum_{j=1}^n \bigg(w(T_1,X_{nj}),\dotsc, w(T_k,X_{nj}), h_n(X_{nj})-2^{-1}n^{-1/2} h_n^2(X_{nj}) \bigg)' \boldsymbol{u} \\
\label{cramerwoldconv_matrixcase}
\xrightarrow{d} \mathcal{N}_{k+1} \left( (0,\dotsc,0,-2^{-1}\sigma^2) \boldsymbol{u}, \boldsymbol{u}'\begin{bmatrix} \Sigma & \boldsymbol{c} \\ \boldsymbol{c'} & \sigma^2 \end{bmatrix} \boldsymbol{u} \right).
\end{multline}
Now, let $Y_1,Y_2,\dotsc$ be i.i.d. $P_0$-distributed $m \times m$ random matrices, and define
$$
k_n(Y_j)=\sum_{l=1}^k w(T_l,Y_j) u_l + \big( h_n(Y_j)-2^{-1}n^{-1/2}h_n^2(Y_j) \big) u_{k+1}.
$$
Under $\boldsymbol{P_n}$, $n^{-1/2} \sum_{j=1}^n k_n(Y_j)$ has the same distribution as
$$
\frac{1}{\sqrt{n}} \sum_{j=1}^n \bigg(w(T_1,X_{nj}),\dotsc, w(T_k,X_{nj}), h_n(X_{nj})-2^{-1} n^{-1/2}h_n^2(X_{nj}) \bigg) \boldsymbol{u},
$$
$\boldsymbol{u} \in \mathbb{R}^{k+1}$. Since $E(w(T_i, Y_1))=0$, $i=1,\dotsc, k$ then 
$$
\mu_n := E(k_n(Y_1))=-2^{-1} n^{-1/2} u_{k+1} E(h_n^2(Y_1)).
$$
Denote by $\tau^2_n$ the variance of $k_n(Y_1)$.  Then, 
\begin{align*}
\tau^2_n &=\sum_{i=1}^k \sum_{j=1}^k \Cov \big(w(T_i,Y_1),w(T_j,Y_1)\big) u_i u_j + \Var (h_n(Y_1)) u^2_{k+1} \\
& \qquad + (4n)^{-1} \Var (h^2_n(Y_1)) u^2_{k+1} - n^{-1/2} \Cov (h_n(Y_1), h^2_n(Y_1)) u^2_{k+1} \\
& \qquad\qquad + 2 u_{k+1} \sum_{i=1}^k \Cov \bigg(w(T_i, Y_1),\big( h_n(Y_1)-2^{-1} n^{-1/2}h_n^2(Y_1) \big)\bigg) u_i.
\end{align*}
By Assumptions \ref{2assumptions_matrixcase}, we obtain $\Var (h^2_n(Y_1)) < \infty$ and $\Cov \big( h_n(Y_1), h^2_n(Y_1) \big) < \infty.$  Thus, as $n \rightarrow \infty$, 
\begin{equation}
\label{tao_contiguous_matrixcase}
\tau^2_n \to \sum_{i=1}^k \sum_{j=1}^k K(T_i, T_j) u_i u_j + \sigma^2 u^2_{k+1} + 2 u_{k+1}\sum_{i=1}^k u_i c(T_i):= \tau^2.
\end{equation}
Similarly, it can be shown that, as $n \rightarrow \infty$, 
\begin{align}
\label{asconvcontiguous_matrixcase}
\big(k_n(Y_1)- \mu_n\big)^2 & \to \sum_{i=1}^k \sum_{j=1}^k w(T_i, Y_1) w(T_j,Y_1) u_i u_j \nonumber \\
& \qquad + h(Y_1) \bigg(h(Y_1) u_{k+1} + \sum_{i=1}^k w(T_i,Y_1) u_i \bigg) u_{k+1},
\end{align}
$P_0$-almost surely. In addition, we notice that 
\begin{multline}
\label{expected_asconvcontiguous_matrixcase}
E \bigg[\sum_{i=1}^k \sum_{j=1}^k w(T_i, Y_1) w(T_j,Y_1) u_i u_j + h(Y_1) \bigg(h(Y_1) u_{k+1} + \sum_{i=1}^k w(T_i, Y_1) u_i \bigg) u_{k+1} \bigg] \\
\ = \ \sum_{i=1}^k \sum_{j=1}^k K(T_i, T_j) u_i u_j + \sigma^2 u^2_{k+1} + 2 u_{k+1} \sum_{i=1}^k c(T_i) u_i \ \equiv \ \tau^2.
\end{multline}
For every $\epsilon > 0$, 
\begin{eqnarray}
\label{result1contiguous_matrixcase}
0 \le (k_n(Y_1)- \mu_n)^2  \ I(|(k_n(Y_1)- \mu_n| > \epsilon \sqrt{n} \tau_n) \le (k_n(Y_1)- \mu_n)^2.
\end{eqnarray}
Also, for every $\epsilon > 0$, 
$$
0 \le (k_n(Y_1)- \mu_n)^2 \ I(|(k_n(Y_1)- \mu_n| > \epsilon \sqrt{n} \tau_n) \le  \frac{(k_n(Y_1)- \mu_n)^4}{ \epsilon^2 n \tau^2_n},
$$
from which we conclude that as $n \rightarrow \infty$, 
\begin{eqnarray}
\label{result2contiguous_matrixcase}
(k_n(Y_1)- \mu_n)^2 \  I(|(k_n(Y_1)- \mu_n| > \epsilon \sqrt{n} \tau_n) \rightarrow 0,
\end{eqnarray}
$P_0$-almost surely.  As the results (\ref{tao_contiguous_matrixcase}) -- (\ref{result2contiguous_matrixcase}) are the sufficient conditions in Pratt's version of the Dominated Convergence Theorem \cite[p.~221, Theorem 5.5]{gut}, we conclude that as $n \to \infty$, 
$$
E\big( (k_n(Y_1)- \mu_n)^2  \ I(|(k_n(Y_1)- \mu_n| > \epsilon \sqrt{n} \tau_n) \big) \to 0.
$$
This result is equivalent to the Lindeberg condition, i.e., for every $\epsilon > 0$, 
$$
\lim_{n \to \infty} \frac{1}{n \tau^2_n} \sum_{j=1}^n E\big( (k_n(Y_j)- \mu_n)^2 \ I(|(k_n(Y_j)- \mu_n| > \epsilon \sqrt{n} \tau_n) \big) = 0.
$$
Thus, we deduce from the Lindeberg-Feller Central Limit Theorem that 
$$
\frac{1}{\sqrt{n} \tau_n} \sum_{j=1}^n (k_n(Y_j)-\mu_n) \xrightarrow{d} \mathcal{N}(0,1),
$$
therefore,
$$
\frac{1}{\sqrt{n}} \sum_{j=1}^n k_n(Y_j) \xrightarrow{d} \mathcal{N} \big( -2^{-1} \sigma^2 u_{k+1}, \tau^2 \big).
$$
Note also that $(0,\dotsc,0,-2^{-1}\sigma^2)\boldsymbol{u} = -2^{-1} \sigma^2 u_{k+1}$ and that 
\begin{align*}
\tau^2 &= \sum_{i=1}^k \sum_{j=1}^k u_i u_j K(T_i, T_j) + u^2_{k+1}\sigma^2 +2 u_{k+1}\sum_{i=1}^k u_i c(T_i)\\
&= \boldsymbol{u}' \begin{bmatrix} \Sigma & \boldsymbol{c} \\ \boldsymbol{c'} & \sigma^2 \end{bmatrix} \boldsymbol{u}.
\end{align*}
Therefore, (\ref{cramerwoldconv_matrixcase}) is proved. Finally, (\ref{convz3ln_contiguous_matrixcase}) follows from (\ref{convprobzn3ln_matrixcase}), (\ref{cramerwoldconv_matrixcase}), and \cite[p.~25, Theorem 4.1]{ref21}.
\end{proof}

Now, we proceed to the proof of Theorem \ref{theoremcontiguous_matrixcase}.

\medskip

\begin{proof}[Proof of Theorem \ref{theoremcontiguous_matrixcase}]
By (\ref{convln_contiguous_matrixcase}) and Le Cam's first lemma (see \cite[p.~140, Theorem D.5]{ref27} or \cite[p.~311, Corollary 6.124]{wittingmuller}), $\boldsymbol{P}_n$ and $\boldsymbol{Q}_n$ are mutually contiguous.  Also, by (\ref{convz3ln_contiguous_matrixcase}) and Le Cam's third lemma (see \cite[p.~141, Theorem D.6]{ref27} or  \cite[p.~329, Corollary 6.139]{wittingmuller}), under $\boldsymbol{Q}_n$, 
\begin{eqnarray}
\label{convz3_contiguous_matrixcase}
\big(\mathcal{Z}_{n,3}(T_1),\dotsc,\mathcal{Z}_{n,3}(T_k) \big)' \xrightarrow{d} \mathcal{N}_{k} (\boldsymbol{c},\Sigma).
\end{eqnarray}
By \cite[p.~138, Theorem D.2]{ref27} or \cite[p.~56, Theorem 5.51]{wittingmuller}, the convergence in distribution of $\mathcal{Z}_{n,3}$ under $\boldsymbol{P}_n$ in $L^2$ implies that $\mathcal{Z}_{n,3}$ is tight in $L^2$ under $\boldsymbol{P}_n$. Further, since $\boldsymbol{Q}_n$ is contiguous to $\boldsymbol{P}_n$, by \cite[p.~139, Theorem D.4]{ref27} or \cite[p.~295, Theorem 6.113 (a)]{wittingmuller}, $\mathcal{Z}_{n,3}$ is tight in $L^2$ under $\boldsymbol{Q}_n$. 

By (\ref{convz3_contiguous_matrixcase}) and the tightness of $\mathcal{Z}_{n,3}$ in $L^2$ under $\boldsymbol{Q}_n$, we obtain $\mathcal{Z}_{n,3} \xrightarrow{d} \mathcal{Z} + c$ under $\boldsymbol{Q}_n$ (see \cite[Theorem 2, Example 4]{cremerskadelka}).  Moreover, since $\| \mathcal{Z}_n - \mathcal{Z}_{n,3}\|_{L^2} \xrightarrow{p} 0$ under $\boldsymbol{P}_n$ and $\boldsymbol{Q}_n$ is contiguous to $\boldsymbol{P}_n$, we have under $\boldsymbol{Q}_n$, $\| \mathcal{Z}_n - \mathcal{Z}_{n,3}\|_{L^2} \xrightarrow{p} 0$.  Thus, by Billingsley \cite[p.~25, Theorem 4.1]{ref21}, we obtain $\mathcal{Z}_n \xrightarrow{d} \mathcal{Z} + c$ under $\boldsymbol{Q}_n$.  

Finally, by the Continuous Mapping Theorem \cite[p.~31, Corollary 1]{ref21}, we have $\lVert \mathcal{Z}_{n} \rVert^2_{L^2} \xrightarrow{d} \lVert \mathcal{Z} +c \rVert^2_{L^2}$ under $\boldsymbol{Q}_n$, i.e., 
$$
\boldsymbol{T}^2_{n} \xrightarrow{d} \int_{T > 0} {(\mathcal{Z}(T)+c(T))^2}\, \dd P_0(T)
$$
under $\boldsymbol{Q}_n$.  The proof now is complete. 
\end{proof}

\section{The Efficiency of the Test}
\label{efficiency_matrix}

In this Section, we investigate the approximate Bahadur slope of the test statistic $\boldsymbol{T}_n^2$ under local alternatives.  Further, we show the validity of a modified Wieand condition. The proof of Wieand's condition, under which the Bahadur and Pitman efficiencies agree, remains an open problem.  By applying the results of this section, we are able to calculate the approximate asymptotic relative efficiency (ARE) of the proposed test relative to potential alternative tests. 

For $m \ge 2$, let $X_1, X_2, \dotsc $ be i.i.d., $m \times m$ positive-definite random matrices with unknown distribution $P$. We assume that $P$ is indexed by a parameter $\theta \in \Theta :=(-\eta, \eta)$, for some $\eta >0$. We let $\theta \in \Theta_0=\{\theta_0 \}=\{ 0 \}$ to represent the null hypothesis and $\theta \in \Theta_1=\Theta \setminus \{0\}$ to represent the alternative hypothesis. In Section \ref{goodnessoffittests_wishart}, we showed that $\boldsymbol{T}^2_n$ is scale-invariant, i.e., it does not depend on the unknown scale matrix $\Sigma$. Thus, under the null hypothesis $\theta_0=0$, we assume that $X_1, X_2,\dotsc $ are i.i.d., $m \times m$ positive-definite $P_0$-distributed random matrices and under the local alternatives, represented by $\theta \in \Theta_1$, we assume that $X_1, X_2,\dotsc $ are i.i.d., $m \times m$ positive-definite $P_\theta$-distributed random matrices. 

The Radon-Nikodym derivative of $P_{\theta}$ with respect to $P_0$ is $d P_{\theta}/{d P_0}=1+\theta h_{\theta}$. We assume that as $\theta \rightarrow 0$, the function $h_{\theta}$ converges to some function $h$ in mean square, i.e., 
\begin{align}
\label{meanconvefficiency_matrix}
\lim_{\theta \rightarrow 0} \int_{X > 0} {|h_{\theta}(X)-h (X)|^2}\, \dd P_0(X)=0.
\end{align}
 Since $ \int {(\dd P_{\theta}/ \dd P_0)}\, \dd P_0=1$, we have 
\begin{equation}
\label{meanzero_matrix} 
\int_{X >0} {h_{\theta} (X)}\, \dd P_0(X)=0,
\end{equation}
for $\theta \in \Theta_1$. Further, we shall assume that for $\theta \in \Theta_1$,
\begin{equation}
\label{assumption_efficiency_matrix}
\int_{X > 0} {X h_{\theta}(X) }\, \dd P_0(X)=0.
\end{equation}

\subsection{The approximate Bahadur slope of the test}
For a description of the approximate Bahadur slope of a test under local alternatives and for the definition of a standard sequence, we refer to Bahadur \cite{bahadur1, bahadur2}, Taherizadeh \cite[Chapter 5]{ref27} or to Section 5 in \cite{hadjicostarichards}. 

We have the following result for the test statistic $\boldsymbol{T}_n^2$.

\begin{theorem}
The sequence of test statistics $\{\boldsymbol{T}_n : n \in \bN \}$ is a standard sequence.  Further, $a = \tilde{\delta}_1^{-1}$, the inverse of the largest eigenvalue of the covariance operator $\mathcal{S}$, 
\begin{equation}
\label{bsquaredfunction_matrix}
b^2(\theta) = \theta^2 \int_{ T >0} {\bigg[\int_{ X >0} { \Gamma_m(\alpha) A_{\nu}(T, \alpha^{-1}X) h_{\theta} (X)}\, \dd P_0(X) \bigg]^2}\, \dd P_0(T)
\end{equation}
and 
$$
\lim_{\theta \rightarrow 0} \frac{c^{(a)}(\theta)}{\theta^2} = \tilde{\delta}_1^{-1} \int_{ T >0} {\bigg[\int_{ X >0} { \Gamma_m(\alpha) A_{\nu}(T, \alpha^{-1}X) h (X)}\, \dd P_0(X) \bigg]^2}\, \dd P_0(T).
$$
\end{theorem}

\Pro
The proof of this theorem follows along the lines of the proof of Theorem 5.1  in \cite{hadjicostarichards}. For completeness, we provide the details here. 

First, we will establish that $\{\boldsymbol{T}_n : n \in \bN \}$ is a standard sequence. In Section \ref{goodnessoffittests_wishart}, we showed that the limiting null distribution of the test statistic $\boldsymbol{T}^2_n$ is the same as that of  $\sum_{k \ge 1} \tilde{\delta}_k \chi^2_{N(\tilde{\delta}_k)}$, where $\tilde{\delta}_k$, $k \ge 1$ is an enumeration, listed in non-increasing order, of the distinct eigenvalues of $\mathcal{S}$ with corresponding multiplicities $N(\tilde{\delta}_k)$, and $\{\chi^2_{N(\tilde{\delta}_k)} \}$ are i.i.d. $\chi^2_{N(\tilde{\delta}_k)}$-distributed random variables. From the Monotone Convergence Theorem, we have 
$$
\lim_{M \rightarrow \infty} E \bigg( \sum_{k=1}^{M} \tilde{\delta}_k \chi^2_{N(\tilde{\delta}_k)} \bigg)=E \bigg(\sum_{k \ge 1} \tilde{\delta}_k \chi^2_{N(\tilde{\delta}_k)} \bigg) =\sum_{k \ge 1} \tilde{\delta}_k \ N(\tilde{\delta}_k),
$$
which is finite since $\mathcal{S}$ is of trace-class.  Thus, $\sum_{k \ge 1} \tilde{\delta}_k \chi^2_{N(\tilde{\delta}_k)}$ is almost surely a positive random variable with continuous probability distribution function. 

By Zolotarev \cite{zolotarev}, 
\begin{align*}
1-F(t) &=P \bigg( \sum_{k \ge 1} \tilde{\delta}_k \chi^2_{N(\tilde{\delta}_k)} > t^2 \bigg)\\
&= \frac{1}{\Gamma(N(\tilde{\delta_1})/2)} \ \bigg[ \prod_{k \ge 2} \bigg( 1-\frac{\tilde{\delta}_k}{\tilde{\delta}_1} \bigg)^{-N(\tilde{\delta}_k)} \ \bigg] \ \bigg( \frac{t^2}{2\tilde{\delta}_1} \bigg)^{(N(\tilde{\delta}_1)/2)-1} \\
& \qquad\qquad\qquad\qquad\qquad\qquad\qquad\qquad \times \exp(-t^2/2 \tilde{\delta}_1) \ [1+ o_{p}(1)],
\end{align*}
where $o_{p}(1) \xrightarrow{t \rightarrow \infty} 0$. Therefore, 
\begin{align*}
- & 2 t^{-2} \log[1-F(t)] \\
&= 2 t^{-2} \bigg[ \log[\Gamma(N(\tilde{\delta_1})/2)] + \sum_{k \ge 2} N(\tilde{\delta}_k) \log \bigg( 1-\frac{\tilde{\delta}_k}{\tilde{\delta}_1} \bigg) - \bigg( \frac{N(\tilde{\delta}_1)}{2}-1 \bigg) \log \bigg( \frac{t^2}{2\tilde{\delta}_1} \bigg) \bigg]
+ \tilde{\delta}_1^{-1},
\end{align*}
which converges to $\tilde{\delta}_1^{-1}$ as $t \rightarrow \infty$.  

By assumption (\ref{assumption_efficiency_matrix}), for $\theta \in \Theta_1$, 
\begin{align*}
E_\theta (X_1):=\int_{X >0} { X }\, \dd P_{\theta}(X) = \int_{X > 0} { X (1+\theta h_{\theta}(X))}\, \dd P_0(X)=\alpha I_m.
\end{align*}
From the proof of Theorem \ref{consistency_matrixcase}, we have
$$
n^{-1}\boldsymbol{T}^2_n \xrightarrow{p} \int_{T > 0} {\bigg( E_{\theta} [\Gamma_m(\alpha) A_{\nu} (T, \alpha^{-1} X_1)] -\etr(-\alpha^{-1} T) \bigg)^2}\, \dd P_0(T).
$$
Since $\dd P_\theta/\dd P_0 = 1 + \theta h_\theta$ then, by (\ref{meanzero_matrix}), 
$$
E_{\theta} [\Gamma_m(\alpha) A_{\nu} (T, \alpha^{-1} X_1)]
=\etr(-\alpha^{-1} T) + \theta \int_{X >0} { \Gamma_m(\alpha) A_{\nu} (T, \alpha^{-1} X) h_{\theta}(X)}\, \dd P_{0}(X),
$$
and then it follows that $n^{-1} \boldsymbol{T}^2_n \xrightarrow{p} b^2(\theta)$, the function defined in (\ref{bsquaredfunction_matrix}).  Therefore, $n^{-1/2} \boldsymbol{T}_n \xrightarrow{p} b(\theta)$ in $P_\theta$-probability, so the sequence of test statistics $\{\boldsymbol{T}_n: n \in \bN\}$ is a standard sequence.  

Finally, we find the limiting approximate Bahadur slope, as $\theta \rightarrow 0$. By applying the Cauchy-Schwarz inequality, (\ref{2besselineq_matrixargument}), and assumption (\ref{meanconvefficiency_matrix}), it is straightforward to establish that
\begin{align*}
\limsup_{\theta \rightarrow 0} \bigg| \frac{b^2(\theta)}{\theta^2}-\int_{T >0} {\bigg[\int_{X > 0} { \Gamma_m(\alpha) A_{\nu} (T, \alpha^{-1} X) h(X)}\, \dd P_0(X) \bigg]^2}\, \dd P_0(T) \bigg|=0.
\end{align*}
Therefore, 
\begin{align*}
\lim_{\theta \rightarrow 0} \bigg| \frac{b^2(\theta)}{\theta^2}-\int_{T >0} {\bigg[\int_{X > 0} { \Gamma_m(\alpha) A_{\nu} (T, \alpha^{-1} X) h(X)}\, \dd P_0(X) \bigg]^2}\, \dd P_0(T) \bigg|=0.
\end{align*}
The proof is now complete. $\qed$

\subsection{A modified form of Wieand's condition}

Wieand \cite{wieand} showed that if two standard sequences of test statistics satisfy an additional condition, now called the \textit{Wieand condition}, then the limiting approximate Bahadur efficiency is in accord with the limiting Pitman efficiency, as the level of significance decreases to $0$. For a description about Pitman's asymptotic relative efficiency, we refer to Taherizadeh \cite[Chapter 5]{ref27} or to Section 5 in \cite{hadjicostarichards}.  Although the proof of Wieand's condition remains an open problem in the matrix setting, we show that a modified form of Wieand's condition is valid for the test statistics $ \{ \boldsymbol{T}_n : n \in \bN \}$. 

\begin{theorem}
\label{wieandTn_matrix}
There exists a constant $\theta^{*} > 0$ such that for any $\epsilon > 0$ and $\gamma \in (0,1)$, there exists a constant $C > 0$ such that 
$$
P(|n^{-1/2}\boldsymbol{T}_n - b(\theta)| \ \le \epsilon b(\theta)) > 1-\gamma
$$
for any $\theta \in \Theta_1 \cap (-\theta^{*}, \theta^{*})$ and $n^{1/2} > C/b^2(\theta)$.  
\end{theorem}

\Pro
For $T > 0$ and $\theta \in \Theta$, consider the orthogonally invariant Hankel transform, 
$$
\mathcal{H}_{X_1, \theta} (T) = E_{\theta} [\Gamma_m(\alpha) A_{\nu} (T, \alpha^{-1} X_1)].
$$
We have
$$
 n^{-1/2} \boldsymbol{T}_n =\bigg[ \int_{T > 0} \bigg( \frac{1}{n}\sum_{j=1}^{n} \Gamma_m(\alpha) A_{\nu} (T, Z_j)-\etr(-\alpha^{-1} T) \bigg)^2 \, \dd P_0(T) \bigg]^{1/2}.
$$
By adding and subtracting the term $\mathcal{H}_{X_1, \theta} (T)$ inside the squared term, and then applying Minkowski's inequality, we obtain 
\begin{multline}
\label{wieand_ineq1_matrix}
 n^{-1/2} \boldsymbol{T}_n \le \ \bigg[ \int_{T >0} {  \bigg( \frac{1}{n}\sum_{j=1}^{n} \Gamma_m(\alpha) A_{\nu} (T, Z_j)-\mathcal{H}_{X_1, \theta}(T)\bigg)^2}\, \dd P_0(T)\bigg]^{1/2} \\
+ \bigg[ \int_{T >0} { \big( \mathcal{H}_{X_1, \theta}(T) -\etr(-\alpha^{-1} T) \big)^2}\, \dd P_0(T)  \bigg]^{1/2}.
\end{multline}
Now set 
\begin{align*}
b(\theta) := \bigg[ \int_{T >0} { \big( \mathcal{H}_{X_1, \theta}(T) -\etr(-\alpha^{-1} T) \big)^2}\, \dd P_0(T)  \bigg]^{1/2}.
\end{align*}
By adding and subtracting the term $$ \frac{1}{n}\sum_{j=1}^{n} \Gamma_m(\alpha) A_{\nu} (T, Z_j)$$ inside the squared term, and then again applying Minkowski's inequality, we get
\begin{equation}
\label{wieand_ineq2_matrix}
b(\theta) \le n^{-1/2} \boldsymbol{T}_n 
+ \bigg[ \int_{T >0}  {\bigg( \frac{1}{n}\sum_{j=1}^{n} \Gamma_m(\alpha) A_{\nu} (T, Z_j)-\mathcal{H}_{X_1, \theta}(T) \bigg)^2}\, \dd P_0(T)\bigg]^{1/2}.
\end{equation}
Combining (\ref{wieand_ineq1_matrix}) and (\ref{wieand_ineq2_matrix}), we conclude that
\begin{equation}
\label{wieand_ineq3_matrix}
| n^{-1/2}\boldsymbol{T}_n - b(\theta)| 
\le \bigg[ \int_{T >0}  {\bigg( \frac{1}{n}\sum_{j=1}^{n} \Gamma_m(\alpha) A_{\nu} (T, Z_j)-\mathcal{H}_{X_1, \theta}(T) \bigg)^2}\, \dd P_0(T)\bigg]^{1/2}.
\end{equation}
Further, by subtracting and adding the term $$\frac{1}{n}\sum_{j=1}^{n} \Gamma_m(\alpha) A_{\nu} (T, \alpha^{-1} X_j)$$ inside the squared term
$$
\bigg( \frac{1}{n}\sum_{j=1}^{n} \Gamma_m(\alpha) A_{\nu} (T, Z_j)-\mathcal{H}_{X_1, \theta}(T) \bigg)^2,
$$
and then applying the Cauchy-Schwarz inequality, we obtain 
\begin{multline}
\bigg( \frac{1}{n}\sum_{j=1}^{n} \Gamma_m(\alpha) A_{\nu} (T, Z_j)-\mathcal{H}_{X_1, \theta}(T) \bigg)^2 \\
\label{wieand_ineq4_matrix}
 \le 2 \ \bigg[ \frac{1}{n} \sum_{j=1}^n \Gamma_m(\alpha) \bigg( A_{\nu} (T, Z_j) -A_{\nu} (T, \alpha^{-1} X_j) \bigg) \bigg]^2 \\
+ 2 \bigg[ \frac{1}{n} \sum_{j=1}^n \Gamma_m(\alpha) A_{\nu} (T, \alpha^{-1} X_j)  - \mathcal{H}_{X_1, \theta}(T)\bigg]^2.
\end{multline}
Next, by (\ref{lipschitz_bessel_2matrixargument}), 
\begin{equation}
\label{inequality_w_matrix}
\frac{1}{n} \sum_{j=1}^n  \Gamma_m(\alpha) | A_{\nu} (T, Z_j) - A_{\nu} (T, \alpha^{-1} X_j)| \ \le 
2 m^{3/4} \, \lVert T \rVert^{1/2}_F \ \frac{1}{n}  \sum_{j=1}^n \lVert Z_j-\alpha^{-1} X_j \rVert^{1/2}_F.
\end{equation}
Since 
\begin{align*}
Z_j-\alpha^{-1} X_j&= X_j^{1/2} \bar{X}^{-1}_n X_j^{1/2} -\alpha^{-1} X_j^{1/2} X_j^{1/2}\\
&= \alpha^{-1} X_j^{1/2} \bar{X}_n^{-1/2} (\alpha I_m-\bar{X}_n) \bar{X}_n^{-1/2} X_j^{1/2},
\end{align*}
and since the trace is invariant under cyclic permutations and the Frobenius norm is sub-multiplicative then 
\begin{align*}
\frac{1}{n}  \sum_{j=1}^n \lVert Z_j-\alpha^{-1} X_j \rVert^{1/2}_F & = \alpha^{-1} \frac{1}{n} \sum_{j=1}^n \lVert X_j^{1/2} \bar{X}_n^{-1/2} (\alpha I_m-\bar{X}_n) \bar{X}_n^{-1/2} X_j^{1/2} \rVert^{1/2}_F \\
&=\alpha^{-1} \frac{1}{n} \sum_{j=1}^n \lVert \bar{X}_n^{-1/2}  X_j \bar{X}_n^{-1/2} (\alpha I_m-\bar{X}_n) \rVert^{1/2}_F \\
& \le \alpha^{-1} \frac{1}{n} \sum_{j=1}^n \lVert \bar{X}_n^{-1/2}  X_j \bar{X}_n^{-1/2} \rVert^{1/2}_F \ \lVert \alpha I_m-\bar{X}_n \rVert^{1/2}_F.
\end{align*}
By the Cauchy-Schwarz inequality, 
\begin{align*}
\frac{1}{n} \sum_{j=1}^n \lVert \bar{X}_n^{-1/2}  X_j \bar{X}_n^{-1/2} \rVert^{1/2}_F & \le \frac{1}{\sqrt{n}} \ \bigg( \sum_{j=1}^n \lVert \bar{X}_n^{-1/2}  X_j \bar{X}_n^{-1/2} \rVert_F \bigg)^{1/2}\\
&=\frac{1}{\sqrt{n}} \ \bigg( \sum_{j=1}^n \big[ \tr ( \bar{X}_n^{-1/2}  X_j \bar{X}_n^{-1/2})^2 \big]^{1/2} \bigg)^{1/2}.
\end{align*}
Since $\bar{X}_n^{-1/2}  X_j \bar{X}_n^{-1/2}$ is a positive definite matrix then 
$$
\tr ( \bar{X}_n^{-1/2}  X_j \bar{X}_n^{-1/2})^2 \le ( \tr \bar{X}_n^{-1/2}  X_j \bar{X}_n^{-1/2})^2,
$$ 
and therefore,
\begin{align*}
\frac{1}{n} \sum_{j=1}^n \lVert \bar{X}_n^{-1/2}  X_j \bar{X}_n^{-1/2} \rVert^{1/2}_F  & \le n^{-1/2} \ \bigg( \sum_{j=1}^n \tr \bar{X}_n^{-1/2}  X_j \bar{X}_n^{-1/2} \bigg)^{1/2}\\
&=n^{-1/2} \ \bigg( \tr \bar{X}_n^{-1/2}  n \bar{X}_n \bar{X}_n^{-1/2} \bigg)^{1/2}\\
&= n^{-1/2} \ (n \tr I_m)^{1/2} = m^{1/2}.
\end{align*}
Therefore,
$$
\frac{1}{n}  \sum_{j=1}^n \lVert Z_j-\alpha^{-1} X_j \rVert^{1/2}_F \le \alpha^{-1} m^{1/2} \ \lVert \alpha I_m-\bar{X}_n \rVert^{1/2}_F,
$$
and by (\ref{inequality_w_matrix}), we obtain 
\begin{equation}
\label{wieand_ineq5_matrix}
\frac{1}{n} \sum_{j=1}^n  \Gamma_m(\alpha) | A_{\nu} (T, Z_j) - A_{\nu} (T, \alpha^{-1} X_j)| \ \le 2 \alpha^{-1} m^{5/4} \ \lVert T \rVert^{1/2}_F \ \lVert \alpha I_m-\bar{X}_n \rVert^{1/2}_F.
\end{equation}

By (\ref{wieand_ineq3_matrix}), Markov's inequality, and Fubini's theorem, 
\begin{multline}
\label{inequality_w2_matrix}
P\big(| n^{-1/2}\boldsymbol{T}_n - b(\theta)| \le \epsilon b(\theta) \big) \\
\ge 1 - \frac{1}{\epsilon^2 b^2(\theta)} \int_{T >0}  { E_{\theta} \bigg( \frac{1}{n}\sum_{j=1}^{n} \Gamma_m(\alpha) A_{\nu} (T, Z_j)-\mathcal{H}_{X_1, \theta}(T) \bigg)^2}\, \dd P_0(T).
\end{multline}
By (\ref{wieand_ineq4_matrix}) and (\ref{wieand_ineq5_matrix}), we see that (\ref{inequality_w2_matrix}) is greater than or equal to 
\begin{multline*}
1 - \frac{1}{\epsilon^2 b^2(\theta)} \bigg[ 8 \alpha^{-2} m^{5/2} \bigg( \int_{T > 0} {\lVert T \rVert_F }\, \dd P_0(T) \bigg) \  E_{\theta} \lVert \alpha I_m-\bar{X}_n \rVert_F \\
+ 2 \int^{\infty}_0 { E_{\theta} \bigg( \frac{1}{n} \sum_{j=1}^n \Gamma_m(\alpha) A_{\nu} (T, \alpha^{-1} X_j)-\mathcal{H}_{X_1, \theta}(T) \bigg)^2 }\, \dd P_0(T) \bigg].
\end{multline*}
In Theorem \ref{consistency_matrixcase}, we showed that $\tilde{C}:=\int_{T > 0} {\lVert T \rVert_F }\, \dd P_0(T) < \infty$. Further, by (\ref{2besselineq_matrixargument}), 
\begin{align*}
E_{\theta} \bigg( \frac{1}{n} \sum_{j=1}^n \Gamma_m(\alpha) A_{\nu} (T, \alpha^{-1} X_j)-\mathcal{H}_{X_1, \theta}(T) \bigg)^2 
&= n^{-1} \Var_{\theta}  \big( \Gamma_m(\alpha) A_{\nu} (T, \alpha^{-1} X_1) \big) \\ 
&\le n^{-1};
\end{align*}
therefore 
\begin{equation}
\label{wieand_ineq6_matrix}
P\big(| n^{-1/2}\boldsymbol{T}_n - b(\theta)| \le \epsilon b(\theta) \big) 
\ge 1- \frac{1}{\epsilon^2 b^2(\theta)} \bigg[ 8 \alpha^{-2} m^{5/2} \ \tilde{C} \ E_{\theta} \lVert \alpha I_m-\bar{X}_n \rVert_F + \frac{2}{n} \bigg].
\end{equation}

Next, we write 
\begin{align*}
\lVert \alpha I_m-\bar{X}_n \rVert_F = \big( \tr(\alpha I_m-\bar{X}_n)^2 \big)^{1/2}&=\bigg( \tr \bigg[ \frac{1}{n^2} \bigg(\sum_{j=1}^n (X_j -\alpha I_m) \bigg)^2 \bigg] \bigg)^{1/2}\\
&=\frac{1}{n} \bigg( \tr \bigg(\sum_{j=1}^n (X_j -\alpha I_m) \bigg)^2 \bigg)^{1/2},
\end{align*}
and expand the sum.  By the Cauchy-Schwarz inequality, and using the i.i.d. property of $X_1,\ldots, X_n$, we obtain 
\begin{align*}
E_{\theta} \lVert \alpha I_m-\bar{X}_n \rVert_F 
&\le \frac{1} {n} \bigg[ E_{\theta} \bigg( \tr \bigg(\sum_{j=1}^n (X_j -\alpha I_m) \bigg)^2 \bigg) \bigg]^{1/2} \nonumber\\
&=\frac{1}{n} \bigg[ E_{0}  \bigg( \tr \bigg(\sum_{j=1}^n (X_j -\alpha I_m) \bigg)^2 \cdot \prod_{j=1}^n (1+\theta h_{\theta}(X_j)) \bigg) \bigg]^{1/2}.
\end{align*}
Squaring the above sum and using the fact that $X_1,\ldots, X_n$ are i.i.d., we obtain 
\begin{align}
\label{expected_are_matrix}
E_{\theta} \lVert & \alpha I_m-\bar{X}_n \rVert_F \nonumber\\
&\le n^{-1/2} \bigg[ E_{0} \bigg( \tr \big[ (X_1-\alpha I_m)^2 \big] \cdot \prod_{j=1}^n (1+\theta h_{\theta}(X_j)) \bigg) \bigg]^{1/2} \nonumber\\
& \quad + \bigg(\frac{n-1}{n} \bigg)^{1/2} \bigg[ E_{0} \bigg( \tr \big[ (X_1-\alpha I_m) (X_2-\alpha I_m) \big] \cdot \prod_{j=1}^n (1+\theta h_{\theta}(X_j)) \bigg) \bigg]^{1/2}.
\end{align}
Since $E_0 h_{\theta}(X) = 0$ and, by (\ref{assumption_efficiency_matrix}), $E_0 X h_{\theta}(X) = 0$ 
for $\theta \in \Theta_1$, then $E_0 (1 + \theta h_{\theta} (X_1))=1$ and 
$$
E_{0} \bigg( \tr \big[ (X_1-\alpha I_m) \big] \cdot (1+ \theta h_{\theta}(X_1)) \bigg)= \tr E_0 \bigg( (X_1-\alpha I_m) (1+ \theta h_{\theta}(X_1)) \bigg)=0.
$$
Thus, the first term in the right-hand side of (\ref{expected_are_matrix}) equals 
$$
n^{-1/2} \bigg[ E_{0} \bigg( \tr \big[ (X_1-\alpha I_m)^2 \big] \cdot (1+\theta h_{\theta}(X_1)) \bigg) \bigg]^{1/2}
$$
and the second term equals $0$. 

Further, by applying the Cauchy-Schwarz inequality, we also find that 
\begin{align*}
E_{0} \bigg( \tr & \big[ (X_1-\alpha I_m)^2 \big] \cdot (1+\theta h_{\theta}(X_1)) \bigg)\\
&=E_{0} \big( \tr \big[ (X_1-\alpha I_m)^2 \big] \big)  +\theta E_{0} \big( \tr \big[ (X_1-\alpha I_m)^2 \big] \cdot h_\theta(X_1) \big)\\
& \le  \bigg[ E_{0} \big( \tr \big[ (X_1-\alpha I_m)^2 \big] \big)^2 \bigg]^{1/2} \ \bigg[ 1+ | \theta | \big( E_0 (h^2_{\theta}(X_1)) \big)^{1/2} \bigg].
\end{align*}
To show that $E_{0} \big( \tr \big[ (X_1-\alpha I_m)^2 \big] \big)^2$ is finite, we write 
\begin{align*}
\tr \big[ (X_1-\alpha I_m)^2 \big] &= \tr(X_1^2 -2\alpha X_1 +\alpha^2 I_m)\\
&=\tr X_1^2 -2\alpha \tr X_1 +\alpha^2 m,
\end{align*}
and since $(a+b+c)^2 \le 3(a^2+b^2+c^2)$, for $a, b, c \in \mathbb{R}$, it is sufficient to show that $E_0 (\tr X_1^2)^2 < \infty$ and $E_0 (\tr X_1)^2 < \infty$. However,
\begin{align*}
E_0 (\tr X_1^2)^2 \le E_0 (\tr X_1)^4 
= \sum_{|\kappa|=4} E_0 (C_\kappa (X_1)) < \infty,
\end{align*}
by (\ref{zonalintegral}).  By another application of (\ref{zonalintegral}),
\begin{align*}
E_0 (\tr X_1)^2 &= \sum_{|\kappa|=2} E_0 (C_\kappa (X_1)) < \infty.
\end{align*}
By assumption (\ref{meanconvefficiency_matrix}), we conclude that there exists $\theta^{*} \in (0, \eta)$ such that
$$
{\bar\sigma}^2 := \sup_{\theta \in (-\theta^{*}, \theta^{*})}E_{0} \bigg( \tr  \big[ (X_1-\alpha I_m)^2 \big] \cdot (1+\theta h_{\theta}(X_1)) \bigg) < \infty.
$$
Therefore, (\ref{wieand_ineq6_matrix}) can be written as
\begin{align*}
P\big(| n^{-1/2}\boldsymbol{T}_n - b(\theta)| \le \epsilon b(\theta) \big) 
&\ge 1- \frac{1}{n^{1/2} \epsilon^2 b^2(\theta)} \bigg[ 8 \alpha^{-2} m^{5/2} \ \tilde{C} \bar\sigma +\frac{2}{n^{1/2}} \bigg]\\
&\ge 1- \frac{8 \alpha^{-2} m^{5/2} \, \tilde{C} \bar\sigma + 2}{n^{1/2} \epsilon^2 b^2(\theta)},
\end{align*}
for all $\theta \in (-\theta^{*}, \theta^{*})$.  Setting 
$C = \big(8 \alpha^{-2} m^{5/2} \, \tilde{C} \bar\sigma +2 \big)/\epsilon^2 \gamma$ then, for all 
$\theta \in (-\theta^{*}, \theta^{*})$ and $n^{1/2} > C/b^2(\theta)$, 
$$
P\big(| n^{-1/2}\boldsymbol{T}_n - b(\theta)| \le \epsilon b(\theta) \big) \ge 1-\frac{\gamma C}{n^{1/2} b^2(\theta)} > 1-\gamma. 
$$
The proof now is complete. 
$\qed$

\bibliographystyle{ims}

\end{document}